\documentclass[12pt, lettersize]{amsart}

\usepackage{setspace}

\usepackage{amssymb}   

\usepackage{amsmath}

\usepackage{mathrsfs}

\usepackage{stmaryrd}

\usepackage{amsthm}
\usepackage{newlfont}
\usepackage{amscd}
\usepackage{mathtools}
\usepackage{bm}
\usepackage{tikz-cd}
\usepackage{graphicx}

\usepackage{hyperref}
\usepackage{framed}
\usepackage{comment}

\usepackage{enumitem}

\usepackage[all]{xy}

\usepackage{textcomp}

\usepackage{amsbsy}

\usepackage{lmodern}
\usepackage[T1]{fontenc}

\newtheorem{thm}{Theorem}[section]

\newtheorem{thm-defn}[thm]{Theorem/Definition}
\newtheorem{lem}[thm]{Lemma}
\newtheorem{prop}[thm]{Proposition}
\newtheorem{cor}[thm]{Corollary}

\theoremstyle{definition}
\newtheorem{defn}[thm]{Definition}

\newtheorem{eg}[thm]{Example}

\newtheorem{construction}[thm]{Construction}
\newtheorem{set-up}[thm]{Set-up}
\newtheorem{notation}[thm]{Notation}

\theoremstyle{remark}
\newtheorem{rem}[thm]{Remark}

\numberwithin{equation}{section}

\usepackage{relsize}
\usepackage[bbgreekl]{mathbbol}
\usepackage{amsfonts}

\DeclareSymbolFontAlphabet{\mathbb}{AMSb}
\DeclareSymbolFontAlphabet{\mathbbl}{bbold}

\DeclareMathOperator{\rank}{rank}

\DeclareMathOperator{\Spec}{Spec}

\DeclareMathOperator{\Spf}{Spf}

\DeclareMathOperator{\Ker}{Ker}

\DeclareMathOperator{\Hom}{Hom}

\DeclareMathOperator{\Ext}{Ext}

\DeclareMathOperator{\Sh}{Sh}

\DeclareMathOperator{\Ab}{Ab}
\DeclareMathOperator{\CA}{CA}

\newcommand{\gp}{\mathrm{gp}}

\newcommand{\id}{\mathrm{id}}
\newcommand{\CR}{\mathrm{CR}}

\newcommand{\perf}{\mathrm{perf}}

\newcommand{\aff}{\mathrm{aff}}

\newcommand{\cris}{\mathrm{cris}}
\newcommand{\CRIS}{\mathrm{CRIS}}

\newcommand{\strat}{{\text{-}\mathrm{strat}}}

\newcommand{\st}{\mathrm{st}}

\newcommand{\et}{\mathrm{\acute{e}t}}

\newcommand{\N}{\mathbf{N}}
\newcommand{\F}{\mathbf{F}}
\newcommand{\Z}{\mathbf{Z}}
\newcommand{\Q}{\mathbf{Q}}

\newcommand{\A}{\mathbf{A}}

\newcommand{\calA}{\mathcal{A}}

\newcommand{\calC}{\mathcal{C}}
\newcommand{\calD}{\mathcal{D}}
\newcommand{\calE}{\mathcal{E}}
\newcommand{\calF}{\mathcal{F}}
\newcommand{\calG}{\mathcal{G}}
\newcommand{\calH}{\mathcal{H}}
\newcommand{\calI}{\mathcal{I}}
\newcommand{\calJ}{\mathcal{J}}
\newcommand{\calK}{\mathcal{K}}
\newcommand{\calL}{\mathcal{L}}
\newcommand{\calM}{\mathcal{M}}

\newcommand{\calO}{\mathcal{O}}

\newcommand{\calS}{\mathcal{S}}
\newcommand{\calT}{\mathcal{T}}

\newcommand{\calX}{\mathcal{X}}
\newcommand{\calY}{\mathcal{Y}}

\mathchardef\mhyphen="2D

\begin{document}

\pagenumbering{arabic}

\title{On log crystalline higher direct image}

\author{Heng Du} 
\address{Yau Mathematical Sciences Center, Tsinghua University, Beijing 100084, China}
\email{hengdu@mail.tsinghua.edu.cn}
\author{Yong Suk Moon} 
\address{Beijing Institute of Mathematical Sciences and Applications, Beijing 101408, China}
\email{ysmoon@bimsa.cn}
\author{Koji Shimizu}
\address{Yau Mathematical Sciences Center, Tsinghua University, Beijing 100084, China;
Beijing Institute of Mathematical Sciences and Applications, Beijing 101408, China}
\email{shimizu@tsinghua.edu.cn}

\begin{abstract} 
We define the big crystalline site for a log scheme and prove the basic properties. In particular, we show the boundedness, base change, and perfectness theorems for the crystalline higher direct image of quasi-coherent crystals between fine log schemes. We also introduce the big absolute crystalline sites and discuss the Frobenius isogeny property of the crystalline higher direct image of $F$-isocrystals.
\end{abstract}

\maketitle

\section{Introduction}

Our main goal in this article is to prove the boundedness and perfectness theorems for the crystalline higher direct image of quasi-coherent crystals between fine log schemes, which seems to be missing in the literature in a general form.

Crystalline cohomology is a $p$-adic Weil cohomology theory for smooth projective varieties over a perfect field of characteristic $p>0$. Grothendieck proposed the definition of the crystalline topos in \cite{Grothendieck-crystals-deRhamCohomology}, and Berthelot laid the detailed foundation of crystalline theory in his monumental work \cite{Berthelot-book}. While one often needs the theory of rigid cohomology and arithmetic $\calD$-modules for a well-behaved coefficient theory or six-functor formalism, the recently-developed $A_\mathrm{inf}$-cohomology by Bhatt--Morrow--Scholze \cite{bhatt-morrow-scholze-integralpadic} and prismatic cohomology by Bhatt--Scholze \cite{bhatt-scholze-prismaticcohom} are directly related to crystalline cohomology. 

Kato \cite{Kato-log} and Hyodo--Kato \cite{hyodo-kato} initiated the generalization of crystalline theory to log schemes in the sense of Fontaine--Illusie. They were motivated by the $C_\st$-conjecture of Fontaine and Jannsen about the $p$-adic \'etale cohomology of smooth proper varieties over a $p$-adic field with semistable reduction, and some of the main results in \cite{Berthelot-book} were generalized to the log case in these articles.
Since then, further developments of log crystalline theory have been given by several authors: 
to name a few, \cite{ogus-griffiths, Tsuji-PoincareDuality, Shiho-I,faltings-almostetale, nakkajima-shiho, Beilinson-crystalline-period-map}.

However, a full generalization of Berthelot's results to the logarithmic case does not seem to be available in the literature. For example, the boundedness theorem of quasi-coherent modules \cite[Thm.~V.3.2.4]{Berthelot-book} is proved in the log context under further assumptions (see Remark~\ref{rem:remark on main theorem} below).
This article proves a general boundedness theorem in log crystalline theory. 

Another goal is to discuss the definition and basic results of the \emph{big} crystalline site for a log scheme. This is motivated by log prismatic theory developed by Koshikawa \cite{koshikawa}. In fact, a part of the current manuscript was originally written as Appendix B.1 in \cite{du-liu-moon-shimizu-purity-F-crystal}, where we studied analytic prismatic $F$-crystals on the absolute logarithmic prismatic site of a semistable $p$-adic log formal scheme and associated to them $F$-isocrystals on the big absolute logarithmic crystalline site of the mod $p$ reduction. As a continuation of the work, we were naturally led to studying log crystalline and prismatic higher direct images, and we realized that there are not enough references for the basic general theorems in the former case. Proving such theorems has resulted in this article.

Now we explain the set-up and main results of the current paper.
Let $(\Sigma,M_\Sigma)$ be a fine log scheme together with a quasi-coherent PD-ideal sheaf $(\calJ_\Sigma, \gamma_\Sigma)$, which we simply denote by $\Sigma^\sharp$. We assume that $p$ is nilpotent in $\calO_\Sigma$. Let $(X,M_X)$ be a fine log scheme over $(\Sigma,M_\Sigma)$ such that $\gamma_\Sigma$ extends to $\calJ_\Sigma\calO_X$.
The \emph{big relative crystalline site} $((X,M_X)/\Sigma^\sharp)_\CRIS$ consists of tuples $(U,T,M_T, f, i,\gamma)$ where $f\colon U\rightarrow X$ is a morphism of schemes over $\Sigma$, $(T,M_T)$ is a log $\Sigma^\sharp$-scheme with $M_T$ integral and quasi-coherent, $i$ is an exact closed immersion $(U,M_U\coloneqq f^\ast M_X)\rightarrow (T,M_T)$ over $(\Sigma,M_\Sigma)$, and $\gamma$ is a PD-structure on the ideal of $\calO_T$ defining $U$ compatible with the one on $\calJ_\Sigma\calO_T$ (Definition~\ref{def:big relative logarithmic crystalline site}). Assigning $\Gamma(T,\calO_T)$ to such a tuple defines the sheaf $\calO_{X/\Sigma}$ of rings. 
Big relative crystalline sites enjoy the functorial property, and our main results concern the higher direct images of quasi-coherent $\calO_{X/\Sigma}$-modules:

\begin{thm}\label{thm:main theorem in introduction}
Let $f\colon (X,M_X)\rightarrow (Y,M_Y)$ be a morphism of fine log $\Sigma^\sharp$-schemes.
\begin{enumerate}
 \item (Boundedness: Theorem~\ref{thm:boundedness of crystalline pushforward}) If $Y$ is quasi-compact and $f$ is quasi-separated of finite type, then there exists an integer $r$ such that $R^qf_{\CRIS,\ast}\calE=0$ for every $q\geq r$ and every quasi-coherent $\calO_{X/\Sigma}$-module $\calE$. 
 \item (Base Change: Theorem~\ref{thm:crystalline base change for crystalline pushforward})
Suppose that we have the following diagram
\[
\xymatrix{
(X',M_{X'})\ar[r]^-{g'} \ar[d]_-{f'}& (X,M_X)\ar[d]_-f\\
(Y',M_{Y'})\ar[r]^-g \ar[d]& (Y,M_Y)\ar[d]\\
\Sigma'^\sharp=(\Sigma',M_{\Sigma'},\calJ',\gamma') \ar[r]& \Sigma^\sharp=(\Sigma,M_\Sigma,\calJ,\gamma).
}
\]
If $f\colon (X,M_X)\rightarrow (Y,M_Y)$ is smooth and integral, $f\colon X\rightarrow Y$ is qcqs, and if the upper-square is Cartesian, then for every flat quasi-coherent $\calO_{X/\Sigma}$-module $\calE$ on $((X,M_X)/\Sigma^\sharp)_\CRIS$, the canonical morphism
\[
g^\ast_\CRIS Rf_{\CRIS,\ast}\calE\rightarrow Rf'_{\CRIS,\ast}g'^\ast_\CRIS\calE
\]
is a quasi-isomorphism.
 \item 
 (Perfectness: Theorem~\ref{thm:crystalline derived pushforward of perfect complexes})
If $f\colon (X,M_X)\rightarrow (Y,M_Y)$ is smooth and integral and if $f\colon X\rightarrow Y$ is proper, then for every finite locally free $\calO_{X/\Sigma}$-module $\calE$, the derived pushforward $Rf_{\CRIS,\ast}\calE$ is a perfect complex of $\calO_{Y/\Sigma}$-modules.
\end{enumerate}
The same properties hold for small crystalline site.
\end{thm}

Note that quasi-coherent $\calO_{X/\Sigma}$-modules automatically satisfy the crystal property (Proposition~\ref{prop:quasi-coherent crystals}) and are also referred to as (quasi-coherent) crystals in the literature.

\begin{rem}\label{rem:remark on main theorem}
\hfill
\begin{enumerate}
 \item Theorem~\ref{thm:main theorem in introduction}(1) in the small site case is proved in \cite[Thm.~V.3.2.4]{Berthelot-book} in the non-log case and in \cite[Prop.~2.3.11, Rem.~2.3.13]{nakkajima-shiho} in the case of Zariski log structures or (log) smooth morphisms. 
 \item The small site version of Theorem~\ref{thm:main theorem in introduction}(2) is proved in \cite[Thm.~V.3.5.1]{Berthelot-book} (non-log case) and stated in \cite[Thm.~6.10]{Kato-log} (log case) without proof.
 \item 
 For Theorem~\ref{thm:main theorem in introduction}(3), the perfectness of the evaluation at a \emph{Noetherian} PD-thickening is available in \cite[Thm.~7.16]{berthelot-ogus-book}. However, a statement as in (3) does not seem to be available in the classical literature: see \cite[Thm.~14]{Faltings-very-ramified} and \cite[07MX]{stacks-project} for (3) in the non-log case and \cite[1.11, Thm.~(ii)]{Beilinson-crystalline-period-map} for (3) for $\calE=\calO_{X/\Sigma}$ and $f$ being of Cartier type. In \cite{faltings-almostetale}, Faltings mentions the statement and gives a sketch of the proof.\footnote{Strictly speaking, $Y$ is the formal spectrum of a $p$-adic DVR in \textit{loc.~cit.}. This approach was explained and suggested to us by Teruhisa Koshikawa.} We provide a detailed proof.
\end{enumerate}
\end{rem}

Let us explain the outline of the proof of these main theorems.
The first step is to study the higher direct image $Ru_{Y/\Sigma}\calF$ for a quasi-coherent $\calO_{Y/\Sigma}$-module $\calF$ along the projection to the \'etale site $u_{Y/\Sigma}\colon \Sh(((Y,M_Y)/\Sigma^\sharp)_\CRIS)\rightarrow \Sh(Y_\et)$. When $Y$ can be embedded into a smooth log scheme $(Z,M_Z)$ over $\Sigma$, $Ru_{Y/\Sigma}\calF$ is computed by a de Rham complex with coefficient for $Z/\Sigma$ (Theorem~\ref{thm:cohomology of crystal in terms of log de Rham complex}). In fact, this is \cite[Thm.~6.4]{Kato-log}. While the key arguments are explained in \textit{loc.~cit.}, and some further relevant computations are given in \cite{ogus-griffiths}, the log crystalline Poincar\'e lemma is not explicitly proved in these references. For the sake of completeness, we provide it in Theorem~\ref{thm:log crystalline Poincare lemma}  by following \cite[\S 2.2]{nakkajima-shiho}, which treats the case of smooth schemes with relative SNCD.
Another subtlety lies in the fact that we work on the \emph{big} crystalline site: the crystalline pushforward along an exact closed immersion does not preserve quasi-coherent sheaves. This is already observed in the (non-log) big crystalline site case (e.g.~\cite[\S~2.1]{deJong-dieudonnemodule}) and comes from the fact that the formation of taking the PD-envelope commutes with the base change along flat morphisms but it fails in general. By providing an extra argument to take care of this issue in Example~\ref{eg:description of i-pushforward of crystal}, 
we obtain Theorem~\ref{thm:cohomology of crystal in terms of log de Rham complex}. 

Our proof of Theorem~\ref{thm:main theorem in introduction}(1) diverges from that of \cite{Berthelot-book} when reducing the general case to the above smooth embeddable case. In \textit{loc.~cit.}, one can take a finite Zariski covering of smooth embeddable affine open subschemes of $Y$ and use the alternating \v{C}ech complex. This is based on the fact that $((U,M_U)/\Sigma^\sharp)_\CRIS$ gives rise to an \emph{open subtopos} of $((Y,M_Y)/\Sigma^\sharp)_\CRIS$ when $U$ is an open subscheme of $Y$. However, it is no longer the case when $U$ is \'etale over $Y$, while one can find an embedding into a smooth log scheme only \emph{\'etale locally} on $Y$ in general. Our proof will use an alternating \v{C}ech complex argument in the \'etale topology, which is well developed in \cite{stacks-project} for boundedness of quasi-coherent sheaves on algebraic spaces (see also\cite[\S~5, Lem.~1]{faltings-almostetale}). Note that Olsson proves a boundedness result for the crystalline site of a tame Noetherian Deligne--Mumford stack in \cite[Thm.~2.5.15, Rem.~2.5.18]{Olsson-crystallinecohomology} with a different method.
The proof of (2) follows Berthelot's argument since the smoothness assumption allows us to use a cohomological descent for the Zariski covering. 

For Theorem~\ref{thm:main theorem in introduction}(3), we need to compute the higher direct image $Ru_{X_U/T}\calE$ where $(U,T)\in ((Y,M_Y)/\Sigma^\sharp)_\CRIS$ and $f_U\colon X_U=X\times_YU\rightarrow U$. The key observation is that, while there may be no morphism $f_T\colon X_T\rightarrow T$ lifting 
$f_U$ globally, the relative Frobenius twist $f_T'$ is globally defined. Namely, observe that the absolute Frobenius $F_T$ factors as $F_T\colon T\xrightarrow{F_{(U,T)}}U\hookrightarrow T$ since $U\hookrightarrow T$ is a PD-immersion. We then define $f_T'\colon X_T'\coloneqq X_U\times_{U,F_{(U,T)}}T\rightarrow T$. The proof ultimately relies on the fact that $Rf'_{T,\ast}$ preserves pseudo-coherent complexes as it is flat and proper. It is worth pointing out that no Cartier type assumption is necessary since we do not use the Cartier isomorphism. 

To relate crystalline theory to prismatic theory, one also wants to introduce the \emph{big absolute crystalline site} $(X,M_X)_\CRIS$ for a fine log scheme $(X,M_X)$ in which $p$ is nilpotent; this is simply defined as the colimit of $((X,M_X)/(\Z_p/p^n)^\sharp)_\CRIS$ (Definition~\ref{def:big absolute log crystallien site}). The crystalline higher direct image is naturally computed via the relative crystalline site, and thus we can easily obtain the big absolute crystalline version of Theorem~\ref{thm:main theorem in introduction}: see Theorems~\ref{thm:absolute boundedness of crystalline pushforward}, \ref{thm:absolute crystalline base change for crystalline pushforward}, and \ref{thm:absolute crystalline derived pushforward of perfect complexes} with $\calS^\sharp=\Z_p^\sharp$.
We then introduce the notion of $F$-isocrystals in perfect complexes on the absolute crystalline site (Definition~\ref{def:F-isocrystals}) and prove the following Frobenius isogeny property.

\begin{thm}[{Theorem~\ref{thm:Frobenius isogeny property}, Remark~\ref{rem:Frobenius isogeny property}}]
Let $f\colon (X,M_X)\rightarrow (Y,M_Y)$ be smooth of Cartier type between fine log schemes over $\F_p$ such that $f\colon X\rightarrow Y$ is proper and $Y$ is quasi-compact.
Assume moreover that $(Y,M_Y)$ is either smooth and integral over a perfect field $k$ of characteristic $p$, or the mod $p$ reduction of a semistable $p$-adic formal scheme. 
If $(\calE_\Q,\varphi_{\calE_\Q})$ is a finite locally free $F$-crystal, then $(Rf_{\CRIS,\ast}\calE_\Q,Rf_{\CRIS,\ast}\varphi_{\calE_\Q})$ is an $F$-isocrystal in perfect complexes on $(Y,M_Y)_\CRIS$.
\end{thm}

We prove this Frobenius property by using the theorems we have established and reducing to the case $\calE=\calO_{X/\Z_p}$, which was treated by Hyodo and Kato \cite{hyodo-kato}.

This paper is organized as follows. After briefly recalling the log PD-envelope in \S~\ref{sec:log PD-envelope}, we define the big relative log crystalline site in \S~\ref{sec: Definition of big log crystalline sites}. The functoriality of crystalline sites and other relevant morphisms of topoi are discussed in \S~\ref{sec: Functoriality of crystalline topoi} and \S~\ref{sec: Projections to small etale topoi}, respectively. We discuss log crystalline Poincar\'e lemma in the smooth embeddable case in \S~\ref{sec:Embedding into a smooth log scheme}. After establishing these basic results, we prove Theorem~\ref{thm:main theorem in introduction}(1), (2), and (3) in \S~\ref{sec:Boundedness of log crystalline cohomology}, \ref{sec:Base change}, and \ref{sec:Crystalline derived pushforward of finite locally free crystals}, respectively. 

We define the big log crystalline sites in the $p$-adic base and absolute cases and discuss the basic theorems in \S~\ref{sec:Big absolute crystalline sites}. The Frobenius isogeny property is proved in \S~\ref{sec:Frobenius isogeny property}. 

\medskip
\noindent
\textbf{Conventions}.
In this article, log schemes are defined using the sheaf of monoids in \'etale topology, following \cite{Kato-log}. 
In particular, see \cite[(3.3) (resp.~Def.~4.3)]{Kato-log} for the definition of a morphism $(X,M_X)\rightarrow (Y,M_Y)$ of fine log schemes being smooth (resp.~integral); the former is also referred to as log smooth in some literature, and the latter should not be confused with integral morphisms of schemes.

\medskip
\noindent
\textbf{Acknowledgments.}

Part of this paper about the big absolute crystalline sites was initially written as an appendix to our work \cite{du-liu-moon-shimizu-purity-F-crystal} with Tong Liu. While discussing the follow-up projects, we realized that it is better to extend the content of the appendix to the relative crystalline site and also prove the basic theorems of the crystalline higher direct image. We thank Tong Liu for his permission for us to write a new article about these results and for his encouragement. We also appreciate the valuable discussions and suggestions provided by Shizhang Li. Theorem~\ref{thm:crystalline derived pushforward of perfect complexes} in a preliminary version had an issue with the formulation and proof. We thank  
Teruhisa Koshikawa for suggesting the approach in \cite{faltings-almostetale} and commenting on the literature. The first author was partially supported by the National Key R\&D Program of China (No.~2023YFA1009703) and the Beijing Natural Science Foundation (Youth Program, No. 1254044). The third author was partially supported by the NSFC Excellent Young Scientists Fund Program (Overseas).

\section{Log PD-envelopes}\label{sec:log PD-envelope}
Let us recall the following terminology and result in \cite[\S~1]{Beilinson-crystalline-period-map}. A \emph{log PD-scheme} over $\Z_p$ consists of a quadruple $\Sigma^\sharp\coloneqq (\Sigma,M_\Sigma,\calJ_\Sigma,\gamma_\Sigma)$ where $(\Sigma,M_\Sigma)$ is a log scheme, $\calJ_\Sigma\subset \calO_{\Sigma}$ is a quasi-coherent ideal sheaf (on the small \'etale site $\Sigma_\et$), and $\gamma_\Sigma$ is a PD-structure on $\calJ_{\Sigma}$. A \emph{log $\Sigma^\sharp$-scheme} is a log scheme $(Y,M_Y)$ over $(\Sigma,M_\Sigma)$ such that $\gamma_\Sigma$ extends to $\calJ_\Sigma\calO_Y$, and a morphism of $\Sigma^\sharp$-schemes is defined in an obvious way. 

\begin{thm}\label{thm:existence of log PD envelope}
Let $i\colon (Y,M_Y)\hookrightarrow (Z,M_Z)$ be an immersion of log schemes over $(\Sigma,M_\Sigma)$ such that $M_Z$ is quasi-coherent and $M_Y$ is integral and quasi-coherent. Assume one of the following conditions.
\begin{enumerate}
 \item[(a)] $M_Y$ is fine and $M_Z$ is coherent.
 \item[(b)] $i$ is a morphism of log $\Sigma^\sharp$-schemes (i.e., $\gamma_\Sigma$ extends to $Y$ and $Z$).
\end{enumerate}

Then there exists a PD-envelope $i'\colon (Y',M_{Y'})\hookrightarrow (D,M_D)$ of $i$ relative to $\Sigma^\sharp$. Namely, $i'$ is an exact closed immersion of log $\Sigma^\sharp$-schemes with integral and quasi-coherent log structure sitting in the diagram
\[
\xymatrix{
(Y',M_{Y'}) \ar@{^{(}->}[r]^-{i'}\ar[d] & (D,M_D)\ar[d] \\
(Y,M_{Y}) \ar@{^{(}->}[r]^-{i} & (Z,M_Z),
}
\]
and $i'$ is universal among such.
Moreover, the following properties hold.
\begin{enumerate}
  \item If $i$ is an exact closed immersion, then $i'$ is the PD-envelope of $Y\hookrightarrow Z$ relative to $(\Sigma,\calJ_\Sigma, \gamma_\Sigma)$ with the pullback log structure from $M_Z$.
  \item If $\gamma_\Sigma$ extends to $Y$ (which is the case in (b)), $(Y',M_{Y'})=(Y,M_Y)$. 
\end{enumerate}
\end{thm}

\begin{proof}
First consider case (a). The existence of the PD-envelope and part (1) are deduced from \cite[Prop.~5.3]{Kato-log}; in fact, $i'$ exists and is an exact closed immersion of \emph{fine} log schemes. Part (2) is \cite[(5.5.2)]{Kato-log}.\footnote{Note that ``If $\gamma$ extends to $Y$'' therein should read ``If $\gamma$ extends to $X$''; see \cite[Prop.~I.2.4.3(iii)]{Berthelot-book}.} Case (b) is \cite[1.3, Thm.]{Beilinson-crystalline-period-map}.
\end{proof}

\begin{prop}\label{prop:etale pushforward from PD-envelope}
With the notation and assumption (a) in Theorem~\ref{thm:existence of log PD envelope}, let $q\colon D\rightarrow Z$ denote the resulting morphism of schemes.
If $p$ is locally nilpotent in $Z$, then $q_{\et,\ast}\colon \Ab(D_\et)\rightarrow \Ab(Z_\et)$ is exact.
\end{prop}

Note that $q_{\et,\ast}$ is obviously exact under assumption (b) by Theorem~\ref{thm:existence of log PD envelope}(2).

\begin{proof}
We need to recall the construction of $D$. We may check the assertion \'etale locally on $Z$. Observe that $i$ is still a closed immersion and $(Y,M_Y)$ is unchanged if $(Z,M_Z)$ is replaced with the exactification of $(Y,M_Y)\hookrightarrow (Z,M_Z)^\mathrm{int}$; under this additional assumption, $i'\colon Y'\hookrightarrow D$ is the PD-envelope of $Y\hookrightarrow Z$. Moreover, if we let $Y''\hookrightarrow Y$ denote the closed immersion defined by $\calJ_\Sigma\calO_Y$, the PD-envelope of the composite $Y''\hookrightarrow Z$ is of the form $i''\colon Y''\hookrightarrow D''$ by \cite[Cor.~I.2.3.2(iii)]{Berthelot-book}. It follows from \cite[Thm.~I.2.4.1, Pf.]{Berthelot-book} that $D''=D$. In summary, we have the commutative diagram
\[
\xymatrix{
Y'' \ar@{^{(}->}[r]^-{i''}\ar@{^{(}->}[d] & D\ar[d]\ar[d]^-q \\
Y \ar@{^{(}->}[r]^-{i} & Z.
}
\]
Note that $i'$ is a nil immersion since it is defined by the PD-ideal of $\calO_D$ generated by the image of the defining ideal of $i$ and $p$ is nilpotent in $\calO_D$. Hence $q_{\et,\ast}$ is exact.
\end{proof}

\section{Definition of big log crystalline sites}\label{sec: Definition of big log crystalline sites}

Kato \cite[\S~5, 6]{Kato-log} introduces the logarithmic crystalline theory for fine log schemes, and Beilinson \cite[\S~1]{Beilinson-crystalline-period-map} discusses a generalization of Kato's theory to the case of integral quasi-coherent log schemes. We closely follow Beilinson's formulation but work on the \emph{big} logarithmic crystalline sites.

\begin{set-up}\label{set-up:relative log-crystalline site}
Let $\Sigma^\sharp\coloneqq (\Sigma,M_\Sigma,\calJ_\Sigma,\gamma_\Sigma)$ be a log PD-scheme over $\Z_p$ such that $p$ is nilpotent in $\calO_\Sigma$ and $M_\Sigma$ is quasi-coherent, and let $(Y,M_Y)$ be a log $\Sigma^\sharp$-scheme such that $M_Y$ is integral and quasi-coherent. Write $\pi\colon (Y,M_Y)\rightarrow (\Sigma,M_\Sigma)$ for the structure map.
\end{set-up}

\begin{defn}\label{def:big relative logarithmic crystalline site}
Define \emph{the big relative logarithmic crystalline site}\footnote{Our convention is that $((Y,M_Y)/\Sigma^\sharp)_\CRIS$ denotes the crystalline site and $\Sh(((Y,M_Y)/\Sigma^\sharp)_\CRIS)$ denotes the crystalline topos.}
\[
((Y,M_Y)/\Sigma^\sharp)_\CRIS
\]
as follows:
\begin{itemize}
 \item an object is a tuple $(U,T,M_T, f, i,\gamma)$ where $f\colon U\rightarrow Y$ is a morphism of schemes over $\Sigma$, $(T,M_T)$ is a log $\Sigma^\sharp$-scheme with $M_T$ integral and quasi-coherent, $i$ is an exact closed immersion $(U,M_U\coloneqq f^\ast M_Y)\rightarrow (T,M_T)$ over $(\Sigma,M_\Sigma)$, and $\gamma$ is a PD-structure on the ideal of $\calO_T$ defining $U$ compatible with the PD-structure on $\calJ_\Sigma\calO_T$. Such a tuple is called \emph{affine} if $T$ (equivalently, $U$) is affine;
 \item a morphism from $(U,T,M_T,i,\gamma)$ to $(U',T',M_{T'},i',\gamma')$ is a pair of morphisms of log schemes $g_T\colon (T,M_T)\rightarrow (T', M_{T'})$ and $g_U\colon (U,M_U)\rightarrow (U',M_{U'})$ over $(\Sigma,M_\Sigma)$ such that $g_T\circ i=i'\circ g_U$, $f=f'\circ g_U$, and $g_T$ is compatible with the PD-structures $\gamma,\gamma'$. Note that $g_T$ is strict since $g_U$ is strict and all the log structures are integral. Such a morphism is called \emph{Cartesian} if the diagram
 \[
    \xymatrix{
    U \ar@{^(->}[r]^{i}\ar[d]_{g_U}&T\ar[d]^{g_T}\\ U'\ar@{^(->}[r]^{i'}&T'
    }
 \]
  is Cartesian, and it is called \emph{\'etale} if $g_T$ and $g_U$ are \'etale as morphisms of underlying schemes;
 \item a cover is a family of \'etale Cartesian morphisms that is jointly surjective.
\end{itemize}
Let $((Y,M_Y)/\Sigma^\sharp)_{\CRIS}^{\mathrm{aff}}\subset ((Y,M_Y)/\Sigma^\sharp)_\CRIS$ denote the full subcategory of affine objects with the induced topology.

By Lemma~\ref{lem:nonempty finite limit in abs log cristalline site}(2) below, $((Y,M_Y)/\Sigma^\sharp)_\CRIS$ and $((Y,M_Y)/\Sigma^\sharp)_{\CRIS}^{\mathrm{aff}}$ are sites. The inclusion is a special cocontinuous functor in the sense of \cite[Tag~03CG]{stacks-project} and thus induces an equivalence of topoi 
\[
\Sh(((Y,M_Y)/\Sigma^\sharp)_{\CRIS}^{\mathrm{aff}})\xrightarrow{\cong} \Sh(((Y,M_Y)/\Sigma^\sharp)_\CRIS).
\]

To simplify the notation, we often write $(U,T)$ or $(U,T,M_T)$ for $(U,T,M_T,i,\gamma)$. Write $\calJ_T$ for the nil ideal sheaf of $T$ that defines $U$. 
\end{defn}

\begin{lem}\label{lem:nonempty finite limit in abs log cristalline site}
\hfill
\begin{enumerate}
 \item The category $((Y,M_Y)/\Sigma^\sharp)_\CRIS$ has non-empty finite products. 
 \item For a diagram $(U_1,T_1)\xrightarrow{g} (V,Q) \leftarrow (U_2,T_2)$ in $((Y,M_Y)/\Sigma^\sharp)_\CRIS$, the fiber product exists. 
 Moreover, the fiber product is an affine object if $T_1$, $T_2$, and $Q$ are all affine. If $g$ is \'etale and Cartesian, so is the projection from the fiber product to $(U_2,T_2)$.
\end{enumerate}
\end{lem}

\begin{proof}
(1) 
For the existence of finite products, the construction (a) in \cite[1.5, Prop.]{Beilinson-crystalline-period-map} works, whose outline we recall now: take $(U_j,T_j)\in ((Y,M_Y)/\Sigma^\sharp)_\CRIS$ ($j=1,2$).
We set $U\coloneqq U_1\times_YU_2$ equipped with the pullback log structure along $U\rightarrow Y$. Let $(T'',M_{T''})$ denote the product $(T_1,M_{T_1})\times (T_2,M_{T_2})$ in the category of integral and quasi-coherent log schemes.
Let $(U,M_{U})\hookrightarrow (T',M_{T'})$ denote the PD-envelope of the closed immersion $i_{T''}\colon (U,M_U)\hookrightarrow (T'',M_{T''})$ relative to the log PD-scheme $(T'',M_{T''}, \calJ'',\gamma'')$ where $\calJ''\subset \calO_{T''}$ is the quasi-coherent ideal sheaf generated by the pullbacks of $\calJ_{T_j}+\calJ_\Sigma\calO_{T_j}$. Then the desired product is given by $(U,M_{U})\hookrightarrow (T,M_T)$ where $(T,M_T)$ is defined to be the largest exact closed log PD-subscheme of $(T',M_{T'})$ such that the two compositions $(T,M_T)\rightarrow (T_j,M_{T_j})\rightarrow (\Sigma,M_\Sigma)$ coincide.  

(2)
Set $W=U_1\times_VU_2$; this is a $Y$-scheme.
Consider the diagonal embedding
\[
i_T\colon (W,M_Y|_W)\hookrightarrow (T,M_T)\coloneqq (T_1,M_{T_1})\times_{(Q,M_Q)}(T_2,M_{T_2}),
\]
where the fiber product in the latter is taken in the category of integral and quasi-coherent log schemes. 
In fact, $T=T_1\times_Q T_2$, and $i_T$ is an exact closed immersion since the map $(T_j,M_{T_j})\rightarrow (Q,M_Q)$ is strict.
Let $(W',M_{W'})\hookrightarrow(T',M_{T'})$ denote the PD-envelope of $i_T\colon W\hookrightarrow T$ relative to $(\Sigma,\calJ_\Sigma, \gamma_\Sigma)$ with the pullback log structure from $M_T$. Then the morphism $T'\rightarrow T_j$ is a PD-morphism, and $(W',T')$ represents the desired fiber product.
When $T_1$, $T_2$, and $Q$ are all affine, so are $T$ and $T'$; to see that $T'$ is affine, recall the construction of PD-envelope in the non-log case. Hence $(W',T')$ is an affine object. 
Finally, if $g$ is \'etale and Cartesian, then $(T,M_T)\rightarrow (T_2,M_{T_2})$ is \'etale and $W=U_2\times_{T_2}T$. In particular, the defining ideal of $i_T$ is the restriction of the defining ideal of $(U_2,M_{U_2})\hookrightarrow (T_2,M_{T_2})$ to $T_\et$, and thus it admits a PD-structure extending the one on $\calJ_{T_2}$ and compatible with $\gamma_\Sigma$. Hence $(W',T')=(W,T)$ and the last assertion holds.
\end{proof}

\begin{rem}\label{rem:fine log structure in absolute log crystalline site}
Assume that $M_Y$ is fine. Then for every $(U,T)\in ((Y,M_Y)/\Sigma^\sharp)_\CRIS$, $(T,M_T)$ is a fine log scheme. In fact, since $(U,M_U)\hookrightarrow (T,M_T)$ is an exact closed immersion with $M_T$ integral, it is a log thickening in the sense of \cite[Def.~IV.2.1.1]{Ogus-log}. Since $M_U$ is the pullback of $M_Y$ and thus fine, we conclude that $M_T$ is also fine by \cite[Prop.~IV.2.1.3.1]{Ogus-log}.
\end{rem}

\begin{rem}\label{rem:small crystalline site}
Let $((Y,M_Y)/\Sigma^\sharp)_\cris\subset ((Y,M_Y)/\Sigma^\sharp)_\CRIS$ denote the full subcategory consisting of objects $(U,T)$ such that $U$ is \'etale over $Y$. It is easy to see from the proof of Lemma~\ref{lem:nonempty finite limit in abs log cristalline site} that non-empty finite inverse limits exist in $((Y,M_Y)/\Sigma^\sharp)_\cris$ and agree with those in $((Y,M_Y)/\Sigma^\sharp)_\CRIS$.
Hence $((Y,M_Y)/\Sigma^\sharp)_\cris$ becomes a site and admits a morphism of topoi 
\[r_{Y/\Sigma}\colon \Sh(((Y,M_Y)/\Sigma^\sharp)_\cris)\rightarrow \Sh(((Y,M_Y)/\Sigma^\sharp)_\CRIS)
\]
such that $r_{Y/\Sigma}^\ast$ is the restriction from $((Y,M_Y)/\Sigma^\sharp)_\CRIS$ to $((Y,M_Y)/\Sigma^\sharp)_\cris$.
Similar to the non-log case in \cite[III.4.1.3]{Berthelot-book}, one can also define another morphism of topoi 
\[
p_{Y/\Sigma}\colon \Sh(((Y,M_Y)/\Sigma^\sharp)_\CRIS)\rightarrow \Sh(((Y,M_Y)/\Sigma^\sharp)_\cris)
\]
such that 
\[
p_{Y/\Sigma}^\ast=r_{Y/\Sigma,!},\quad
p_{Y/\Sigma,\ast}=r_{Y/\Sigma}^\ast,\quad\text{and}\quad
p_{Y/\Sigma}^!=r_{Y/\Sigma,\ast}.
\]
Note that for every sheaf $\calF$ on $(Y,M_Y)/\Sigma^\sharp)_\CRIS$, we have
\[
\Gamma((Y,M_Y)/\Sigma^\sharp)_\CRIS, \calF)=\Gamma((Y,M_Y)/\Sigma^\sharp)_\cris, p_{Y/\Sigma,\ast}\calF)=\Gamma((Y,M_Y)/\Sigma^\sharp)_\cris,r_{Y/\Sigma}^\ast\calF).
\]
\end{rem}

Observe that for every $(U,T)\in ((Y,M_Y)/\Sigma^\sharp)_\CRIS$ and every \'etale morphism $T'\rightarrow T$, the induced closed immersion $U'\coloneqq U\times_TT'\hookrightarrow T'$ naturally defines an object $(U',T')$ of $((Y,M_Y)/\Sigma^\sharp)_\CRIS$ by \cite[Cor.~3.22]{berthelot-ogus-book}, and the induced morphism $(U',T')\rightarrow (U,T)$ is \'etale and Cartesian. 
Hence giving a sheaf $\calF$ on $((Y,M_Y)/\Sigma^\sharp)_\CRIS$ is the same as giving a sheaf $\calF_{(U,T)}$ in the small \'etale topos $\Sh(T_\et)=\Sh(U_\et)$ for every $(U,T)\in ((Y,M_Y)/\Sigma^\sharp)_\CRIS$ together with a transition morphism $g_\calF^\ast\colon g^{-1}_\et\calF_{(U,T)}\rightarrow \calF_{(U',T')}$ for every morphism $g\colon (U',T')\rightarrow (U,T)$ satisfying the standard compatibilities for the identity and composition. 

\begin{defn}\label{defn:localization of crystalline site by open}
For any open $V\subset Y$, set $V_\CRIS(U,T)=\{\ast\}$ if $U\rightarrow Y$ factors through $V$ and $V_\CRIS(U,T)=\emptyset$ otherwise. Then $V_\CRIS$ is a subsheaf of the final object of $\Sh(((Y,M_Y)/\Sigma^\sharp)_\CRIS)$. Moreover, the localization $\Sh(((Y,M_Y)/\Sigma^\sharp)_\CRIS)/V_\CRIS$ is equivalent to $\Sh(((V,M_V)/\Sigma^\sharp)_\CRIS)$ (cf.~\cite[Prop.~III.3.1.2]{Berthelot-book}).
\end{defn}

\begin{defn}
For $(U,T)\in ((Y,M_Y)/\Sigma^\sharp)_\CRIS$, set
\[
\calO_{Y/\Sigma}(U,T)=\Gamma(T,\calO_T) \quad\text{and}\quad
\calJ_{Y/\Sigma}(U,T)=\Gamma(T,\calJ_T).
\]
These presheaves are indeed sheaves on $((Y,M_Y)/\Sigma^\sharp)_\CRIS$ by the preceding discussion.
We write $\calJ_{Y/\Sigma}^{[n]}$ for the $n$th divided power of $\calJ_{Y/\Sigma}$.
\end{defn}

\begin{defn}
A \emph{crystal of $\calO_{Y/\Sigma}$-modules} on $((Y,M_Y)/\Sigma^\sharp)_\CRIS$ is a sheaf $\calF$ of $\calO_{Y/\Sigma}$-modules such that for every morphism $g\colon (U',T')\rightarrow (U,T)$ in $(Y,M_Y)_\CRIS$, the induced map
 \[
  g_\et^\ast\calF_{(U,T)}\coloneqq \calO_{T'}\otimes_{g_{\et}^{-1}\calO_T} g_{\et}^{-1}\calF_{(U,T)}\rightarrow \calF_{(U',T')}
 \]
is an isomorphism of sheaves on $T'_\et$.
Similarly, we define the notion of crystals of $\calA$-modules for any $\calO_{Y/\Sigma}$-algebra $\calA$.
\end{defn}

\begin{prop}\label{prop:quasi-coherent crystals}
Let $\calF$ be an $\calO_{Y/\Sigma}$-module $\calF$.
\begin{enumerate}
 \item $\calF$ is flat (i.e., the functor $-\otimes_{\calO_{Y/\Sigma}}\calF$ is exact on the category of $\calO_{Y/\Sigma}$-modules) if and only if $\calF_{(U,T)}$ is a flat $\calO_T$-module for every $(U,T)\in ((Y,M_Y)/\Sigma^\sharp)_\CRIS$.
 \item $\calF$ is quasi-coherent (resp.~of finite type) if and only if $\calF$ is a crystal of $\calO_{Y/\Sigma}$-module, and $\calF_{(U,T)}$ is quasi-coherent (resp.~of finite type) as an $\calO_T$-module for every $(U,T)\in ((Y,M_Y)/\Sigma^\sharp)_\CRIS$.
\end{enumerate}
\end{prop}

See, for example, \cite[Tag 03DL]{stacks-project} for the definition of quasi-coherent modules (resp. modules of finite type) on a ringed site.

\begin{proof}
Part (1) is straightforward; for (2), the proof of \cite[Prop.~IV.1.1.3]{Berthelot-book} works verbatim.
\end{proof}

\begin{notation}
For $(U,T)\in ((Y,M_Y)/\Sigma^\sharp)_\CRIS$, the functor $\Sh(((Y,M_Y)/\Sigma^\sharp)_\CRIS)\rightarrow \Sh(T_\et)$ sending $\calF$ to $\calF_{(U,T)}$ is exact (i.e., it commutes with any finite limits and colimits).
Hence it extends to 
\[
D(((Y,M_Y)/\Sigma^\sharp)_\CRIS,\calO_{Y/\Sigma})\rightarrow D(T_\et,\calO_T), 
\]
which we still denote by $\calK\mapsto \calK_{(U,T)}$ (cf.~\cite[V.1.1.1]{Berthelot-book}).
\end{notation}

\section{Functoriality of crystalline topoi}\label{sec: Functoriality of crystalline topoi}
Big logarithmic crystalline sites enjoy natural functoriality (cf.~\cite[III.4.2]{Berthelot-book}):
let $\phi\colon \Sigma'^\sharp=(\Sigma',M_{\Sigma'},\calJ_{\Sigma'},\gamma_{\Sigma'}) \rightarrow \Sigma^\sharp=(\Sigma,M_\Sigma,\calJ_\Sigma,\gamma_\Sigma)$ be a morphism of log PD-schemes with quasi-coherent log structures over $\Z_p$ and assume that $p$ is nilpotent in $\calO_\Sigma$. 
Suppose that we have the following commutative diagram
\begin{equation}\label{eq:functoriality diagram in crystalline topos}
\xymatrix{
(Y',M_{Y'})\ar[r]^-f \ar[d]_-{\pi'}& (Y,M_Y)\ar[d]_-\pi\\
\Sigma'^\sharp=(\Sigma',M_{\Sigma'},\calJ_{\Sigma'},\gamma_{\Sigma'}) \ar[r]^-\phi& \Sigma^\sharp=(\Sigma,M_\Sigma,\calJ_\Sigma,\gamma_\Sigma),
}
\end{equation}
where $f$ is a morphism of integral and quasi-coherent log schemes, and $(Y,M_Y)$ (resp.~$(Y',M_{Y'})$) is a log $\Sigma^\sharp$-scheme (resp.~log $\Sigma'^\sharp$-scheme).

\begin{defn}
An \emph{$f$-PD morphism} from $(U',T',M_{T'},i',\gamma')\in ((Y',M_{Y'})/\Sigma'^\sharp)_\CRIS$ to $(U,T, M_T,i,\gamma)\in ((Y,M_Y)/\Sigma^\sharp)_\CRIS$ (relative to $\phi$) consists of a commutative diagram 
\[
\xymatrix{
\Sigma'^\sharp\ar[d]_-\phi &(Y',M_{Y'})\ar[d]^-f \ar[l]^-{\pi'} & (U',M_{U'}\coloneqq g'^\ast M_{Y'})\ar[l]^-{g'}\ar[d]^-{h_{U',U}}\ar@{^{(}->}[r]_-{i'} & (T',M_{T'})\ar[d]^-{h_{T',T}}\ar@/_1.2pc/[lll]\\
\Sigma^\sharp &(Y,M_Y) \ar[l]_-{\pi}& (U,M_U\coloneqq g^\ast M_Y)\ar[l]_-{g}\ar@{^{(}->}[r]^-i & (T,M_T),\ar@/^1.2pc/[lll]
}
\]
where $h_{U',U}$ and $h_{T',T}$ are morphisms of log schemes such that $h_{T',T}$ is compatible with the PD-structures $\gamma_\Sigma, \gamma_{\Sigma'},\gamma, \gamma'$. 

For $(U,T)\in ((Y,M_Y)/\Sigma^\sharp)_\CRIS$, define the presheaf $f^\ast(U,T)$ on $((Y',M_{Y'})/\Sigma'^\sharp)_\CRIS$ by
\[
\bigl(f^\ast(U,T)\bigr)(U',T')\coloneqq \Hom_{f\text{-PD}}\bigl((U',T'),(U,T)\bigr).
\]
It is straightforward to check that $f^\ast(U,T)$ is a sheaf of sets. 
\end{defn}

\begin{construction}\label{construction:strict log structure for functoriality}
Keep the notation as above.
For $(U',T')\in ((Y',M_{Y'})/\Sigma'^\sharp)_\CRIS$, let $M_{(T',f)}$ be the \'etale sheaf of monoids on $T'$ defined as the fiber product $(f\circ g')^\ast M_Y \times_{M_{U'}}M_{T'}$ appearing in the diagram
\[
\xymatrix{
M_{(T',f)} \ar[r]\ar[d]   & M_{T'}\ar[d] \ar[r]& \calO_{T'} \ar[d]\\
(f\circ g')^\ast M_Y \ar[r]  &M_{U'} \ar[r]&\calO_{U'}
}
\]
under the equivalence $T'_\et\xrightarrow{\cong}U'_\et$ of small \'etale sites.
The pair $(T',M_{(T',f)})$ is an integral and quasi-coherent log scheme with exact closed immersion $i'\colon (U',(f\circ g')^\ast M_Y) \hookrightarrow (T',M_{(T',f)})$ and defines an object $(U',T',i',M_{(T',f)},\gamma')\in ((Y,M_Y)/\Sigma^\sharp)_\CRIS$. 
In fact, to see that $M_{(T',f)}$ is integral and quasi-coherent, we may work \'etale locally and assume that $T'=\Spec A$ and $U'=\Spec A/I$ are affine and $M_Y$ admits a chart $P\rightarrow \Gamma(Y,M_Y)$ from an integral monoid $P$. Set $\widetilde{P}\coloneqq P\times_{\Gamma(U',M_{U'})}\Gamma(T',M_{T'})$. Then $\widetilde{P}$ is integral and the map $\widetilde{P}\rightarrow P$ is a $(1+I)$-torsor (cf.~\cite[\S 1.1, Ex.~(iii)]{Beilinson-crystalline-period-map}). The induced map $\widetilde{P}^a_{T'}\rightarrow M_{(T',f)}$ of log structures on $T'_\et$ is an isomorphism since the resulting map $(\widetilde{P}^a_{T'})_{\overline{t}}/\calO_{T',\overline{t}}^\times\rightarrow (M_{(T',f)})_{\overline{t}}/\calO_{T',\overline{t}}^\times$ is an isomorphism for every geometric point $\overline{t}$ of $T'$ (cf.~\cite[Prop.~II.2.1.4]{Ogus-log}).
\end{construction}

The verification of the following lemma is straightforward.
\begin{lem}\label{lem:f-PH morphism}
The $f$-PD morphism $(U',T',i',M_{T'},\gamma')\rightarrow (U',T',i',M_{(T',f)},\gamma')$ is an initial object in the category of $f$-PD morphisms from $(U',T')$. 
\end{lem}

\begin{prop}\label{prop:functoriality of relative crystalline topoi}
There exists a unique morphism of topoi
\[
f_\CRIS=(f_\CRIS^\ast,f_{\CRIS,\ast})\colon \Sh(((Y',M_{Y'})/\Sigma'^\sharp)_\CRIS)\rightarrow \Sh(((Y,M_Y)/\Sigma^\sharp)_\CRIS)
\]
such that
\[
f_{\CRIS,\ast}(\calF')(U,T)= \Hom(f^\ast(U,T), \calF')
\]
and
\[
(f_\CRIS^\ast\calF)(U',T',i',M_{T'},\gamma')=\calF(U',T',i',M_{(T',f)},\gamma').
\]
Moreover, $f_\CRIS^\ast$ is exact on abelian sheaves and $f_\CRIS^\ast\calO_{Y/\Sigma}=\calO_{Y'/\Sigma'}$.
\end{prop}

As usual, $f_\CRIS^\ast\calF$ is also written as $\calF|_{((Y',M_{Y'})/\Sigma'^\sharp)_\CRIS}$.

\begin{proof}
 Set $C=((Y,M_Y)/\Sigma^\sharp)_\CRIS$, $C'=((Y',M_{Y'})/\Sigma'^\sharp)_\CRIS$, and $\varphi(U,T)=f^\ast(U,T)$.
By applying \cite[Prop.~5.7]{berthelot-ogus-book} to this triple, we obtain the adjoint pair 
\[
f_\CRIS=(f_\CRIS^\ast,f_{\CRIS,\ast})\colon \operatorname{PSh}(((Y',M_{Y'})/\Sigma'^\sharp)_\CRIS)\rightarrow\operatorname{PSh}(((Y,M_Y)/\Sigma^\sharp)_\CRIS)
\]
given by $f_{\CRIS,\ast}(\calF')(U,T)\coloneqq \Hom(f^\ast(U,T), \calF')$ and 
$f_{\CRIS}^\ast(\calF)(U',T')=\varinjlim_{(U,T)}\calF(U,T)$ where $(U,T)$ runs over the category of $f$-PD morphisms from $(U',T')$.
It is easy to check that $f_{\CRIS,\ast}$ sends a sheaf to a sheaf. By Lemma~\ref{lem:f-PH morphism}, we see 
\begin{equation}\label{eq:pullback formula for crystalline site}
f_{\CRIS}^\ast(\calF)(U',T',i',M_{T'},\gamma')=\calF(U',T',i',M_{(T',f)},\gamma').
\end{equation}
Hence $f_{\CRIS}^\ast$ sends a sheaf to a sheaf and commutes with finite limits and colimits, making $f_\CRIS$ the desired morphism of topoi. The last assertions also follow from this.
\end{proof}

Note that the exactness of $f_\CRIS^\ast$ comes from the fact that we work on the big site. It also implies that the inverse image functor for the morphism of ringed topoi $(\Sh((Y',M_{Y'})_\CRIS),\calO_{Y'/\Sigma'})\rightarrow (\Sh((Y,M_Y)_\CRIS),\calO_{Y/\Sigma})$ is computed by the same formula \eqref{eq:pullback formula for crystalline site},  which will be used implicitly when we consider the inverse image of $\calO_{Y/\Sigma}$-modules.

\begin{prop}\label{prop:crystalline pullback along closed immersion}
Assume that $\Sigma'^\sharp=\Sigma^\sharp$, $\phi=\id$, and $f=i\colon (Y',M_{Y'})\rightarrow (Y,M_Y)$ is a closed immersion of fine log schemes. For $(U,T)\in ((Y,M_Y)/\Sigma^\sharp)_\CRIS$, let $U'\coloneqq U\times_YY'$ equipped with the pullback log structure $M_{U'}$ from $M_{Y'}$, and let $i_D\colon (U'',M_{U''})\hookrightarrow(D,M_D,\gamma_D)$ denote the PD-envelope of the closed immersion $(U',M_{U'})\hookrightarrow (T,M_T)$ relative to $\Sigma^\sharp$. Then $(U'',D,i_D,M_D,\gamma_D)$ is an object of $((Y',M_{Y'})/\Sigma^\sharp)_\CRIS$ and represents the sheaf $f^\ast(U,T)$.
\end{prop}

\begin{proof}
Note that the PD-envelope $i_D$ exists by Theorem~\ref{thm:existence of log PD envelope} and Remark~\ref{rem:fine log structure in absolute log crystalline site}; now the assertion follows from the universality of the PD-envelope and Proposition~\ref{prop:functoriality of relative crystalline topoi}.
\end{proof}

\begin{cor}\label{cor:exactness of i-pushforward}
With the set-up and assumptions as in Proposition~\ref{prop:crystalline pullback along closed immersion}, the functor $i_{\CRIS,\ast}\colon \Ab(((Y',M_{Y'})/\Sigma^\sharp)_\CRIS)\rightarrow\Ab(((Y,M_Y)/\Sigma^\sharp)_\CRIS)$ is exact.
\end{cor}

\begin{proof}
Take $(U,T)\in ((Y,M_Y)/\Sigma^\sharp)_\CRIS$ and let $q\colon D\rightarrow T$ denote the induced map. For any sheaf $\calF'$ on $((Y',M_{Y'})/\Sigma^\sharp)_\CRIS$, the associated \'etale sheaf $\calF'_{(U,T)}$ on $T_\et$ is $q_{\et,\ast}(\calF'_{(U',D)})$. Now the exactness of $i_{\CRIS,\ast}$ follows from Proposition~\ref{prop:etale pushforward from PD-envelope}.
\end{proof}

\begin{rem}\label{rem:functoriality of small crystalline sites}
Keep the notation as in \eqref{eq:functoriality diagram in crystalline topos}. Then we also have the morphism of small crystalline topoi
\[
f_\cris=(f_\cris^\ast,f_{\cris,\ast})\colon \Sh(((Y',M_{Y'})/\Sigma'^\sharp)_\cris)\rightarrow \Sh(((Y,M_Y)/\Sigma^\sharp)_\cris)
\]
such that $f_{\cris,\ast}(\calF')(U,T)= \Hom(f^\ast(U,T), \calF')$: see \cite[(5.9)]{Kato-log}. Moreover, one can easily check, as in \cite[III.4.2.2]{Berthelot-book}, that there is a canonical isomorphism of morphisms of topoi $f_\cris\circ p_{Y'/\Sigma'}\cong p_{Y/\Sigma}\circ f_{\CRIS}$.
\end{rem}

\section{Projections to small \'etale topoi}\label{sec: Projections to small etale topoi}
Keep the notation as in Set-up~\ref{set-up:relative log-crystalline site}.

\begin{prop}\label{prop:projection from crystallie site to etale site}
There exists a morphism of topoi (called the \emph{projection from crystalline topos to \'etale topos})
\[
u_{Y/\Sigma}\coloneqq u_{(Y,M_Y)/\Sigma^\sharp}\colon \Sh(((Y,M_Y)/\Sigma^\sharp)_\CRIS)\rightarrow \Sh(Y_\et)
\]
given by
\[
(u_{(Y,M_Y)/\Sigma^\sharp,\ast}\calF)(V)=\Gamma(((V,M_V)/\Sigma^\sharp)_\CRIS,\calF|_{((V,M_V)/\Sigma^\sharp)_\CRIS})
\]
(where $M_V$ is the pullback log structure from $M_Y$) and
\[
(u_{(Y,M_Y)/\Sigma^\sharp}^{-1}\calG)(U,T,g\colon U\rightarrow Y)=(g_\et^{-1}\calG)(U).
\]
\end{prop}

Note that $\gamma_\Sigma$ extends to $\calO_V$ since $V$ is \'etale over $Y$.

\begin{proof}
By Proposition~\ref{prop:functoriality of relative crystalline topoi}, we have $(\calF|_{((V,M_V)/\Sigma^\sharp)_\CRIS})(U,T)=\calF(U,T)$ for any $(U,T)\in ((V,M_V)/\Sigma^\sharp)_\CRIS$.
One can deduce from this that the first formula defines a sheaf.
We easily see that $(u_{(Y,M_Y)/\Sigma^\sharp}^{-1}\calG)$ is a sheaf and $u_{(Y,M_Y)/\Sigma^\sharp}^{-1}$ commutes with finite limits.
It is now straightforward to see that $u_{(Y,M_Y)/\Sigma^\sharp}^{-1}$ is left adjoint to $u_{(Y,M_Y)/\Sigma^\sharp,\ast}$.
\end{proof}

Note that $u_{Y/\Sigma}$ is not a morphism of ringed topoi for $\calO_{Y/\Sigma}$ and $\calO_Y$.

\begin{defn}
For $V\in Y_\et$, set $V_\CRIS\coloneqq u_{Y/\Sigma}^{-1}h_V$, where $h_V$ denotes the \'etale sheaf on $Y$ represented by $V$. Note that this agrees with Definition~\ref{defn:localization of crystalline site by open} when $V$ is an open subscheme of $Y$.
\end{defn}

\begin{lem}\label{lem:properties of VCRIS}
Let $V,V'\in Y_\et$.
\begin{enumerate}
 \item The localized topos $\Sh(((Y,M_Y)/\Sigma^\sharp)_\CRIS)/V_\CRIS$ is naturally identified with $\Sh(((V,M_V)/\Sigma^\sharp)_\CRIS)$.
 \item There is a natural isomorphism $V_\CRIS\times V_\CRIS'\cong (V\times_YV')_\CRIS$.
\end{enumerate}
\end{lem}

\begin{proof}
Both are straightforward from the definition.
\end{proof}

\begin{rem}\label{rem:u-pushforward depends only on restriction to small site}
Similar to the non-log case in \cite[III, \S~3.2, (4.4.9)]{Berthelot-book}, one can also define a morphism of topoi $u_{Y/\Sigma,\mathrm{small}}\colon \Sh(((Y,M_Y)/\Sigma^\sharp)_\cris)\rightarrow \Sh(Y_\et)$ together with a natural isomorphism of morphisms of topoi $u_{Y/\Sigma,\mathrm{small}}\circ p_{Y/\Sigma}\cong u_{Y/\Sigma}$. It follows that for a morphism $a\colon \calF\rightarrow\calF'$ of abelian sheaves on $((Y,M_Y)/\Sigma^\sharp)_\CRIS$, if the restriction $p_{Y/\Sigma,\ast}a=r_{Y/\Sigma}^\ast a$ of $a$ to $((Y,M_Y)/\Sigma^\sharp)_\cris$ is an isomorphism, then $Ru_{Y/\Sigma,\ast}a\colon Ru_{Y/\Sigma,\ast}\calF\rightarrow Ru_{Y/\Sigma,\ast}\calF'$ is a quasi-isomorphism.
\end{rem}

\begin{prop}\label{prop:compatibility of u and crystalline pushforward}
Let $f\colon (Y', M_{Y'})\rightarrow (Y,M_Y)$ be a morphism  of integral and quasi-coherent log schemes over a morphism $\phi\colon \Sigma'^\sharp\rightarrow \Sigma^\sharp$ of quasi-coherent log PD-schemes as in \eqref{eq:functoriality diagram in crystalline topos}.
There is a canonical isomorphism of morphisms of topoi
\[
f_\et\circ u_{Y'/\Sigma'}\xrightarrow{\cong}u_{Y/\Sigma}\circ f_\CRIS\colon \Sh((Y',M_{Y'})/\Sigma'^\sharp)_\CRIS)\rightarrow \Sh(Y_\et).
\]
\end{prop}

\begin{proof}
For every $V\in Y_\et$, set $f^{-1}(V)\coloneqq V\times_YY'$ and let $M_{f^{-1}(V)}$ denote the restriction of $M_{Y'}$ to $f^{-1}(V)_\et$. For $\calF'\in \Sh((Y',M_{Y'})/\Sigma'^\sharp)_\CRIS)$, we have
\begin{align*}
((f_{\et,\ast}\circ u_{Y'/\Sigma',\ast})\calF')(V)
&=(u_{Y'/\Sigma',\ast})\calF')(f^{-1}(V))\\
&=\Gamma(((f^{-1}(V),M_{f^{-1}(V)})/\Sigma'^\sharp)_\CRIS,\calF'|_{((f^{-1}(V),M_{f^{-1}(V)})/\Sigma'^\sharp)_\CRIS})
\end{align*}
and
\[
((u_{Y/\Sigma,\ast}\circ f_{\CRIS,\ast})\calF')(V)=\Gamma(((V,M_V)/\Sigma^\sharp)_\CRIS,(f_{\CRIS,\ast}\calF')|_{((V,M_V)/\Sigma^\sharp)_\CRIS}).
\]
Using Lemma~\ref{lem:f-PH morphism}, one can see that these two are canonically isomorphic, and the resulting isomorphism is functorial in $V$ and $\calF$.
\end{proof}

\begin{lem}
Let $(f\colon V\rightarrow Y)\in Y_\et$. For $\calF\in \Ab(((Y,M_Y)/\Sigma^\sharp)_\CRIS)$, the map $f_\et^{-1}Ru_{Y/\Sigma,\ast}\calF\rightarrow Ru_{V/\Sigma,\ast}(f_\CRIS^{\ast}\calF)$ given by adjunction is a quasi-isomorphism.
\end{lem}

We often write the above map as $(Ru_{Y/\Sigma,\ast}\calF)|_{V_\et}\rightarrow Ru_{V/\Sigma,\ast}(\calF|_{((V,M_V)/\Sigma^\sharp)_\CRIS})$.

\begin{proof}
We see from Lemma~\ref{lem:properties of VCRIS}(1) that the map $f_\et^{-1}u_{Y/\Sigma,\ast}\calF\rightarrow u_{V/\Sigma,\ast}(f_\CRIS^{\ast}\calF)$ is an isomorphism and $f_\CRIS^{\ast}$ sends injective sheaves to injective sheaves. So the assertion follows easily by taking an injective resolution of $\calF$.
\end{proof}

\begin{prop}\label{prop:morphism of topoi i from et to cris}
There exists a morphism of topoi
\[
i_{Y/\Sigma}\coloneqq i_{(Y,M_Y)/\Sigma^\sharp}\colon \Sh(Y_\et)\rightarrow \Sh(((Y,M_Y)/\Sigma^\sharp)_\CRIS)
\]
given by
\[
(i_{(Y,M_Y)/\Sigma^\sharp,\ast}\calG)(U,T, g\colon U \rightarrow Y)= (g_\et^{-1}\calG)(U)
\] 
and
\[
(i_{(Y,M_Y)/\Sigma^\sharp}^{-1}\calF)(V)=\calF(V,V,\id_V\colon V\rightarrow V).
\]
\end{prop}

\begin{proof}
This is straightforward (cf.~\cite[\S III.4.4]{Berthelot-book}).
\end{proof}

Note that $i_{(Y,M_Y)/\Sigma^\sharp,\ast}=u_{(Y,M_Y)/\Sigma^\sharp}^{-1}$. In other words, $i_{(Y,M_Y)/\Sigma^\sharp}^{-1}$ is the left adjoint to $u_{(Y,M_Y)/\Sigma^\sharp}^{-1}$. So we also write $u_{(Y,M_Y)/\Sigma^\sharp,!}$ for $i_{(Y,M_Y)/\Sigma^\sharp}^{-1}$. 

\begin{defn}\label{def:projection to etale site of base}
Let $\pi_{Y/\Sigma}\coloneqq \pi_{(Y,M_Y)/\Sigma^\sharp}$ denote the composite
\[
\Sh(((Y,M_Y)/\Sigma^\sharp)_\CRIS)\xrightarrow{u_{Y/\Sigma}} \Sh(Y_\et)\xrightarrow{\pi_\et}\Sh(\Sigma_\et).
\]
Since $\calO_{Y/\Sigma}$ is a $\pi_{Y/\Sigma}^{-1}\calO_{\Sigma}$-algebra, $\pi_{Y/\Sigma}$ yields a morphism of the ringed topoi
\[
(\Sh(((Y,M_Y)/\Sigma^\sharp)_\CRIS),\calO_{Y/\Sigma})\rightarrow(\Sh(\Sigma_\et),\calO_\Sigma),
\]
for which we still write $\pi_{Y/\Sigma}=(\pi_{Y/\Sigma}^\ast,\pi_{Y/\Sigma,\ast})$.
\end{defn}

The morphism $\pi_{Y/\Sigma}$ plays a primary role in computing the higher direct image of quasi-coherent $\calO_{Y/\Sigma}$-modules. To explain this, consider a morphism $f\colon (Y', M_{Y'})\rightarrow (Y,M_Y)$ of integral and quasi-coherent log schemes over the identity  $\phi=\id\colon \Sigma'^\sharp= \Sigma^\sharp\rightarrow \Sigma^\sharp$ of a quasi-coherent log PD-scheme as in \eqref{eq:functoriality diagram in crystalline topos}.

Take $(U,T)\in ((Y,M_Y)/\Sigma^\sharp)_\CRIS$. Set
\[
(f^{-1}(U),M_{f^{-1}(U)})\coloneqq (Y',M_{Y'})\times_{(Y,M_Y)}(U,M_U)
\]
where the fiber product is taken in the category of integral and quasi-coherent log schemes. Indeed, since $(U,M_U)\rightarrow(Y,M_Y)$ is strict, $f^{-1}(U)=Y'\times_Y U$ as schemes and $M_{f^{-1}(U)}$ is the pullback log structure from $M_{Y'}$. Let $T^\sharp$ denote the log PD-scheme $(T,M_T,\calJ_T+\calJ_\Sigma\calO_T,\tilde{\gamma})$ where $\tilde{\gamma}$ is the unique PD-structure on $\calJ_T+\calJ_\Sigma\calO_T$ extending $\gamma$ and $\gamma_\Sigma$.

\begin{lem}\label{lem:evalutation of crystalline pushforward via projection}
With the notation as above, take any sheaf $\calF$ on $((Y',M_{Y'})/\Sigma^\sharp)_\CRIS$
and continue to write $\calF$ for its pullback to $((f^{-1}(U),M_{f^{-1}(U)})/T^\sharp)_\CRIS$.
\begin{enumerate}
 \item We have
\[
\Gamma((U,T),f_{\CRIS,\ast}\calF)
=\Gamma(((f^{-1}(U),M_{f^{-1}(U)})/T^\sharp)_\CRIS,\calF)
\]
and
\[
(f_{\CRIS,\ast}\calF)_{(U,T)}=\pi_{f^{-1}(U)/T,\ast}\calF.
\]
 \item There are natural quasi-isomorphisms
\[
R\Gamma((U,T), Rf_{\CRIS,\ast}\calF)\cong R\Gamma(((f^{-1}(U),M_{f^{-1}(U)})/T^\sharp)_\CRIS,\calF)
\]
and
\[
(Rf_{\CRIS,\ast}\calF)_{(U,T)}=R\pi_{f^{-1}(U)/T,\ast}\calF.
\]
\end{enumerate}
\end{lem}

\begin{proof}
(1) We know from Proposition~\ref{prop:functoriality of relative crystalline topoi} that $\Gamma((U,T),f_{\CRIS,\ast}\calF)=\Hom(f^\ast(U,T),\calF)$. Arguing as in \cite[p.~318]{Berthelot-book}, we can see that  the latter is canonically identified with $\Gamma(((f^{-1}(U),M_{f^{-1}(U)})/T^\sharp)_\CRIS,\calF)$. By varying this equality over objects of $T_\et$ and 
unwinding the definition of $\pi_{f^{-1}(U)/T,\ast}$, we obtain the identification $(f_{\CRIS,\ast}\calF)_{(U,T)}=\pi_{f^{-1}(U)/T,\ast}\calF$.

(2) Since $f_\CRIS^\ast$ is exact, $f_{\CRIS,\ast}$ sends injectives to injectives. So we obtain the first quasi-isomorphism. The second follows from (1). 
\end{proof}

\section{Embedding into a smooth log scheme}\label{sec:Embedding into a smooth log scheme}
Throughout this section, we consider the following setting and discuss the \v{C}ech--Alexander method, a description of crystals in terms of quasi-nilpotent connections, and the crystalline Poincar\'e lemma.

\begin{set-up}\label{set-up:description of crystals in terms of connections}
Let $\Sigma^\sharp$ be affine log PD-scheme in which $p$ is nilpotent and $i\colon (Y,M_Y)\hookrightarrow (Z,M_Z)$ a closed immersion  of \emph{fine} log $\Sigma^\sharp$-schemes (in particular, $i^\ast M_Z\rightarrow M_Y$ is surjective) such that $(Z,M_Z)\rightarrow (\Sigma,M_\Sigma)$ is smooth (as log schemes). 
Let $\omega^1_{Z/\Sigma}\coloneqq \omega^1_{(Z,M_Z)/(\Sigma,M_\Sigma)}$ be the sheaf of log differentials. 
\end{set-up}

We first introduce several PD-envelopes.
For $\nu\geq 0$, write $(Z^{\nu+1},M_{Z^{\nu+1}})$ for the $(\nu+1)$st self-fiber product $(Z,M_{Z})\times_{(\Sigma,M_{\Sigma})}\cdots\times_{(\Sigma,M_{\Sigma})}(Z,M_Z)$ (as an integral log scheme). 
Let $(Z,M_{Z})\hookrightarrow (D_Z(\nu),M_{D_Z(\nu)})$ denote the PD-envelope of $(Z,M_Z)\hookrightarrow (Z^{\nu+1},M_{Z^{\nu+1}})$ relative to $\Sigma^\sharp$, and let $(Y,M_{Y})\hookrightarrow (D_Y(Z^{\nu+1}),M_{D_Y(Z^{\nu+1})})$ denote the PD-envelope of $(Y,M_Y)\hookrightarrow (Z^{\nu+1},M_{Z^{\nu+1}})$ relative to $\Sigma^\sharp$ (see Theorem~\ref{thm:existence of log PD envelope}(2)); moreover, the construction shows that the underlying schemes are all affine over $Z$. For simplicity, set $D\coloneqq D_Y(Z^1)$, and for $1\leq l\leq \nu+1$, we let $p_l$ denote the $i$th projection from $Z^{\nu+1}$, $D_Z(\nu)$, or $D_Y(Z^{\nu+1})$ to $Z$. Following the standard convention, we write $\calD$ for $\calO_D$ and use a similar notation for $\calD_Z(\nu)$ and $\calD_Y(Z^{\nu+1})$; we often regard them as sheaves of rings on $Z_\et$ and write $p_{l,\ast}\calD_{Z}(\nu)$ or $p_{l,\ast}\calD_Y(Z^{\nu+1})$ if we specify the $\calO_Z$-algebra structure given by the $l$th projection $p_l$.

Recall the construction of $(D_Z(\nu),M_{D_Z(\nu)})$. Assume first that there is a factorization $(Z,M_Z)\hookrightarrow ((Z^{\nu+1})',M_{(Z^{\nu+1})'})\rightarrow (Z^{\nu+1},M_{Z^{\nu+1}})$ where the first morphism is an exact closed immersion and the second is \'etale as a morphism between log schemes. Then $D_Z(\nu)$ is the usual PD-envelope of the closed immersion $Z\hookrightarrow (Z^{\nu+1})'$ relative to $(\Sigma,\calJ_\Sigma,\gamma_\Sigma)$ and $M_{D_Z(\nu)}$ is the pullback of $M_{(Z^{\nu+1})'}$.
We also know from \cite[Rem.~5.8]{Kato-log}, or \cite[Rem.~IV.1.1.8, IV.3.4.5]{Ogus-log}
and \cite[Cor.I.4.4.3]{Berthelot-book}, that $\omega_{Z/\Sigma}^1\cong H/H^{[2]}$ where $H= \Ker(\calD_Z(1)\rightarrow\calO_Z)$.
In general, such a factorization exists \'etale locally by \cite[Prop.~4.10(1)]{Kato-log}, and the above construction satisfies \'etale descent. So we obtain $D_Z(\nu)$ in the general case and the above properties continue to hold.

To give a more concrete description of $(Z^{\nu+1})'$, take, \'etale locally, a chart $(Q_Z\rightarrow M_Z,P_\Sigma\rightarrow M_\Sigma, \theta\colon P\rightarrow Q)$ for $(Z,M_Z)\rightarrow (\Sigma,M_\Sigma)$ where $\theta$ is an injective map of integral monoids such that the torsion part of the cokernel of $P^\gp\rightarrow Q^\gp$ is a finite group of order invertible in $\Sigma$ and $Z\rightarrow \Sigma\times_{\Spec \Z[P]}\Spec \Z[Q]$ is \'etale (\cite[Thm.~IV.3.3.1]{Ogus-log}).
Then the summation map $\Delta(\nu)\colon Q(\nu)\coloneqq Q\oplus_P\cdots\oplus_PQ\rightarrow Q$ from the $(\nu+1)$st amalgamated sum (in the category of integral monoids) yields a chart of $(Z,M_Z)\hookrightarrow (Z^{\nu+1},M_{Z^{\nu+1}})$ (cf.~\cite[Prop.~I.1.3.4, III.2.1.5]{Ogus-log}). Set $Q(\nu)'\coloneqq Q(\nu)^\gp\times_{Q^\gp}Q=(\Delta(\nu)^\gp)^{-1}(Q)\subset Q(\nu)^\gp$. Then we have the following commutative diagram, in which the vertical maps are \'etale and the right square is Cartesian by \cite[Prop.~III.2.3.5]{Ogus-log}:
\[
\xymatrix{
& (Z^{\nu+1})'\ar[r]\ar[d] & Z^{\nu+1}\ar[d] \\
Z\ar@{^{(}->}[ur]\ar[r] & \Sigma\times_{\Spec \Z[P]}\Spec \Z[Q(\nu)'] \ar[r]&\Sigma\times_{\Spec \Z[P]}\Spec \Z[Q(\nu)].
}
\]
If $\nu=1$, then $Q^\gp\oplus_{P^\gp}Q^\gp\rightarrow Q^\gp\oplus Q^\gp/P^\gp$ sending $[(q_1,q_2)]\mapsto (q_1+q_2,[q_1])$ induces an isomorphism $Q(1)'\xrightarrow{\cong}Q\oplus Q^\gp/P^\gp$ by \cite[Prop.~I.4.2.19]{Ogus-log}. 
Similarly, one can show $Q(\nu)'\cong Q\oplus Q^\gp/P^\gp\oplus\cdots\oplus Q^\gp/P^\gp$. For $\nu$ and $\nu'$, the projections from $Z^{\nu+\nu'+1}$ to $Z^{\nu+1}$ and $Z^{\nu'+1}$ induce, by universality, an isomorphism $\calD_{Z}(\nu)\otimes_{\calO_Z}\calD_Z(\nu')\xrightarrow{\cong}\calD_{Z}(\nu+\nu')$: this can be checked \'etale locally, and the above description shows that the morphism $(Z^{\nu+\nu'+1})'\rightarrow(Z^{\nu+1})'\times_Z(Z^{\nu'+1})'$ induced similarly is an isomorphism. So the assertion follows from \cite[Lem.~II.1.3.5]{Berthelot-book}.

\begin{eg}\label{eg:basic properties of PD-envelopes}
Let us explicitly describe $\calD_Z(\nu)$ \'etale locally. For simplicity, we treat the case $\nu=1$ but the general case is similar. Since each projection $((Z^2)',M_{(Z^{2})'})\rightarrow (Z,M_Z)$ is strict and smooth, each $(Z^{2})'\rightarrow Z$ is a smooth morphism of schemes by \cite[Prop.~3.8]{Kato-log}.
Take a geometric point $\overline{z}\rightarrow Z$ and continue to write $\overline{z}$ for its images in $(Z^2)'$ and $D_Z(1)$. Take $t_1,\ldots,t_r\in M_{Z,\overline{z}}$ such that $(d\operatorname{log}(t_l))_{1\leq l\leq r}$ is a basis of $\omega^1_{Z/\Sigma,\overline{z}}$. For $1\leq l\leq r$, let $u_{l}$ denote the element of $\Ker(\calO_{(Z^2)',\overline{z}}^\ast\rightarrow \calO_{Z,\overline{z}}^\ast)\subset M_{(Z^2)',\overline{z}}$ defined by $p_2^\ast(t_l)=p_1^\ast(t_l)u_{l}$. Then the preceding discussion shows that $u_1-1,\ldots, u_r-1$ form a smooth coordinate of $p_1\colon (Z^2)'\rightarrow Z$ at $\overline{z}$ and their restrictions along $Z\hookrightarrow (Z^2)'$ are zero. Hence the morphism $\calO_{Z,\overline{z}}[T_1,\ldots, T_r]\rightarrow \calO_{(Z^2)',\overline{z}}$ given by $T_l\mapsto u_l-1$ is a well-defined isomorphism. This together with \cite[Cor.~I.2.5.3]{Berthelot-book} shows that the morphism $\calO_{Z,\overline{z}}\{T_1,\ldots, T_r\}_\mathrm{PD}\rightarrow \calD_Z(1)_{\overline{z}}$ from the PD-polynomial algebra given by $T_{l}^{[n]}\mapsto (u_{l}-1)^{[n]}$ is a well-defined isomorphism.
\end{eg}

\begin{lem}\label{lem:basic properties of PD-envelopes}
The following properties hold.
\begin{enumerate}
\item Let $(Y,M_Y)\hookrightarrow (D_Y(D_Z(\nu)),M_{D_Y(D_Z(\nu))})$ denote the PD-envelope (relative to $\Sigma^\sharp$) of the closed immersion $(Y,M_Y)\hookrightarrow (Z,M_{Z})\hookrightarrow (D_Z(\nu),M_{D_Z(\nu)})$ of log $\Sigma^\sharp$-schemes. Then $(D_Y(D_Z(\nu)),M_{D_Y(D_Z(\nu))})$ and $(D_Y(Z^{\nu+1}),M_{D_Y(Z^{\nu+1})})$ are isomorphic as log PD-schemes over $\Sigma^\sharp$. 
\item As $\calO_Z$-modules, 
the canonical map $\calD\otimes_{\calO_Z}p_{i,\ast}\calD_Z(\nu)\rightarrow p_{i,\ast}\calD_{Y}(Z^{\nu+1})$ is an isomorphism.
\item The projections induce an isomorphism $\calD_{Y}(Z^{\nu+1})\otimes_{\calD}\calD_Y(Z^{\nu'+1})\xrightarrow{\cong} \calD_{Y}(Z^{\nu+\nu'+1})$.
\end{enumerate}
 \end{lem}

\begin{proof}
Part (1) follows from the universal properties.
For (2), the universality gives morphisms $D_{Y}(Z^{\nu+1})\rightarrow D$ and $D_{Y}(Z^{\nu+1})\rightarrow D_Z(\nu)$, which yield the canonical map. To check that it is an isomorphism, we may work \'etale locally. The case $\nu=1$ follows from \cite[Prop.~6.5]{Kato-log}, and the general case is proved similarly. 
Part (3) follows from (2) and $\calD_{Z}(\nu)\otimes_{\calO_Z}\calD_Z(\nu')\cong \calD_{Z}(\nu+\nu')$ as in \cite[IV.1.6.2]{Berthelot-book}.
\end{proof}

We turn to the \v{C}ech--Alexander method.
Let $((Y,M_Y)/\Sigma^\sharp)_\CRIS^{Z\strat}\subset((Y,M_Y)/\Sigma^\sharp)_\CRIS$ denote the full subcategory consisting of objects $(U,T)$ such that there exists a morphism $h\colon (T,M_T)\rightarrow (Z,M_Z)$ of log $\Sigma^\sharp$-schemes making the following diagram commutative
\begin{equation}\label{eq:def of stratifying site}
\xymatrix{
(U,M_{U}) \ar@{^{(}->}[r]\ar[d] & (T,M_T)\ar@{-->}[d]^-h \\
(Y,M_{Y}) \ar@{^{(}->}[r]^-{i} & (Z,M_Z).
}
\end{equation}
Obviously, $((Y,M_Y)/\Sigma^\sharp)_\CRIS^{Z\strat}$ is a site with induced topology, called the \emph{$Z$-HPD stratifying site} (cf.~\cite[Def.~III.1.2.1]{Berthelot-book}).

\begin{lem}\label{lem:weakly final object using smooth lift in crystalline site}
\hfill
\begin{enumerate}
 \item For every $(U,T)\in ((Y,M_Y)/\Sigma^\sharp)_\CRIS^\aff$, $\Hom_{i\text{-}\mathrm{PD}}((U,T),(Z,Z))$ is non-empty.
 \item The inclusion $((Y,M_Y)/\Sigma^\sharp)_\CRIS^{Z\strat}\subset((Y,M_Y)/\Sigma^\sharp)_\CRIS$ is a special cocontinuous functor and induces an equivalence of topoi 
\[
\Sh(((Y,M_Y)/\Sigma^\sharp)_{\CRIS}^{Z\strat})\xrightarrow{\cong} \Sh(((Y,M_Y)/\Sigma^\sharp)_\CRIS).
\]
 \item We have $\Hom_{((Y,M_Y)/\Sigma^\sharp)_\CRIS}((U,T),(Y,D))=\Hom_{i\text{-}\mathrm{PD}}((U,T),(Z,Z))$, 
 and the sheaf $\widetilde{D}$ represented by $(Y,D)\in ((Y,M_Y)/\Sigma^\sharp)_\CRIS$ covers the final object of $\Sh(((Y,M_Y)/\Sigma^\sharp)_\CRIS)$.
 \item The $(\nu+1)$st self-product sheaf $\widetilde{D}^{\nu+1}$ is represented by $(Y,D_Y(Z^{\nu+1}))\in ((Y,M_Y)/\Sigma^\sharp)_\CRIS$
 \end{enumerate}
\end{lem}

\begin{proof}
The smoothness assumption and \cite[Prop.~IV.3.1.4.2]{Ogus-log} give (1). We note that the statement in \textit{loc.~cit.} assumes that the log thickening $U\hookrightarrow T$ is given by a nilpotent ideal. However, this can be relaxed to a nil ideal by a standard limit argument similar to \cite[07K4]{stacks-project}.
Hence $((Y,M_Y)/\Sigma^\sharp)_{\CRIS}^{\aff}\subset ((Y,M_Y)/\Sigma^\sharp)_{\CRIS}^{Z\strat}$, and (2) follows easily.
Finally, (3) follows from Proposition~\ref{prop:crystalline pullback along closed immersion} and (1), and part (4) is proved similarly.
\end{proof}

\begin{construction}
The simplicial log $\Sigma^\sharp$-scheme $(Z^{\bullet},M_{Z^{\bullet}})$ yields a simplicial object $(Y,D_Y(Z^\bullet))$ of $((Y,M_Y)/\Sigma^\sharp)_\CRIS$, and the latter represents the simplicial sheaf $\widetilde{D}^\bullet$ by Lemma~\ref{lem:weakly final object using smooth lift in crystalline site}(4).

Let $\calF$ be a sheaf on $((Y,M_Y)/\Sigma^\sharp)_\CRIS$. The evaluations $\calF_{(Y,D_Y(Z^{\nu+1}))}\in \Sh(Y_\et)=\Sh(D_Y(Z^{\nu+1}))$ form a cosimplicial sheaf $\calF_{(Y,D_Y(Z^{\bullet+1}))}$ on $Y_\et$. When $\calF$ is a sheaf of abelian groups, we write $\CA_Z^\bullet(\calF)$ for the associated normalized simple complex of abelian sheaves. Moreover, this construction is functorial in $\calF$, and if $\calF$ is an $\calA$-module for a sheaf $\calA$ of rings, then $\CA_Z^\bullet(\calF)$ is a complex of $u_{Y/\Sigma,\ast}\calA$-modules on $Y_\et$.
\end{construction}

\begin{thm}\label{thm:CA method in crystalline site}
With the notation and assumption as above, let $\calF^\bullet$ be a bounded below complex of $\calA$-modules on $((Y,M_Y)/\Sigma^\sharp)_\CRIS$. Then there is a canonical isomorphism
\[
Ru_{Y/\Sigma,\ast}\calF^\bullet\xrightarrow{\cong}\CA_Z^\bullet(\calF^\bullet)
\]
in $D^+(Y_\et,u_{Y/\Sigma,\ast}\calA)$, where $\CA_Z^\bullet(\calF^\bullet)$ denotes the simple complex associated to the bicomplex.
\end{thm}

\begin{proof}
The arguments in \cite[Thm.~V.1.2.5, Pf.]{Berthelot-book} works in our setting by virtue of Lemma~\ref{lem:weakly final object using smooth lift in crystalline site}.
\end{proof}

Next we will explain that crystals can be understood in terms of modules with logarithmic connections. 
Recall that an \emph{HPD stratification} on an $\calO_Z$-module $N$ is a $\calD_Z(1)$-linear isomorphism
\[
\epsilon\colon p_2^\ast N\xrightarrow{\cong}p_1^\ast N,
\quad \text{or equivalently,} \quad
\epsilon\colon \calD_Z(1)\otimes_{\calO_Z}N\xrightarrow{\cong}N\otimes_{\calO_Z}\calD_Z(1),
\]
that satisfies the cocycle condition over $D_Z(2)$. Note that the
pullback of $\epsilon$ along the diagonal $Z\rightarrow D_Z(1)$ is necessarily $\id_N$, which can be
seen by pulling back the cocycle condition along $Z\rightarrow D_Z(2)$. Combining this remark with Lemma~\ref{lem:basic properties of PD-envelopes}(1), one obtains a log connection on $N$
\[
\nabla\colon N\rightarrow N\otimes_{\calO_Z}\omega^1_{Z/\Sigma}
\]
given by $\epsilon(1\otimes n)\equiv n\otimes 1+\nabla(n) \pmod{H^{[2]}}$ for $n\in N$ where $H= \Ker(\calD_Z(1)\rightarrow\calO_Z)$.
Moreover, one can show that $\nabla$ is integrable and quasi-nilpotent\footnote{See \cite[Thm.~6.2(iii)]{Kato-log} and \cite[p.~19]{ogus-griffiths} for the definition.} (see Proposition~\ref{prop: crystals and quasi-nilpotent connections} below).

For $\calO_Z$-modules $N$ and $N'$, an \emph{HPD differential operator} from $N$ to $N'$ is a left $\calO_Z$-linear map $u\colon \calD_Z(1)\otimes_{\calO_Z}N\rightarrow N'$. We say that $u$ is of order $\leq n$ if $u$ factors through $\calD_Z(1)\otimes_{\calO_Z}N\rightarrow \calD_Z(1)/H^{[n+1]}\otimes_{\calO_Z}N$.
If
$u'\colon \calD_Z(1)\otimes_{\calO_Z}N'\rightarrow N''$ is another HPD differential operator, define the composition $u'\circ u$ by
\[
u'\circ u\colon \calD_Z(1)\otimes N \xrightarrow{\delta\otimes \id}  \calD_Z(1)\otimes\calD_Z(1)\otimes N\xrightarrow{\id\otimes u} \calD_Z(1)\otimes N'\xrightarrow{u'} N'', 
\]
where $\delta\colon \calD_Z(1) \rightarrow  \calD_Z(2)\cong\calD_Z(1)\otimes_{\calO_Z}\calD_Z(1)$ is induced from $p_{13}\colon D_Z(2)\rightarrow D_Z(1)$.  
If $N$ and $N'$ are $\calD$-modules, then $\calD_Z(1)\otimes_{\calO_Z}N=\calD_Y(Z^2)\otimes_{\calD}N$ by Lemma~\ref{lem:basic properties of PD-envelopes}(2), and thus an HPD differential operator from $N$ to $N'$ is nothing but a left $\calD$-linear map $\calD_Y(Z^2)\otimes_{\calD}N\rightarrow N'$, which is treated in \cite[Def.~1.1.3]{ogus-griffiths}.

\begin{eg}
The identity on $\calD_Z(1)$ is an HPD stratification on $\calO_Z$ and the associated connection is the differential $d\colon \calO_Z\rightarrow \omega_{Z/\Sigma}^1$. Working \'etale locally, let us now use the notation in Example~\ref{eg:basic properties of PD-envelopes}. For $\underline{n}=(n_1,\ldots,n_r)\in\N^r$, set $\lvert \underline{n}\rvert=n_1+\cdots+n_r$ and $T^{[\underline{n}]}=T_1^{[n_1]}\cdots T_r^{[n_r]}$. The left $\calO_Z$-linear map $\partial_{\underline{n}}\colon \calD_Z(1)=\calO_Z\{T_1,\ldots,T_r\}_{\mathrm{PD}}\rightarrow\calO_Z$ defined by $\partial_{\underline{n}}(T^{[\underline{n'}]})=\delta_{\underline{n},\underline{n'}}$ (Kronecker's delta) is an HPD differential operator of order $\lvert \underline{n}\rvert$. If we write $e_1,\ldots,e_r$ for the standard basis of $\N^r$, we have
\begin{equation}\label{eq:formulate for higher HPD differential operator}
\partial_{\underline{n}}=\prod_{1\leq l\leq r}\prod_{0\leq m_l< n_l}(\partial_{e_l}-m_l),
\end{equation}
where the right-hand side is the composite of HPD differential operators (which is independent of order). This can be deduced from the formulae $\delta(T_l)=T_l\otimes T_l+1\otimes T_l+T_l\otimes 1$ and 
\begin{equation}\label{eq:formula for delta}
\delta(T_l^{[n_l]})
=\sum_{n+n'+n''=n_l} n''!
\begin{pmatrix}
n+n''\\ n''
\end{pmatrix}
\begin{pmatrix}
n'+n''\\ n''
\end{pmatrix}
T_l^{[n+n'']}\otimes T_l^{[n'+n'']}.
\end{equation}
See \cite[Lem.~1.1.5]{Ogus-log} (in the case $Y=Z$) for the proof of \eqref{eq:formulate for higher HPD differential operator} and \eqref{eq:formula for delta}.\footnote{There is a typo in \textit{loc.~cit.}: see \cite[Rem.~3.2.3]{Ogus-Higgs}.}

Lemma~\ref{lem:basic properties of PD-envelopes}(2) gives a $\calD_Z(1)$-linear isomorphism $\epsilon_\calD\colon \calD_Z(1)\otimes\calD\xrightarrow{\cong}\calD\otimes\calD_Z(1)$. One can show from the proof that $\epsilon_\calD$ is an HPD stratification on $\calD$ and that  
the map $\calO_Z\rightarrow \calD$ is horizontal with respect to these HPD stratifications.
\end{eg}

\begin{prop} \label{prop: crystals and quasi-nilpotent connections}
With the assumption as in Set-up~\ref{set-up:description of crystals in terms of connections}, there are equivalences of categories among
\begin{enumerate}
 \item[(i)] the category of quasi-coherent $\calO_{Y/\Sigma}$-modules on $((Y,M_Y)/\Sigma^\sharp)_\CRIS$;
 \item[(ii)] the category of quasi-coherent $\calD$-modules together with a quasi-nilpotent integrable log connection (as an $\calO_Z$-module) compatible with $\nabla\colon \calD\rightarrow\calD\otimes\omega_{Z/\Sigma}^1$;
 \item[(iii)] the category of quasi-coherent $\calD$-modules with an HPD stratification (as an $\calO_Z$-module) compatible with the HPD stratification on $\calD$ (see below)
\end{enumerate}
such that the quasi-coherent $\calD$-module $E$ with an HPD stratification $\epsilon$ attached to a quasi-coherent $\calO_{Y/\Sigma}$-module $\calE$ satisfies $E=\calE_{(Y,D)}$ and $\epsilon\colon p_2^\ast E\cong 
\calE_{(Y,D_Z(1))}\cong p_1^\ast E$.
\end{prop}

Here the compatibility in (iii) means that an HPD stratification on a $\calD$-module $N$
\[
\epsilon\colon (\calD_Z(1)\otimes_{\calO_Z}\calD)\otimes_\calD N=\calD_Z(1)\otimes_{\calO_Z}N\xrightarrow{\cong}N\otimes_{\calO_Z}\calD_Z(1)=N\otimes_\calD(\calD\otimes_{\calO_Z}\calD_Z(1))
\]
satisfies $\epsilon((a\otimes b)\otimes n)=\epsilon_\calD(a\otimes b)\epsilon((1\otimes 1)\otimes n)$ for $a\otimes b\in \calD_Z(1)\otimes_{\calO_Z}\calD$ and $n\in N$.

Recall also from Proposition~\ref{prop:quasi-coherent crystals}(2) that quasi-coherent $\calO_{Y/\Sigma}$-modules are exactly quasi-coherent crystals of $\calO_{Y/\Sigma}$-modules.

\begin{proof}
The proof of \cite[Thm.~6.2]{Kato-log} and \cite[Thm.~1.1.8]{ogus-griffiths} for the small crystalline site also works in our set-up thanks to Lemma~\ref{lem:weakly final object using smooth lift in crystalline site}. Note that the quasi-nilpotence is a consequence of \eqref{eq:formulate for higher HPD differential operator}.
\end{proof}

\begin{rem}\label{rem: crystals and quasi-nilpotent connections}
Keep the notation as in Proposition~\ref{prop: crystals and quasi-nilpotent connections}.
\begin{enumerate}
 \item Contrary to the convention in \cite{Berthelot-book,berthelot-ogus-book}, we usually use the calligraphic font $\calE$ to denote a quasi-coherent $\calO_{Y/\Sigma}$-module and the roman font $E$ to denote the corresponding quasi-coherent $\calD$-module.
 \item Each of the categories (i)-(iii) has a natural notion of tensor products, and the above equivalences are compatible with tensor products.
 \item  We also apply the above proposition to the case $((Z,M_Z)/\Sigma^\sharp)_\CRIS$, in which $\calD$ is replaced by $\calO_Z$ in (ii) and (iii).
 \item Let $E_Z$ be a quasi-coherent $\calO_Z$-module with an HPD stratification and write $\calE_Z$ for the corresponding quasi-coherent $\calO_{Z/\Sigma}$-module on $((Z,M_Z)/\Sigma^\sharp)_\CRIS$. Then the quasi-coherent $\calO_{Y/\Sigma}$-module $\calE\coloneqq i_{\CRIS}^\ast\calE_Z$ corresponds to $\calD\otimes_{\calO_Z}E_Z$ with the tensor product HPD stratification: we compute $\calE_{(Y,D)}=(\calE_Z)_{(Y,D)}=\calD\otimes_{\calO_Z}E_Z$, and the description of the HPD stratification is deduced similarly.
\end{enumerate}
\end{rem}

\begin{eg}\label{eg:description of i-pushforward of crystal}
It is easy to see that the multiplication map $\calD\otimes_{\calO_Z}\calD\rightarrow \calD$ is compatible with HPD stratifications. Hence $\calD$ defines a quasi-coherent sheaf $\calA_i$ of $\calO_{Z/\Sigma}$-algebras on $((Z,M_Z)/\Sigma^\sharp)_\CRIS$. Let $(U,T)\in ((Z,M_Z)/\Sigma^\sharp)_\CRIS^{Z\strat}$. Set $U_Y\coloneqq U\times_ZY$ with the pullback log structure from $M_{Y}$, and let $i_T\colon (U_Y',M_{U_Y'})\hookrightarrow(D_T,M_{D_T})$ denote the PD-envelope of $(U_Y,M_{U_Y})\hookrightarrow (T,M_T)$ relative to $\Sigma^\sharp$. By construction and Proposition~\ref{prop:crystalline pullback along closed immersion}, we have an algebra homomorphism
\[
(\calA_i)_{(U,T)}=\calO_T\otimes_{\calO_Z}\calD\rightarrow \calO_{D_T}=(i_{\CRIS,\ast}\calO_{Y/\Sigma})_{(U,T)},
\]
which defines an $\calO_{Z/\Sigma}$-algebra sheaf morphism $\calA_i\rightarrow i_{\CRIS,\ast}\calO_{Y/\Sigma}$. If $U\rightarrow Z$ is \'etale, then we see that $D_T\rightarrow D\times_ZU$ is an isomorphism by checking the universal property, and we also have $U_Y'=U_Y$. Hence the above algebra homomorphism is an isomorphism for such $(U,T)$.

Let $\calE$ be a quasi-coherent $\calO_{Y/\Sigma}$-module. Then $i_{\CRIS,\ast}\calE$ is a quasi-coherent crystal of $i_{\CRIS,\ast}\calO_{Y/\Sigma}$-modules.\footnote{However, it is not a quasi-coherent $\calO_{Z/\Sigma}$-module since we work on the big site. See also \cite[\S~2.1]{deJong-dieudonnemodule}.}
Let $(E,\epsilon)$ be the corresponding quasi-coherent $\calD$-module with an HPD stratification compatible with $(\calD,\epsilon_\calD)$. Regarding $E$ as an $\calO_Z$-module defines a quasi-coherent $\calO_{Z/\Sigma}$-module $\calE_{\calA_i}$ together with an $\calA_i$-module structure such that $(\calE_{\calA_i})_{(Z,Z)}=E$, and it comes with a canonical isomorphism $\calE_{\calA_i}\otimes_{\calA_i}i_{\CRIS,\ast}\calO_{Y/\Sigma}\xrightarrow{\cong}i_{\CRIS,\ast}\calE$.  Finally, the map $\calE_{\calA_i}\rightarrow i_{\CRIS,\ast}\calE$ becomes an isomorphism after restricting it to the small crystalline site, namely, $r_{Z/\Sigma}^\ast\calE_{\calA_i}\xrightarrow{\cong} r_{Z/\Sigma}^\ast i_{\CRIS,\ast}\calE$. In particular, we deduce from Remark~\ref{rem:u-pushforward depends only on restriction to small site} a quasi-isomorphism $Ru_{Z/\Sigma,\ast}\calE_{\calA_i}\xrightarrow{\cong}Ru_{Z/\Sigma,\ast}i_{\CRIS,\ast}\calE$.
\end{eg}

Let us now discuss the linearization and the log crystalline Poincar\'e lemma. 
For an $\calO_Z$-module $N$, define a left $\calO_Z$-module $L(N)_Z$ by
\[
L(N)_Z\coloneqq \calD_Z(1)\otimes_{\calO_Z}N.
\]
For an HPD differential operator $u\colon \calD_Z(1)\otimes_{\calO_Z}N\rightarrow N'$, define an $\calO_Z$-linear map $L(u)_Z\colon L(N)_Z\rightarrow L(N')_Z$ by
\[
L(u)_Z\colon D_Z(1)\otimes N\xrightarrow{\delta\otimes\id}D_Z(1)\otimes D_Z(1)\otimes N\xrightarrow{\id\otimes u}D_Z(1)\otimes N'.
\]
Exactly as in \cite[Lem.~IV.3.1.2]{Berthelot-book}, the construction $L(-)_Z$ defines a functor from the category of (quasi-coherent) $\calO_Z$-modules with HPD differential operators to the category of (quasi-coherent) $\calO_Z$-modules with an HPD stratification whose morphisms are $\calO_Z$-linear maps compatible with HPD stratifications: the HPD stratification on $L(N)_Z$ is defined as
\[
\calD_Z(1)\otimes_{\calO_Z} (\calD_Z(1)\otimes_{\calO_Z} N) \xrightarrow{\alpha} \calD_Z(1)\otimes_{\calO_Z} \calD_Z(1)\otimes_{\calO_Z} \calD_Z(1)\otimes_{\calO_Z} N \xrightarrow{\beta} (\calD_Z(1)\otimes_{\calO_Z} N) \otimes_{\calO_Z} \calD_Z(1)
\]
where $\alpha=\id\otimes \delta\otimes\id$ and $\beta$ is the map induced by universality from the map 
\[
(\calO_Z\otimes\calO_Z)\otimes_{\calO_Z} (\calO_Z\otimes\calO_Z)\otimes_{\calO_Z} ((\calO_Z\otimes\calO_Z)\otimes_{\calO_Z} N) \xrightarrow{\beta'} ((\calO_Z\otimes\calO_Z)\otimes_{\calO_Z} N) \otimes_{\calO_Z} (\calO_Z\otimes\calO_Z)
\]
given by $\beta'((a_1\otimes a_2)\otimes (b_1\otimes b_2)\otimes ((c_1\otimes c_2)\otimes n))=((c_1\otimes c_2)\otimes n)\otimes(a_1b_2\otimes a_2b_1)$ with the left-most $\calO_Z$-module structure on $((\calO_Z\otimes\calO_Z)\otimes_{\calO_Z} N)$ when taking tensor products (cf.~\cite[Prop.~6.15, Pf.]{berthelot-ogus-book}).

\begin{eg}\label{eg:linearization of log de Rham complex}
The log de Rham complex $(\omega_{Z/\Sigma}^\bullet,d)$ is a complex of quasi-coherent $\calO_Z$-modules and HPD differential operators. We describe $L(d)_Z\colon \calD_Z(1)\otimes\omega_{Z/\Sigma}^i\rightarrow \calD_Z(1)\otimes\omega_{Z/\Sigma}^{i+1}$ as follows. Keep the notation as in Example~\ref{eg:basic properties of PD-envelopes}. Replacing $Z$ by an \'etale neighborhood of $\overline{z}$, we assume that $(u_l-1)_{1\leq l\leq r}$ is a smooth coordinate of $(Z^2)'\rightarrow Z$ and $\calO_Z\{T_1,\ldots,T_r\}_\mathrm{PD}\cong \calD_Z(1)$ via $T_l^{[n]}\mapsto (u_l-1)^{[n]}$. 

We claim that for local sections $a\in\calO_Z$ and $\omega\in\omega_{Z/\Sigma}^i$,
\begin{equation}\label{eq:linearization of differential}
\begin{aligned}
&L(d)_Z(aT_1^{[n_1]}\cdots T_r^{[n_r]}\otimes\omega)\\
&=a\bigl(\sum_{l=1}^r T_1^{[n_1]}\cdots T_l^{[n_l-1]}\cdots T_r^{[n_r]}(T_l+1)\,d\operatorname{log}(t_l)\wedge \omega+T_1^{[n_1]}\cdots T_r^{[n_r]}\otimes d\omega \bigr).  
\end{aligned}
\end{equation}
To see this, observe that the HPD differential operator $\widetilde{d}\colon \calD_Z(1)\otimes \omega_{Z/\Sigma}^i\rightarrow \omega_{Z/\Sigma}^{i+1}$ corresponding to $d$ factors through $\calD_Z(1)\otimes\omega_{Z/\Sigma}^i\rightarrow \calD_Z(1)/H^{[2]}\otimes\omega_{Z/\Sigma}^i$ where $H=\Ker(\calD_Z(1)\rightarrow\calO_Z)$, and it satisfies $\widetilde{d}(a\otimes\omega)=a\otimes d\omega$ and $\widetilde{d}(aT_l\otimes\omega)=a\,d\operatorname{log}T_l\wedge \omega$.
Moreover, \eqref{eq:formula for delta} gives
\[
\delta(T_l^{[n_l]})
\equiv 
T_l^{[n_l]}\otimes 1+ T_l^{[n_l-1]}\otimes T_l+n_lT_l^{[n_l]}\otimes T_l \pmod{\calD_Z(1)\otimes H^{[2]}}.
\]
Now \eqref{eq:linearization of differential} is deduced from the definition $L(d)_Z=(\id\otimes d)\circ (\delta\otimes\id)$ and the fact that $\delta$ is a PD-algebra map (cf.~\cite[Lem.~IV.3.2.5, Pf]{Berthelot-book}, \cite[Lem.~6.11, Pf.]{berthelot-ogus-book}, or \cite[(2.2.7.2)]{nakkajima-shiho}).

A direction computation using \eqref{eq:linearization of differential} shows that $(L(\omega_{Z/\Sigma}^\bullet)_Z, L(d)_Z)$ is a complex of quasi-coherent $\calO_Z$-modules with an HPD stratification, which is a priori not obvious from the definition of the composition of two HPD differential operators.

Finally, $p_1^\ast\colon \calO_Z\rightarrow L(\calO_Z)_Z=\calD_Z(1)$ is horizontal with respect to the canonical HPD stratifications on them: this is proved exactly as in \cite[Thm.~6.12, Pf.]{berthelot-ogus-book}. Moreover, we know from \eqref{eq:linearization of differential} that the composite $\calO_Z\rightarrow L(\calO_Z)_Z\rightarrow L(\omega_{Z/\Sigma}^1)_Z$ is zero.
\end{eg}

\begin{lem}\label{lem:linearization and tensor product}
Let $(E_Z,\epsilon)$ be an $\calO_Z$-module with an HPD stratification.
For any $\calO_Z$-module $N$, the $\calD_Z(1)$-linear isomorphism
 \[
 \epsilon\otimes\id_N\colon L(E_Z\otimes N)_Z=\calD_Z(1)\otimes E_Z\otimes N\xrightarrow{\cong} E_Z\otimes \calD_Z(1)\otimes N=E_Z\otimes L(N)_Z
 \]
 is compatible with HPD stratifications (where the target is equipped with the tensor product one).
Moreover, if $N'$ is another $\calO_Z$-module and $u$ is an HPD differential operator from $N$ to $N'$, then the induced map $\id\otimes L(u)\colon E_Z\otimes L(N)_Z\rightarrow E_Z\otimes L(N')_Z$ is compatible with the map $L(E_Z\otimes N)_Z\rightarrow L(E_Z\otimes N')_Z$ induced from an HPD differential operator from $E_Z\otimes_{\calO_Z}N$ to $E_Z\otimes_{\calO_Z}N'$ given by $\epsilon$ and $u$ (defined exactly as in \cite[Prop.~IV.3.1.4(ii)]{Berthelot-book}).
\end{lem}

\begin{proof}
The proof of \cite[Prop.~IV.3.1.4]{Berthelot-book} works in our set-up.
\end{proof}

\begin{construction}
For a quasi-coherent $\calO_Z$-module $N$, let $\calL(N)_Z$ denote the quasi-coherent $\calO_{Z/\Sigma}$-module on $((Z,M_Z)/\Sigma^\sharp)_\CRIS$ corresponding to the quasi-coherent $\calO_Z$-module $L(N)_Z$ with the canonical HPD stratification via Proposition~\ref{prop: crystals and quasi-nilpotent connections}. We set
\[
\calL(N)\coloneqq \calL_{i}(N)\coloneqq i_\CRIS^\ast \calL(N)_Z.
\]
Since $i_\CRIS^\ast$ is exact, $\calL(N)$ is a quasi-coherent $\calO_{Y/\Sigma}$-module (namely, a quasi-coherent crystal of $\calO_{Y/\Sigma}$-modules) on $((Y,M_Y)/\Sigma^\sharp)_\CRIS$.
By unwinding the constructions, one can check that for $(U,T)\in ((Y,M_Y)/\Sigma^\sharp)_\CRIS^{Z\strat}$, the quasi-coherent $\calO_T$-module $\calL(N)_{(U,T)}$ on $T_\et$ is given by $h_\et^\ast L(N)_Z=h_\et^\ast (\calD_Z(1)\otimes_{\calO_Z}N)$ for any $h\colon (T,M_T)\rightarrow (Z,M_Z)$ fitting in \eqref{eq:def of stratifying site} (and is independent of $h$ up to the canonical identification given by the HPD stratification on $L(N)_Z$).
\end{construction}

Example~\ref{eg:linearization of log de Rham complex} gives a complex of quasi-coherent $\calO_{Y/\Sigma}$-modules
\begin{equation}\label{eq:log crystalline Poincare lemma}
\calO_{Y/\Sigma}=i_\CRIS^\ast \calO_{Z/\Sigma}\rightarrow \calL(\calO_Z)\rightarrow \calL(\omega_{Z/\Sigma}^1)\rightarrow\cdots.
\end{equation}

\begin{thm}[Log crystalline Poincar\'e Lemma]\label{thm:log crystalline Poincare lemma}
For every $\calO_{Y/\Sigma}$-module $\calM$, \eqref{eq:log crystalline Poincare lemma} induces a quasi-isomorphism
\[
\calM\rightarrow \calM\otimes_{\calO_{Y/\Sigma}}\calL(\omega_{Z/\Sigma}^\bullet).
\]
\end{thm}

\begin{proof}
Take $(U,T)\in ((Y,M_Y)/\Sigma^\sharp)_\CRIS^{Z\strat}$. Then $\calL(\omega_{Z/\Sigma}^i)_{(U,T)}=\calO_T\otimes_{\calO_Z}\calD_Z(1)\otimes_{\calO_Z}\omega_{Z/\Sigma}^i$. Here $\omega_{Z/\Sigma}^i$ is a locally free $\calO_Z$-module, and $\calD_Z(1)$ is a flat $\calO_Z$ by the explicit description in Example~\ref{eg:basic properties of PD-envelopes}. Hence $\calL(\omega_{Z/\Sigma}^i)_{(U,T)}$ is flat over $\calO_Z$. We know from Proposition~\ref{prop:quasi-coherent crystals}(1) that $\calL(\omega_{Z/\Sigma}^\bullet)$ is a complex of flat $\calO_{Y/\Sigma}$-modules. So it suffices to prove the case $\calM=\calO_{Y/\Sigma}$, namely, 
that $\calO_{Y/\Sigma}\rightarrow \calL(\omega_{Z/\Sigma}^\bullet)$ is a quasi-isomorphism, for which we will follow the last two paragraphs of \cite[Prop.~2.2.7, Pf.]{nakkajima-shiho}. We need to show for each affine $(U,T=\Spec A)$, the evaluation $\calO_T\rightarrow \calL(\omega_{Z/\Sigma}^\bullet)_{(U,T)}$ is a quasi-isomorphism. \'Etale locally, this is associated to
\[
A\rightarrow K^\bullet(T_1,\ldots,T_r)\coloneqq\bigl(A\{T_1,\ldots,T_r\}_{\mathrm{PD}}\otimes_A \bigwedge^\bullet \bigoplus_{l=1}^rA\,d\operatorname{log}t_l,d\bigr)
\]
with differential $d$ given by the same formula as \eqref{eq:linearization of differential}. Observe that $K^\bullet(T_1,\ldots,T_r)$ is a complex of flat $A$-modules and there is a natural isomorphism $K^\bullet(T_1,\ldots,T_r)\cong K^\bullet(T_1,\ldots,T_{r-1})\otimes_A K^\bullet(T_r)$ of complexes. Hence to show that $A\rightarrow K^\bullet(T_1,\ldots,T_r)$ is a quasi-isomorphism, we may assume $r=1$, in which the assertion is equivalent to the exactness of 
\begin{equation}\label{eq:proof in log crystalline Poincare lemma}
0\rightarrow A\rightarrow A\{T\}_{\mathrm{PD}}\xrightarrow{d}A\{T\}_{\mathrm{PD}}\rightarrow 0
\end{equation}
with $d(aT^{[n]})=aT^{[n-1]}(T+1)=anT^{[n]}+aT^{[n-1]}$. Since $p$ is nilpotent in $A$, the sequence \eqref{eq:proof in log crystalline Poincare lemma} is exact.
\end{proof}

Next we upgrade Theorem~\ref{thm:log crystalline Poincare lemma} to the filtered one, following \cite[V.2.1.4]{Berthelot-book}. Recall the morphism $i_{Y/\Sigma}\colon \Sh(Y_\et)\rightarrow \Sh(((Y,M_Y)/\Sigma^\sharp)_\CRIS)$ of topoi in Proposition~\ref{prop:morphism of topoi i from et to cris}.
The pullback of the $\calO_Z$-linear map $\calD_Z(1)\rightarrow \calO_Z$ along $V\rightarrow Y\rightarrow Z$ for each $V\in Y_\et$ defines a map $i_{Y/\Sigma}^{-1}\calL(\calO_Z)\rightarrow \calO_Y$, and by adjunction, we obtain a map $\calL(\calO_Z)\rightarrow i_{Y/\Sigma,\ast}\calO_Y$. We can see that this is a surjection of sheaves on $((Y,M_Y)/\Sigma^\sharp)_\CRIS$ and set 
\[
\calK\coloneqq \Ker(\calL(\calO_Z)\rightarrow i_{Y/\Sigma,\ast}\calO_Y).
\]
Recall the $\calO_Z$-module decomposition $\calD_Z(1)=\calO_Z\oplus H$ with the augmented ideal $H=\Ker(\calD_Z(1)\rightarrow \calO_Z)$. For $(U,T)\in ((Y,M_Y)/\Sigma^\sharp)_\CRIS^{Z\strat}$ with any $h\colon (T,M_T)\rightarrow (Y,M_Y)$ as in \eqref{eq:def of stratifying site}, one obtains an $\calO_T$-module decomposition $\calK_{(U,T)}=(\calJ_{Y/\Sigma})_{(U,T)}\oplus h_\et^\ast H$. Hence we deduce from \cite[Prop.~I.1.6.5, I.1.7.1]{Berthelot-book} that there exists a unique PD-ideal sheaf structure on $\calK_{(U,T)}$ such that $(\calO_T,(\calJ_{Y/\Sigma})_{(U,T)})\rightarrow (\calL(\calO_Z)_{(U,T)},\calK_{(U,T)})\leftarrow(\calD_Z(1),H)$ are PD-morphisms. Using the HPD-stratification on $\calD_Z(1)$, one can check, exactly as in \cite[pp.~298-299]{Berthelot-book}, that this PD-ideal structure on $\calK_{(U,T)}$ is independent of $h$.
The $\calK$-PD-adic filtration $\calK^{[n]}$ of $\calL(\calO_Z)$ induces a filtered map $(\calO_{Y/\Sigma},\calJ_{Y/\Sigma}^{[n]})\rightarrow (\calL(\calO_Z),\calK^{[n]})$ and, by \eqref{eq:linearization of differential}, satisfies
\[
\calL(d)\bigl(\calK^{[n]}\calL(\omega_{Z/\Sigma}^i)\bigr)\subset \calK^{[n-1]}\calL(\omega_{Z/\Sigma}^{i+1}).
\]

\begin{defn}
For each $n$, define the subcomplex $F^n\calL(\omega_{Z/\Sigma}^\bullet)\subset \calL(\omega_{Z/\Sigma}^\bullet)$ by
\[
F^n\calL(\omega_{Z/\Sigma}^i)\coloneqq \calK^{[n-i]}\calL(\omega_{Z/\Sigma}^i).
\]
More generally, for an $\calO_{Y/\Sigma}$-module $\calM$, define $F^n(\calM\otimes_{\calO_{Y/\Sigma}}\calL(\omega_{Z/\Sigma}^\bullet))$ to be the image of $\calM\otimes_{\calO_{Y/\Sigma}}F^n\calL(\omega_{Z/\Sigma}^\bullet)\rightarrow \calM\otimes_{\calO_{Y/\Sigma}}\calL(\omega_{Z/\Sigma}^\bullet)$.
By construction, we have a morphism of complexes
\begin{equation}\label{eq:filtered log crystalline Poincare lemma}
\calJ_{Y/\Sigma}^{[n]}\calM\rightarrow F^n(\calM\otimes_{\calO_{Y/\Sigma}}\calL(\omega_{Z/\Sigma}^\bullet)).
\end{equation}
\end{defn}

\begin{thm}[Filtered log crystalline Poincar\'e Lemma]\label{thm:filtered log crystalline Poincare lemma}
For every $\calO_{Y/\Sigma}$-module $\calM$ and every $n$, the map in \eqref{eq:filtered log crystalline Poincare lemma} is a quasi-isomorphism.
\end{thm}

\begin{proof}
We follow the proof of \cite[Thm.~V.2.1.4]{Berthelot-book}. With the notation the notation of the proof of Theorem~\ref{thm:log crystalline Poincare lemma}, consider the case $r=1$: \eqref{eq:proof in log crystalline Poincare lemma} yields
\begin{equation}\label{eq:proof in filtered log crystalline Poincare lemma}
0\rightarrow J^{[n]}\rightarrow K^{[n]}\xrightarrow{d} K^{[n-1]}\rightarrow 0,
\end{equation}
where $J^{[n]}$ and $K^{[n]}$ are the $A$-modules corresponding to the quasi-coherent $\calO_T$-modules $(\calJ_{Y/\Sigma}^{[n]})_{(U,T)}$ and $\calK_{(U,T)}^{[n]}$ and $d$ is given by $d(aT^{[m]})=aT^{[m-1]}(T+1)$.
The sequence \eqref{eq:proof in filtered log crystalline Poincare lemma} is exact and split with a section $s$ of $d$ given by 
\[
s(aT^{[m]})=\sum_{l=0}^\infty (-1)^l\frac{(m+l)!}{m!}a T^{m+l+1}, 
\]
which is a finite sum since $p$ is nilpotent in $A$.
Once we know that \eqref{eq:proof in filtered log crystalline Poincare lemma} is split exact, the last paragraph of \cite[p.~300]{Berthelot-book} works verbatim in our situation and completes the proof.
\end{proof}

\begin{prop}\label{prop:CA of linearization}
Let $N$ be a quasi-coherent $\calO_Z$-module. The left $\calD$-linear map
\[
\calD\otimes_{\calO_Z}N\xrightarrow{p_2^\ast\otimes\id} \calD_Y(Z^2)\otimes_{\calO_Z}N\cong\calD\otimes_{\calO_Z}\calD_Z(1)\otimes_{\calO_Z}N=\calL(N)_{(U,D)}=\CA_Z^0(\calL(N)),
\]
regarded as a morphism in $\Sh(Y_\et)=\Sh(D_\et)$,
induces a quasi-isomorphism
\[
(\calJ_{Y/\Sigma}^{[n]})_{(Y,D)}(\calD\otimes_{\calO_Z}N)\xrightarrow{\cong}\CA_Z^\bullet(\calK^{[n]}\calL(N)).
\]
Moreover, if $N'$ is another quasi-coherent $\calO_Z$-module and $u$ is an HPD differential operator from $N$ to $N'$, then the induced map $\CA_Z^\bullet(\calL(u))\colon \CA_Z^\bullet(\calL(N))\rightarrow \CA_Z^\bullet(\calL(N))$ is compatible with the map $\calD\otimes_{\calO_Z}N\rightarrow\calD\otimes_{\calO_Z}N'$ induced from an HPD differential operator from $\calD\otimes_{\calO_Z}N$ to $\calD\otimes_{\calO_Z}N'$ given by $u$ (defined exactly as in \cite[Prop.~V.2.2.2(ii)]{Berthelot-book}).
\end{prop}

\begin{proof}
Recall that $(\calJ_{Y/\Sigma}^{[n]})_{(Y,D)}$ denotes the $n$th divided power of $(\calJ_{Y/\Sigma})_{(Y,D)}=\Ker(\calD\rightarrow \calO_Y)$.
The proof of \cite[Prop.~V.2.2.2]{Berthelot-book} works in our set-up. Note that the first assertion is a simple consequence of cosimplicial complexes.
\end{proof}

\begin{defn}\label{def:cohomology of crystal in terms of log de Rham complex}
Keep the assumption and notation as in Set-up~\ref{set-up:description of crystals in terms of connections}. Let $\calE$ be a quasi-coherent $\calO_{Y/\Sigma}$-module on $((Y,M_Y)/\Sigma^\sharp)_\CRIS$ and write $(E,\nabla)$ for the corresponding quasi-coherent $\calD$-module with log connection. Regard $(E,\nabla)$ as an object in $\Sh(Y_\et)=\Sh(D_\et)$ and let $(E\otimes_{\calO_Z}\omega^\bullet_{Z/\Sigma},\nabla)$ denote the induced log de Rham complex.
Define the subcomplex $F^n(E\otimes_{\calO_Z}\omega^\bullet_{Z/\Sigma},\nabla)$ by 
\[
F^n(E\otimes_{\calO_Z}\omega^i_{Z/\Sigma},\nabla)\coloneqq (\calJ_{Y/\Sigma}^{[n-i]})_{(Y,D)}(E\otimes_{\calO_Z}\omega^i_{Z/\Sigma},\nabla).
\]
\end{defn}

\begin{thm}\label{thm:cohomology of crystal in terms of log de Rham complex filtered}
Keep the assumption and notation as in Definition~\ref{def:cohomology of crystal in terms of log de Rham complex}.
Assume further that the quasi-coherent $\calO_{Y/\Sigma}$-module $\calE$ on $((Y,M_Y)/\Sigma^\sharp)_\CRIS$ satisfies $\calE\cong  i_{\CRIS}^\ast\calE_Z$ for a quasi-coherent $\calO_{Z/\Sigma}$-module $\calE_Z$ on $((Z,M_Z)/\Sigma^\sharp)_\CRIS$. If we let $E_Z$ denote the quasi-coherent $\calO_Z$-module with an HPD stratification corresponding to $\calE_Z$, then $E\cong \calD\otimes_{\calO_Z}E_Z$ with $\nabla$ corresponding to the tensor product HPD stratification on $\calD\otimes_{\calO_Z}E_Z$. Moreover, for each $n\geq 0$,  there are natural quasi-isomorphisms
\[
Ru_{Y/\Sigma,\ast}(\calJ_{Y/\Sigma}^{[n]}\calE)\cong (\calJ_{Y/\Sigma}^{[n-\bullet]})_{(Y,D)}(\calD\otimes_{\calO_Z}E_Z\otimes_{\calO_Z}\omega^\bullet_{Z/\Sigma},\nabla)\cong F^n(E\otimes_{\calO_Z}\omega^\bullet_{Z/\Sigma},\nabla).
\]
 in $D^+(Y_\et,\pi^{-1}_\et\calO_\Sigma)$, where $\pi$ denotes $Y\rightarrow \Sigma$.
\end{thm}

\begin{proof}
The first part follows from Remark~\ref{rem: crystals and quasi-nilpotent connections}(2) and (4). The desired quasi-isomorphism is obtained as
\begin{align*}
Ru_{Y/\Sigma,\ast}(\calJ_{Y/\Sigma}^{[n]}\calE)
&\underset{\cong}{\xrightarrow{\text{Thm.~\ref{thm:filtered log crystalline Poincare lemma}}}} Ru_{Y/\Sigma,\ast}(F^n(\calE\otimes \calL(\omega_{Z/Y}^\bullet)))\\
&\underset{\cong}{\xleftarrow{\text{Lem.~\ref{lem:linearization and tensor product}}}} Ru_{Y/\Sigma,\ast}(\calK^{[n-\bullet]}\calL(E_Z\otimes_{\calO_Z}\omega_{Z/Y}^\bullet))\\
&\underset{\cong}{\xrightarrow{\text{Thm.~\ref{thm:CA method in crystalline site}}}}\CA_Z^\bullet(\calK^{[n-\bullet]}\calL(E_Z\otimes_{\calO_Z}\omega_{Z/Y}^\bullet))\\
&\underset{\cong}{\xleftarrow{\text{Prop.~\ref{prop:CA of linearization}}}}(\calJ_{Y/\Sigma}^{[n-\bullet]})_{(Y,D)}(\calD\otimes_{\calO_Z}E_Z\otimes_{\calO_Z}\omega_{Z/Y}^\bullet,\nabla).
\end{align*}
More precisely, for the second isomorphism, we use 
\begin{align*}
\calK^{[n-\bullet]}\calL(E_Z\otimes&\omega_{Z/Y}^\bullet)=\calK^{[n-\bullet]}i_{\CRIS}^\ast\calL(E_Z\otimes\omega_{Z/Y}^\bullet)_Z\\
&\underset{\cong}{\xrightarrow{\text{Lem.~\ref{lem:linearization and tensor product}}}}\calK^{[n-\bullet]}i_{\CRIS}^\ast(\calE_Z\otimes_{\calO_{Z/\Sigma}}\calL(\omega_{Z/Y}^\bullet)_Z)\cong \calK^{[n-\bullet]}(\calE\otimes_{\calO_{Y/\Sigma}} \calL(\omega_{Z/Y}^\bullet)).
\end{align*}
Here we also need to keep track of differentials, and this can be deduced from the second assertions of Lemma~\ref{lem:linearization and tensor product} and Proposition~\ref{prop:CA of linearization} exactly as in \cite[Cor.~V.2.2.4]{Berthelot-book}.
Finally, observe that $\calK^{[n-\bullet]}(\calE\otimes_{\calO_{Y/\Sigma}} \calL(\omega_{Z/Y}^\bullet))$ is identified with $F^n(\calE\otimes \calL(\omega_{Z/Y}^\bullet))$ by construction.
\end{proof}

\begin{rem}
If $Y$ is defined by a quasi-coherent ideal $\calJ_{\Sigma,0}\calO_Z$ for a sub PD-ideal $\calJ_{\Sigma,0}\subset\calJ_\Sigma$, then $D=Z$, and the assumption on $\calE$ in the above theorem always holds.
\end{rem}

\begin{thm}\label{thm:cohomology of crystal in terms of log de Rham complex}
Keep the assumption and notation as in Set-up~\ref{set-up:description of crystals in terms of connections} and Definition~\ref{def:cohomology of crystal in terms of log de Rham complex}. For every quasi-coherent $\calO_{Y/\Sigma}$-module $\calE$ on $((Y,M_Y)/\Sigma^\sharp)_\CRIS$, there is a natural quasi-isomorphism 
\[
Ru_{Y/\Sigma,\ast}\calE\cong (E\otimes_{\calO_Z}\omega^\bullet_{Z/\Sigma},\nabla)
\]
 in $D^+(Y_\et,\pi^{-1}_\et\calO_\Sigma)$, where $\pi$ denotes $Y\rightarrow \Sigma$.
\end{thm}

\begin{proof}
We use the notation of Example~\ref{eg:description of i-pushforward of crystal}. From Corollary~\ref{cor:exactness of i-pushforward} and Proposition~\ref{prop:compatibility of u and crystalline pushforward}, we deduce quasi-isomorphisms
\[
i_{\et,\ast}Ru_{Y/\Sigma,\ast}\calE\xrightarrow{\cong}Ru_{Z/\Sigma,\ast}i_{\CRIS,\ast}\calE\xleftarrow{\cong}Ru_{Z/\Sigma,\ast}\calE_{\calA_i}.
\]
Recall that $\calE_{\calA_i}$ is a quasi-coherent $\calO_{Z/\Sigma}$-module corresponding to $(E,\nabla)$ (regarded as a quasi-coherent $\calO_Z$-module with log connection). So from Theorem~\ref{thm:cohomology of crystal in terms of log de Rham complex filtered} for $((Z,M_Z)/\Sigma^\sharp)_\CRIS$ and $n=0$, we obtain a quasi-isomorphism $Ru_{Z/\Sigma,\ast}\calE_{\calA_i}\cong (E\otimes_{\calO_Z}\omega^\bullet_{Z/\Sigma},\nabla)$ in $D^+(Z_\et,\pi_{Z,\et}^{-1}\calO_\Sigma)$ where $\pi_Z$ denotes $Z\rightarrow \Sigma$.
The theorem follows from these quasi-isomorphisms by applying $i_\et^{-1}$ since $i_\et^{-1}$ is exact and satisfies $i_{\et}^{-1}\circ i_{\et,\ast}=\id$.
\end{proof}

For a later use, let us check the compatibility of the quasi-isomorphism in the theorem. Suppose that we have the following commutative diagram
\[
\xymatrix{
&(Z',M_{Z'})\ar[rr]^-(.4){f_Z} \ar[dd]^-(.7){\pi_{Z'}}|\hole && (Z,M_Z)\ar[dd]^-(.7){\pi_Z}\\
(Y',M_{Y'})\ar@{^{(}->}[ur]^{i'}\ar[rr]^-(.4)f \ar[rd]_-{\pi'}&& (Y,M_Y)\ar[rd]_-\pi\ar@{^{(}->}[ur]^{i}&\\
&\Sigma' \ar[rr]^-\phi&& \Sigma^\sharp.
}
\]
where the bottom parallelogram is as in \eqref{eq:functoriality diagram in crystalline topos} and the two triangles are as in Set-up~\ref{set-up:description of crystals in terms of connections}: $i$ and $i'$ are closed immersions, and $\pi_Z$ and $\pi_{Z'}$ are smooth.

\begin{thm}\label{thm:compatibility of cohomology of crystal in terms of log de Rham complex and pushforward}
The following diagram
\begin{equation}\label{eq:commutativity of Ru and pushforward diagram}
 \xymatrix{
Ru_{Y/\Sigma,\ast}\calE \ar[r]\ar[d]
& (E\otimes_{\calO_Z}\omega^\bullet_{Z/\Sigma},\nabla) \ar[d]\\
Rf_{\et,\ast}Ru_{Y'/\Sigma',\ast}f_\CRIS^\ast\calE \ar[r]
& Rf_{\et,\ast}(\calD'\otimes_{f_{Z,\et}^\ast\calD}f_{Z,\et}^\ast E\otimes_{\calO_{Z'}}\omega^\bullet_{Z'/\Sigma'},\nabla)
}
\end{equation}
is commutative, where the horizontal maps are given in Theorem~\ref{thm:cohomology of crystal in terms of log de Rham complex} and the vertical maps are given by adjunctions.
If, moreover, $\calE$ satisfies the assumption in Theorem~\ref{thm:cohomology of crystal in terms of log de Rham complex filtered}, the same holds for the filtered version (see \eqref{eq:commutativity of Ru and pushforward diagram-special case} below).
\end{thm}

\begin{proof}
We first work in the set-up of Theorem~\ref{thm:cohomology of crystal in terms of log de Rham complex filtered}: assume that there exists a quasi-coherent $\calO_Z$-module $E_Z$ with an HPD stratification such that $\calE\cong i_{\CRIS}^\ast\calE_Z$ where $\calE_Z$ is a quasi-coherent $\calO_{Z/\Sigma}$-module corresponding to $E_Z$. Then $E\cong \calD\otimes_{\calO_Z}E_Z$, $\calE'_{Z'}\coloneqq f_{Z,\CRIS}^\ast\calE_Z$ corresponds to $E'_{Z'}\coloneqq f_{Z,\et}^\ast E_Z$, and $\calE'\coloneqq f_\CRIS^\ast\calE=i_\CRIS'^\ast\calE'_{Z'}$ corresponds to $\calD'\otimes_{\calO_{Z'}}E'_{Z'}$, where we consider the tensor HPD stratifications. In this case, we need to check the commutativity of 
\begin{equation}\label{eq:commutativity of Ru and pushforward diagram-special case}
 \xymatrix{
Ru_{Y/\Sigma,\ast}(\calJ_{Y/\Sigma}^{[n]}\calE) \ar[r]\ar[d]
& F^n(\calD\otimes_{\calO_Z} E_Z\otimes_{\calO_Z}\omega^\bullet_{Z/\Sigma},\nabla) \ar[d]\\
Rf_{\et,\ast}Ru_{Y'/\Sigma',\ast}(\calJ_{Y'/\Sigma'}^{[n]}\calE') \ar[r]
& Rf_{\et,\ast}(F^n(\calD'\otimes_{\calO_{Z'}}E'_{Z'}\otimes_{\calO_{Z'}}\omega^\bullet_{Z'/\Sigma'},\nabla)),
}
\end{equation}
which is proved as exactly as in \cite[Cor.~V.2.3.4]{Berthelot-book} by verifying the compatibility of $\calL_i(-)$ and $\calL_{i'}(-)$ via $f_\CRIS^\ast$.

The general case will be obtained by decomposing the desired commutative diagram into three diagrams \eqref{eq:commutativity of Ru and pushforward diagram-proof1}, \eqref{eq:commutativity of Ru and pushforward diagram-proof2}, and \eqref{eq:commutativity of Ru and pushforward diagram-proof3} below and showing the commutativity of each diagram separately. 
For this, we follow the notations in Example~\ref{eg:description of i-pushforward of crystal}.
We have a crystal $\calE_{\calA_i}$ of $\calA_i$-module with $\calE_{\calA_i}\otimes_{\calA_i}i_{\CRIS,\ast}\calO_{Y/\Sigma}\xrightarrow{\cong}i_{\CRIS,\ast}\calE$. The $\calO_{Z'}$-module $f_{Z,\et}^\ast E$ comes with the pullback HPD stratification, and the associated crystal of $\calO_{Z'/\Sigma'}$-modules is $f_{Z,\CRIS}^\ast\calE_{\calA_i}$. Similarly, for $\calE'=f_\CRIS^\ast\calE$, we have a crystal $\calE'_{\calA_{i'}}$ of $\calA_{i'}$-module with $\calE'_{\calA_{i'}}\otimes_{\calA_{i'}}i_{\CRIS,\ast}'\calO_{Y'/\Sigma'}\xrightarrow{\cong}i_{\CRIS,\ast}'\calE'$. Note that $\calE'_{\calA_{i'}}$ corresponds to the $\calO_{Z'}$-module $\calD'\otimes_{\calO_{Z'}}f_{Z,\et}^\ast E$ with the tensor product HPD stratification. In particularity, we have an $\calO_{Z'/\Sigma'}$-module map $f_{Z,\CRIS}^\ast\calE_{\calA_i}$ inducing $f_{Z,\CRIS}^\ast\calE_{\calA_i}\otimes_{\calO_{Z'/\Sigma'}}\calA_{i'}\xrightarrow{\cong}\calE'_{\calA_{i'}}$.
The above special case \eqref{eq:commutativity of Ru and pushforward diagram-special case} for $Y=Z$ and $\calE_{\calA_i}$ together with adjunction gives the following commutative diagram
\begin{equation}\label{eq:commutativity of Ru and pushforward diagram-proof1}
\xymatrix{
i_\et^{-1}Ru_{Z,\et,\ast}\calE_{\calA_i}\ar[d]_*+[o][F-]{1}\ar[r]^-\cong
&i_\et^{-1}(E\otimes_{\calO_Z}\omega_{Z/\Sigma}^\bullet,\nabla)\ar[d]
\\
i_\et^{-1}Rf_{Z,\et,\ast}Ru_{Z'/\Sigma',\ast}f_{Z,\CRIS}^\ast\calE_{\calA_i}\ar[d]_*+[o][F-]{2}\ar[r]^-\cong
&i_\et^{-1}Rf_{Z,\et,\ast}(f_{Z,\et}^\ast E\otimes_{\calO_{Z'}} \omega_{Z'/\Sigma'}^\bullet,\nabla)\ar[d]
\\
Rf_{\et,\ast}i_\et'^{-1}Ru_{Z'/\Sigma',\ast}f_{Z,\CRIS}^\ast\calE_{\calA_i}\ar[d]_*+[o][F-]{3}\ar[r]^\cong
&Rf_{\et,\ast}i_\et'^{-1}(f_{Z,\et}^\ast E\otimes \omega_{Z'/\Sigma'}^\bullet,\nabla)\ar[d]
\\
Rf_{\et,\ast}i_\et'^{-1}Ru_{Z',\ast}\calE'_{\calA_{i'}}\ar[r]^-\cong
&Rf_{\et,\ast}i_\et'^{-1}(\calD'\otimes_{\calO_{Z'}}f_{Z,\et}^\ast E\otimes_{\calO_{Z'}} \omega_{Z'/\Sigma'}^\bullet,\nabla).
} 
\end{equation}

Next the following diagram with obvious abbreviation
\begin{equation}
\label{eq:commutativity of Ru and pushforward diagram-proof2}
\xymatrix{
&
i^{-1}Ru_{Z,\ast}i_{\CR,\ast}\calE\ar[d]\ar[ld]_*+[o][F-]{4}
&
i^{-1}Ru_{Z,\ast}\calE_{\calA_i}\ar[d]^*+[o][F-]{1}\ar[l]_-\cong
\\
i^{-1}Ru_{Z,\ast}i_{\CR,\ast}Rf_{\CR,\ast}f_{\CR}^\ast\calE \ar[d]^-\cong_*+[o][F-]{5}
&
i^{-1}Ru_{Z,\ast}Rf_{Z,\CR,\ast}f_{Z,\CR}^\ast i_{\CR,\ast}\calE\ar[d]\ar[l]
&
i^{-1}Rf_{Z,\ast}Ru_{Z',\ast}f_{Z,\CR}^\ast\calE_{\calA_i}\ar[d]^*+[o][F-]{2}\ar[l]
\\
i^{-1}Rf_{Z,\ast} Ru_{Z',\ast}i'_{\CR,\ast}f_{\CR}^\ast\calE \ar[rd]
&
Rf_\ast i'^{-1}Ru_{Z',\ast}f_{Z,\CR}^\ast i_{\CR,\ast}\calE\ar[d]
&
Rf_{\ast}i'^{-1}Ru_{Z'/\Sigma',\ast}f_{Z,\CR}^\ast\calE_{\calA_i}\ar[d]^*+[o][F-]{3}\ar[l]
\\
&
Rf_\ast i'^{-1}Ru_{Z',\ast}i'_{\CR,\ast}f_{\CR}^\ast\calE
&
Rf_{\ast}i'^{-1}Ru_{Z',\ast}\calE'_{\calA_{i'}}\ar[l]_-\cong
} 
\end{equation}
is commutative: the commutativity of the left top triangle follows from the dual statement of \cite[Cor.~V.3.3.2(ii)]{Berthelot-book}, and that of the left bottom trapezoid follows from the functoriality of adjunction.

We conclude the commutativity of \eqref{eq:commutativity of Ru and pushforward diagram} by combining \eqref{eq:commutativity of Ru and pushforward diagram-proof1} and \eqref{eq:commutativity of Ru and pushforward diagram-proof2} with the following obvious commutative diagram
\begin{equation}
\label{eq:commutativity of Ru and pushforward diagram-proof3}
\xymatrix{
i^{-1}i_{\ast}Ru_{Y,\ast}\calE\ar[r]^\cong \ar[d]
&
i^{-1}Ru_{Z,\ast}i_{\CR,\ast}\calE \ar[d]^*+[o][F-]{4}
\\
i^{-1}i_{\ast}Ru_{Y,\ast}Rf_{\CR,\ast}f_{\CR}^\ast\calE \ar[r]^-\cong\ar[dr]_-\cong
&
i^{-1}Ru_{Z,\ast}i_{\CR,\ast}Rf_{\CR,\ast}f_{\CR}^\ast\calE\ar[d]_-\cong^*+[o][F-]{5}
\\
&
i^{-1}Rf_{Z,\ast} Ru_{Z',\ast}i'_{\CR,\ast}f_{\CR}^\ast\calE.
}
\end{equation}
\end{proof}

For a later purpose, let us briefly discuss a slightly more general situation.

\begin{rem}\label{rem:locally closed embedding case}
Keep the notation as in Set-up~\ref{set-up:relative log-crystalline site} and assume that there exists a \emph{locally closed} immersion $\overline{\imath}\colon (Y,M_Y)\hookrightarrow (\overline{Z},M_{\overline{Z}})$ of fine log schemes such that $(\overline{Z},M_{\overline{Z}})\rightarrow (\Sigma,M_\Sigma)$ is smooth and $\gamma_\Sigma$ extends to $\overline{Z}$. Then $\overline{\imath}$ factors as $(Y,M_Y)\overset{i}{\hookrightarrow} (Z,M_Z)\overset{j}{\hookrightarrow} (\overline{Z},M_{\overline{Z}})$ where $i$ is a closed immersion and $j$ is an open immersion.
Apply the notation of Set-up~\ref{set-up:description of crystals in terms of connections} and the following paragraph to $i$. Similarly, write $(\overline{Z}^{\nu+1},M_{\overline{Z}^{\nu+1}})$ for the $(\nu+1)$st self-fiber product of $(\overline{Z},M_{\overline{Z}})$ over $(\Sigma,M_\Sigma)$.
Then $(Y,M_{Y})\hookrightarrow (D_Y(Z^{\nu+1}),M_{D_Y(Z^{\nu+1})})$ coincides with the PD-envelope of $(Y,M_Y)\hookrightarrow (\overline{Z}^{\nu+1},M_{\overline{Z}^{\nu+1}})$ relative to $\Sigma^\sharp$, and the canonical map $\calO_D\otimes_{\calO_{\overline{Z}}}\omega^l_{\overline{Z}/\Sigma}\rightarrow \calO_D\otimes_{\calO_Z}\omega^l_{Z/\Sigma}$ (as a morphism in $\Sh(D_\et)$) is an isomorphism.
Given these remarks, it is tedious but straightforward to see that Theorems~\ref{thm:cohomology of crystal in terms of log de Rham complex filtered}, \ref{thm:cohomology of crystal in terms of log de Rham complex}, and \ref{thm:compatibility of cohomology of crystal in terms of log de Rham complex and pushforward} continue to hold (with obvious modifications replacing $Z$ with $\overline{Z}$ and so on).
\end{rem}

\section{Boundedness of log crystalline cohomology}\label{sec:Boundedness of log crystalline cohomology}

In this section, we prove a boundedness result on higher direct image for log crystalline sites. Let $\Sigma^\sharp$ be a fine log PD-scheme in which $p$ is nilpotent.

\begin{thm}\label{thm:boundedness of crystalline pushforward}
Let $f\colon (X,M_X)\rightarrow (Y,M_Y)$ be a morphism of fine log $\Sigma^\sharp$-schemes.
Assume that $Y$ is quasi-compact and $f$ is quasi-separated of finite type. Then there exists an integer $r$ such that $R^qf_{\CRIS,\ast}\calE=0$ for every $q\geq r$ and every quasi-coherent $\calO_{X/\Sigma}$-module $\calE$. The same holds for $Rf_{\cris,\ast}$.
\end{thm}

\begin{proof}
By Lemma~\ref{lem:evalutation of crystalline pushforward via projection}(2), the assertion for $R^qf_{\CRIS,\ast}\calE$ is reduced to Proposition~\ref{prop:vanishing of crystal under projection} below. 
It is tedious but straightforward to check that the proof therein also works for the small crystalline cite with Theorem~\ref{thm:cohomology of crystal in terms of log de Rham complex} replaced by \cite[Thm.~6.4]{Kato-log}, and thus we obtain the second assertion.
\end{proof}

\begin{prop}\label{prop:vanishing of crystal under projection}
Keep the assumption and notation as in Theorem~\ref{thm:boundedness of crystalline pushforward}.
Then there exists an integer $r$ satisfying the following property: Let $U\rightarrow Y$ be an \'etale morphism, $T^\sharp=(T,M_T,\calJ_T,\gamma_T)$ an affine fine log PD-scheme over $\Sigma^\sharp$, and $(U,M_U)\hookrightarrow (T,M_T)$ an exact closed immersion defined by a sub PD-ideal of $\calJ_T$. Set $X_U\coloneqq X\times_YU$ and let $\pi_{X_U/T}$ denote the composite 
\[
\pi_{X_U/T}\colon \Sh(((X_U,M_{X_U})/T^\sharp)_\CRIS)\xrightarrow{u_{X_U/T}} \Sh((X_U)_\et)\rightarrow\Sh(T_\et). 
\]
Then $R^q\pi_{X_U/T,\ast}\calE=0$ for every $q\geq r$ and every quasi-coherent $\calO_{X_U/T}$-module $\calE$ on $((X_U,M_{X_U})/T^\sharp)_\CRIS$.
\end{prop}

The proof of Proposition~\ref{prop:vanishing of crystal under projection} occupies the rest of this section.
First of all, we may assume that $\Sigma$ and $Y$ are affine: since $Y$ is quasi-compact, it admits a finite Zariski covering of affine opens $\{Y_\lambda\}$. With the notation as in Definition~\ref{defn:localization of crystalline site by open}, the family $\{V_{\lambda,\CRIS}\}$ forms a finite open covering of the topos $\Sh(((Y,M_Y)/\Sigma^\sharp)_\CRIS)$ and thus satisfies the assumption of \cite[V.3.1.3]{Berthelot-book}. If Proposition~\ref{prop:vanishing of crystal under projection} holds for each $(X_{Y_\lambda},M_{X_{Y_\lambda}})\rightarrow (Y_\lambda,M_{Y_\lambda})$, then it also holds for $(X,M_X)\rightarrow (Y,M_Y)$ by \cite[Prop.~V.3.1.4; pp.~320-322]{Berthelot-book}. By refining the Zariski covering if necessary, we may and do assume that $\Sigma$ and $Y$ are affine.

\begin{construction}[{cf.~\cite[Prop.~2.2.11]{Shiho-II}}]\label{construction:embedding system}
 \'Etale locally, choose a closed immersion $X\hookrightarrow \A_\Sigma^a$ of schemes and a fine chart $(P_X\rightarrow M_X,Q_\Sigma\rightarrow M_\Sigma, Q\rightarrow P)$ of $f$. Take also a monoid surjection $\N^b\rightarrow P$. Set $Z\coloneqq \A_\Sigma^{a+b}$ with coordinates $t_1,\ldots,t_{a+b}$ and equip $Z$ with log structure $M_Z$ induced by the map $Q\oplus \N^b\rightarrow \Gamma(Z,\calO_Z)$ sending $(n_i)\in \N^b$ to $\prod_i t_{a+i}^{n_i}$. Then $X\hookrightarrow \A_\Sigma^a$ and $Q\oplus \N^b \rightarrow P$ induce a closed immersion $(X,M_X)\hookrightarrow (Z,M_Z)$ of fine log schemes over $(\Sigma,M_\Sigma)$.

Since $X$ is quasi-compact, applying the construction in the preceding paragraph \'etale locally and taking disjoint union gives the following objects:
\begin{itemize}
 \item an \'etale surjection $X_0\rightarrow X$ from an affine scheme $X_0$;
 \item a smooth morphism $(Z_0,M_{Z_0})\rightarrow (\Sigma,M_\Sigma)$ of log schemes with $Z_0$ being affine;
 \item a closed immersion $(X_0,M_{X_0})\hookrightarrow (Z_0,M_{Z_0})$ of log schemes over $(\Sigma,M_\Sigma)$.
\end{itemize}
Moreover, they satisfy the following properties:
\begin{enumerate}
 \item $X_0\rightarrow X$ is separable (this is automatic as $X_0$ is affine);
 \item $Z_0\rightarrow \Sigma$ is flat;
 \item the underlying scheme of any higher self-fiber product of $(Z_0,M_{Z_0})$ over $(\Sigma,M_\Sigma)$ in the category of fine log schemes coincides with the higher self-fiber product of $Z_0$ over $\Sigma$.
\end{enumerate}

 For $l\geq 1$, let $X_l$ denote the $(l+1)$-st self-fiber product of $X_0$ over $X$ and let $(Z_l,M_{Z_l})$ denote the $(l+1)$-st self fiber product of $(Z_0,M_{Z_0})$ over $(\Sigma,M_\Sigma)$. Then we have a locally closed immersion $(X_l,M_{X_l})\hookrightarrow(Z_l,M_{Z_l})$; let $(X_l,M_{X_l})\hookrightarrow (D_l,M_{D_l})$ be its PD-envelope. Note that $X_l\hookrightarrow Z_l$ is a closed immersion if $X\rightarrow \Sigma$ is separated and they are affine if $X\rightarrow \Spec \Z$ is separated. 

By functoriality, we obtain morphisms of simplicial log schemes
\[
(X_\bullet,M_{X_\bullet})\hookrightarrow (D_\bullet,M_{D_\bullet})\rightarrow (Z_\bullet,M_{Z_\bullet}).
\]
By construction, $X_\bullet$ is the $0$-coskeleton of the \'etale covering $X_0\rightarrow X$. 
\end{construction}

Take $U\rightarrow Y$, $T^\sharp$, and $(U,M_U)\hookrightarrow (T,M_T)$ as in Proposition~\ref{prop:vanishing of crystal under projection}.
Consider the \'etale surjection $X_{U,0}\coloneqq X_0\times_YU\rightarrow X_U$; since $X_0$, $Y$, and $U$ are all affine, so is $X_{U,0}$. Let $X_{U,\bullet}$ denote the $0$-coskeleton of $X_{U,0}\rightarrow X_U$, which is the base change of $X_\bullet$ along $U\rightarrow Y$. Similarly, set $(Z_{T,0},M_{Z_{T,0}})\coloneqq (Z_0,M_{Z_0})\times_{(\Sigma,M_\Sigma)}(T,M_T)$, and for $l\geq 1$, define 
$(Z_{T,l},M_{Z_{T,l}})$ to be the $(l+1)$-st self fiber product of $(Z_{T,0},M_{Z_{T,0}})$ over $(T,M_T)$. We have a locally closed immersion $(X_{U,l},M_{X_{U,l}})\hookrightarrow (Z_{T,l},M_{Z_{T,l}})$ and let $(X_{U,l},M_{X_{U,l}})\hookrightarrow (D_{U,l},M_{D_{U,l}})$ denote the PD-envelope with respect to $T^\sharp$.
Consider the commutative diagram of topoi
\[
\xymatrix{
\Sh(((X_{U,\bullet},M_{X_{U,\bullet}})/T^\sharp)_\CRIS)\ar[d]_-{a_\CRIS}\ar[rrd] & &\\
\Sh(((X_U,M_{X_U})/T^\sharp)_\CRIS)\ar[r]^-{u_{X_U/T}}\ar@/_15pt/[rr]_-{\pi_{X_U/T}} & \Sh(X_{U,\et})\ar[r]&\Sh(T_\et),
}
\]
where $\Sh(((X_{U,\bullet},M_{X_{U,\bullet}})/T^\sharp)_\CRIS)$ is the obvious simplicial topos defined by the simplicial log scheme $(X_{U,\bullet},M_{X_{U,\bullet}})$.
We start with the quasi-coherence of the higher direct image.

\begin{prop}\label{prop:quasi-coherence of crystal under projection}
Keep the assumption and notation as above. For every quasi-coherent $\calO_{X_U/T}$-module $\calE$ on $((X_U,M_{X_U})/T^\sharp)_\CRIS$ and every $q$, the higher direct image $R^q\pi_{X_U/T,\ast}\calE$ is a quasi-coherent $\calO_T$-module.
\end{prop}

\begin{proof}
We know from Lemma~\ref{lem:properties of VCRIS}(1)(2) and \cite[Lem.~1.4.24]{Olsson-crystallinecohomology} that the unit map of the adjunction $\calE\rightarrow Ra_{\CRIS,\ast}a_\CRIS^\ast\calE$ is a quasi-isomorphism. As in \cite[09WJ]{stacks-project}\footnote{Our set-up is slightly more general than the one considered therein, but the same proof works.}, we have a spectral sequence
\[
E_1^{s,t}=R^t\pi_{X_{U,s}/T,\ast}(\calE|_{X_{U,s}})\implies R^{s+t}\pi_{X_U/T,\ast}\calE,
\]
where $\calE|_{X_{U,s}}$ denotes the restriction of $\calE$ to $\Sh(((X_{U,s},M_{X_{U,s}})/T^\sharp)_\CRIS)$ and each differential map is $\calO_T$-linear. So it remains to show that each $R^t\pi_{X_{U,s}/T,\ast}(\calE|_{X_{U,s}})$ is a quasi-coherent $\calO_T$-module. Let $(E,\nabla)$ denote the quasi-coherent $\calO_{D_{U,s}}$-module with log connection on $D_{U,s}$ corresponding to $\calE|_{X_{U,s}}$ and write $h$ for the map $D_{U,s}\rightarrow T$.
By Theorem~\ref{thm:cohomology of crystal in terms of log de Rham complex} and Remark~\ref{rem:locally closed embedding case}, we have a quasi-isomorphism
\[
R\pi_{X_{U,s}/T,\ast}(\calE|_{X_{U,s}})\cong Rh_{\et,\ast}(E\otimes_{\calO_{Z_{T,s}}}\omega_{Z_{T,s}/T}^\bullet,\nabla).
\]
Here $(E\otimes_{\calO_{Z_{T,s}}}\omega_{Z_{T,s}/T}^\bullet,\nabla)$ is a bounded complex of quasi-coherent $\calO_{D_{U,s}}$-modules with $h_\et^{-1}\calO_T$-linear differentials, and $h$ is quasi-compact and quasi-separated. Hence the assertion now follows from another spectral sequence and the standard result on the higher direct image of a quasi-coherent sheaf along a qcqs morphism (for the small \'etale sites). For example, see \cite[01XJ, 03OJ, 03P2, 075A]{stacks-project}.
\end{proof}

By Proposition~\ref{prop:quasi-coherence of crystal under projection}, the proof of Proposition~\ref{prop:vanishing of crystal under projection} is reduced to showing that there exists an integer $r$ such that for any $U\rightarrow Y$, $T^\sharp$, and $(U,M_U)\hookrightarrow (T,M_T)$ and any quasi-coherent $\calO_{X_U/T}$-module $\calE$, we have $R^q\Gamma(T_\et,R\pi_{X_U/T,\ast}\calE)=0$ for $q\geq r$ (note that $T$ is affine); equivalently, $R^q\Gamma(X_{U,\et},Ru_{X_U/T,\ast}\calE)=0$ for $q\geq r$.
We use the alternating \v{C}ech complex on the \'etale site by Bhatt and de Jong.

We need a construction as in \cite[072C]{stacks-project}. Recall that $X_0\rightarrow X$ is a separated \'etale surjection. Let $c$ denote the maximum number of points of the geometric fibers of $X_0\rightarrow X$. For each $l\geq 0$, let $X_l^\circ\subset X_l$ be the complement of all the subdiagonals (so $X_0^\circ=X_0$); then $X_l^\circ$ is open and closed in $X_l$, and the symmetric group $S_{l+1}$ in $(l+1)$ letters acts freely on $X_l^\circ$. Let $\calX_l^\circ$ denote the (\'etale sheaf) quotient $X_l^\circ/S_{l+1}$, which is a quasi-compact algebraic space, \'etale over $X$ and separated over $\Spec \Z$, by \cite[071S, 02Z4]{stacks-project}. Finally, take an \'etale surjection $V_l\rightarrow \calX_l^\circ$ from an affine scheme $V_l$. Let $c_l$ denote the maximum number of points of the geometric fibers of $V_l\rightarrow \calX_l^\circ$. We are going to show that $r\coloneqq \max_{0\leq l\leq c}c_l+c+(c+1)\dim(Z_0/\Sigma)$ works, where $\dim(Z_0/\Sigma)$ denotes the relative dimension of $Z_0\rightarrow\Sigma$.

Keep the notation as in the paragraph before Proposition~\ref{prop:quasi-coherence of crystal under projection}. Since $X_{U,0}=X_0\times_YU$ is affine, the morphism $X_{U,0}\rightarrow X_U$ is a separated \'etale surjection. Define $X_{U,l}\supset X_{U,l}^\circ \rightarrow \calX_{U,l}^\circ\coloneqq X_{U,l}^\circ/S_{l+1}$ in the same way. Let $D_{U,l}^\circ\subset D_{U,l}$ denote the open and closed subscheme whose underlying topological space is given by $X_{U,l}^\circ$. The action of $S_{l+1}$ on $Z_{T,l}$ defines a free action of $S_{l+1}$ on $D_{U,l}^\circ$, and we let $\calD_{U,l}^\circ$ denote the algebraic space $D_{U,l}^\circ/S_{l+1}$. One can check that the induced map $\calX_{U,l}^\circ\rightarrow\calD_{U,l}^\circ$ is a nil immersion, and $X_{U,l}^\circ\rightarrow\calX_{U,l}^\circ\times_{\calD_{U,l}^\circ}D_{U,l}^\circ$ is an isomorphism.

The following Lemmas~\ref{lem:explicit vanishing of quasi-coherent sheaves} and \ref{lem:spectral sequence for alternating Cech complex} are taken from \cite{stacks-project}. Combining them with Theorem~\ref{thm:cohomology of crystal in terms of log de Rham complex}, we will deduce Proposition~\ref{prop:vanishing of crystal under projection}.

\begin{lem}\label{lem:explicit vanishing of quasi-coherent sheaves}
For every quasi-coherent sheaf $E$ on $\calX_{U,l,\et}^\circ$, we have $H^q(\calX_{U,l,\et}^\circ,E)=0$ for $q\geq c_l$. A similar statement holds for $\calD_{U,l}^\circ$. 
\end{lem}

\begin{proof}
Observe $\calX_{U,l}^\circ\cong \calX_{l}^\circ\times_YU$. In particular, $V_l\times_YU$ is an affine scheme admitting an \'etale surjection to $\calX_{U,l}^\circ$, and $c_l$ is at least the maximum number of points of the geometric fibers of $V_l\times_YU\rightarrow \calX_{U,l}^\circ$. Hence the assertion follows from \cite[072B]{stacks-project} and the fact that $\calX_{U,l}^\circ\hookrightarrow\calD_{U,l}^\circ$ is a nil immersion. 
\end{proof}

We name the maps as follows:
\[
\xymatrix{
D_{U,l}^\circ \ar@{^{(}->}[r]\ar[d] & D_{U,l}\ar[r] & Z_{U,l}\ar[r] & T\\
\calD_{U,l}^\circ&X_{U,l}^\circ \ar@{^{(}->}[r]\ar[d]\ar[dr]^-{g_l^\circ}\ar@{_{(}->}[ul]_-{i_l^\circ} & X_{U,l} \ar[d]^-{g_l}\ar@{_{(}->}[ul]_-{i_l}\ar@{_{(}->}[u]&\\
&\calX_{U,l}^\circ \ar[r]^-{\bar{g}_l^\circ}\ar@{_{(}->}[ul]_-{\bar{\imath}_l^\circ} & X_U\ar[r]&U.\ar@{^{(}->}[uu]
}
\]
Let $\Z(\chi_l)$ denote a rank one locally free $\Z$-module on $\calX_{U,l,\et}^\circ$ given by the sign character $S_{l+1}\rightarrow \{\pm1\}$. In what follows, we write $g_{0,!}$ for $g_{0,\et,!}$ and so on for simplicity. Set $K^l\coloneqq \bigwedge^{l+1}g_{0,!}\Z\in \Ab(X_{U,\et})$, which defines a Koszul complex $K^\bullet=(\cdots\rightarrow K^2\rightarrow K^1\rightarrow K^0)$. Note $K^l=0$ if $l\geq c$. Moreover, the adjoint $K^0\rightarrow \Z$ to the natural map $\Z\rightarrow g_0^{-1}\Z$ gives a quasi-isomorphism $K^\bullet \xrightarrow{\cong}\Z[0]$ in $D^b(X_{U,\et})\coloneqq D^b(\Ab(X_{U,\et}))$ (see \cite[0723]{stacks-project}).

\begin{lem}\label{lem:spectral sequence for alternating Cech complex}
For every $L\in D^+(X_{U,\et})$, there is a spectral sequence
\[
E_1^{s,t}=\Ext^t(K^s,L)\implies R^{s+t}\Gamma(X_{U,\et},L).
\]
Moreover, there is an isomorphism
\[
\Ext^t(K^s,L)\cong H^t(\calX_{U,s,\et}^\circ,(\bar{g}_s^\circ)^{-1}L\otimes_\Z\Z(\chi_s)).
\]
\end{lem}

\begin{proof}
The proof of \cite[0725, 0728]{stacks-project} works verbatim by taking a quasi-isomorphism $L\rightarrow I$ into a bounded below complex $I$ of injective sheaves of abelian groups on $X_{U,\et}$. 
\end{proof}

Take any quasi-coherent $\calO_{X_U/T}$-module $\mathcal{E}$ and apply the above discussions to $L=Ru_{X_U/T}\calE\in D^+(X_{U,\et})$. We obtain a spectral sequence
\[
E_1^{s,t}=R^t\Gamma(\calX_{U,s,\et}^\circ,(\bar{g}_s^\circ)^{-1}(Ru_{X_U/T}\calE)\otimes_\Z\Z(\chi_s))\implies R^{s+t}\Gamma(X_{U,\et},Ru_{X_U/T}\calE).
\]
Since $\bar{\imath}_s^\circ$ gives an equivalence $\Sh(\calX_{U,s,\et}^\circ)\cong \Sh(\calD_{U,s,\et}^\circ)$, we are reduced to studying 
\[
R^t\Gamma(\calD_{U,s,\et}^\circ,\bar{\imath}_{s,\ast}^\circ(\bar{g}_s^\circ)^{-1}(Ru_{X_U/T}\calE)\otimes\bar{\imath}_{s,\ast}^\circ\Z(\chi_s)).
\]

Let $(E,\nabla)$ denote the quasi-coherent $\calO_{D_{U,s}}$-module with log connection on $D_{U,s}$ corresponding to $\calE|_{X_{U,s}}$. It yields an $S_{s+1}$-equivariant complex $(E\otimes_{\calO_{Z_{T,s}}}\omega_{Z_{T,s}/T}^\bullet,\nabla)$. The restriction $(E\otimes_{\calO_{Z_{T,s}}}\omega_{Z_{T,s}/T}^\bullet,\nabla)|_{D_{U,s,\et}^\circ}$ is $S_{s+1}$-equivariant. Hence it descends to a complex $\bar{E}^\bullet$ of abelian sheaves on $\calD_{U,s,\et}^\circ$, each $\bar{E}^m$ admits a structure of a quasi-coherent $\calO_{\calD_{U,s}^\circ}$-module, and $\overline{E}^m=0$ if $m<0$ or $m>\dim(Z_{T,s}/T)=(s+1)\dim(Z_0/\Sigma)$: first descend each quasi-coherent module using \cite[03M3]{stacks-project} and then check that abelian sheaf morphisms between them also descend.

By Theorem~\ref{thm:cohomology of crystal in terms of log de Rham complex} and Remark~\ref{rem:locally closed embedding case}, we have a quasi-isomorphism
\[
Ru_{X_{U,s}/T,\ast}(\calE|_{X_{U,s}})\cong i_s^{-1}(E\otimes_{\calO_{Z_{T,s}}}\omega_{Z_{T,s}/T}^\bullet,\nabla)
\]
in $D^+(X_{U,s,\et})$.
Observe that the source also admits a natural action of $S_{s+1}$ and the above quasi-isomorphism is $S_{s+1}$-equivariant by Theorem~\ref{thm:compatibility of cohomology of crystal in terms of log de Rham complex and pushforward}.
We obtain a quasi-isomorphism
\[
\bar{\imath}_{s,\ast}^\circ(\bar{g}_s^\circ)^{-1}(Ru_{X_U/T}\calE)\otimes\bar{\imath}_{s,\ast}^\circ\Z(\chi_s)\cong \bar{E}^\bullet\otimes \bar{\imath}_{s,\ast}^\circ\Z(\chi_s)
\]
and a spectral sequence
\[
E_1^{m,n}=R^n\Gamma(\calD_{U,s,\et}^\circ, \bar{E}^m\otimes \bar{\imath}_{s,\ast}^\circ\Z(\chi_s))\implies 
R^{m+n}\Gamma(\calD_{U,s,\et}^\circ,\bar{\imath}_{s,\ast}^\circ(\bar{g}_s^\circ)^{-1}(Ru_{X_U/T}\calE)\otimes\bar{\imath}_{s,\ast}^\circ\Z(\chi_s)).
\]
Here each $\bar{E}^m\otimes \bar{\imath}_{s,\ast}^\circ\Z(\chi_s)$ is a quasi-coherent $\calO_{\calD_{U,s}^\circ}$-module. So $E_1^{m,n}=0$ unless $m\in [0,(s+1)\dim(Z_0/\Sigma)]$ and $n\in [0, c_s)$ by Lemma~\ref{lem:explicit vanishing of quasi-coherent sheaves}.
In particular, $R^{t}\Gamma(\calD_{U,s,\et}^\circ,\bar{\imath}_{s,\ast}^\circ(\bar{g}_s^\circ)^{-1}(Ru_{X_U/T}\calE)\otimes\bar{\imath}_{s,\ast}^\circ\Z(\chi_s))=0$ if $t\geq c_s+(s+1)\dim(Z_0/\Sigma)$.

Recall the spectral sequence in Lemma~\ref{lem:spectral sequence for alternating Cech complex}:
\[
E_1^{s,t}=R^t\Gamma(\calX_{U,s,\et}^\circ,(\bar{g}_s^\circ)^{-1}(Ru_{X_U/T}\calE)\otimes_\Z\Z(\chi_s))\implies R^{s+t}\Gamma(X_{U,\et},Ru_{X_U/T}\calE).
\]
We now know $E_1^{s,t}=0$ unless $s\in [0,c]$ and $t\in [0,c_s+(s+1)\dim(Z_0/\Sigma))$ (note $E_1^{s,t}=\Ext^t(K^s, Ru_{X_U/T}\calE)$ and $K^s=0$ if $s\geq c$).
Hence we conclude that if $q\geq r\coloneqq \max_{0\leq l\leq c}c_l+c+(c+1)\dim(Z_0/\Sigma)$, then $R^{q}\Gamma(X_{U,\et},Ru_{X_U/T}\calE)=0$.
This completes the proof of Proposition~\ref{prop:vanishing of crystal under projection}.

\section{Base change}\label{sec:Base change}

\begin{thm}\label{thm:crystalline base change for projection}
Let $g\colon \Sigma'^\sharp=(\Sigma',M_{\Sigma'},\calJ',\gamma')\rightarrow \Sigma^\sharp=(\Sigma,M_\Sigma,\calJ,\gamma)$ be a morphism of integral and quasi-coherent log PD-schemes in which $p$ is nilpotent.
Suppose that we have the following diagram
\[
\xymatrix{
(X',M_{X'})\ar[r]^-{g'} \ar[d]_-{f_0'} \ar@/_30pt/[dd]_-{\pi'}& (X,M_X)\ar[d]^-{f_0}\ar@/^30pt/[dd]^-{\pi}\\
(\Sigma_0',M_{\Sigma_0'})\ar[r]^-{g_0} \ar@{^{(}->}[d]_-{i'}& (\Sigma_0,M_{\Sigma_0})\ar@{^{(}->}[d]^-{i}\\
(\Sigma',M_{\Sigma'}) \ar[r]^-g& (\Sigma,M_\Sigma),
}
\]
where $i$ and $i'$ are exact closed immersions given by sub PD-ideals of $\calJ$ and $\calJ'$, respectively.
Assume that the upper-square is Cartesian, $\Sigma$ is quasi-compact, $f_0$ is an integral and smooth morphism of fine log schemes with the underlying morphism of schemes being qcqs.
Then for every flat quasi-coherent $\calO_{X/\Sigma}$-module $\calE$ on $((X,M_X)/\Sigma^\sharp)_\CRIS$, the canonical morphism 
\begin{equation}\label{eq:crystalline base change}
Lg^\ast_\et R\pi_{X/\Sigma,\ast}\calE\rightarrow R\pi'_{X'/\Sigma',\ast}g'^\ast_\CRIS\calE.    
\end{equation}
is a quasi-isomorphism (see below).
\end{thm}

In the above theorem, note 
\[
g_{\et}\circ \pi'_{X'/\Sigma'}\cong \pi_{X/\Sigma}\circ g_\CRIS'\colon (\Sh(((X',M_{X'})/\Sigma'^\sharp)_\CRIS),\calO_{X'/\Sigma'})\rightarrow (\Sh(\Sigma_\et),\calO_\Sigma). 
\]
So we have $R\pi_{X/\Sigma,\ast}\calE\rightarrow R\pi_{X/\Sigma,\ast}Rg'_{\CRIS,\ast}g'^\ast_\CRIS\calE \cong Rg_{\et,\ast}R\pi'_{X'/\Sigma',\ast}g'^\ast_\CRIS\calE$, which yields the morphism $Lg^\ast_\et R\pi_{X/\Sigma,\ast}\calE\rightarrow R\pi'_{X'/\Sigma',\ast}g'^\ast_\CRIS\calE$ by adjunction and Proposition~\ref{prop:vanishing of crystal under projection} (recall $g'^\ast_\CRIS=Lg'^\ast_\CRIS$ as we work on the big crystalline sites). Note that the same statement for small crystalline sites is stated in \cite[Thm.~6.10]{Kato-log}.

\begin{proof}
We will follow \cite[Prop.~V.3.5.2, Pf.]{Berthelot-book}.
Since the assertion is local on $\Sigma'$, we may assume that $\Sigma$ and $\Sigma'$ are both affine. We will first reduce to the case where $X$ is affine. Take any finite open covering $X=U_0\cup\cdots \cup U_n$.

Let $\Delta_n$ denote the opposite category of the category consisting of non-empty subsets of $\{0,\ldots,n\}$ with morphisms being inclusions.
For each $\{i_0<\cdots<i_l\}\in \Delta_n$, set $U_{i_0\cdots i_l}\coloneqq U_{i_0}\cap \cdots\cap U_{i_l}$. With the notation as in Definition~\ref{defn:localization of crystalline site by open}, we obtain the topos $\Sh(((U_\bullet,M_{U_\bullet})/\Sigma^\sharp)_\CRIS)$ associated to the diagram of topoi indexed by $\Delta_n$ associating to $\{i_0<\cdots<i_l\}$ the topos $\Sh(((X,M_{X})/\Sigma^\sharp)_\CRIS)/U_{i_0\cdots i_l,\CRIS}\cong \Sh((U_{i_0\cdots i_l},M_{U_{i_0\cdots i_l}})/\Sigma^\sharp)_\CRIS)$.
Set $U_l'\coloneqq U_l\times_{\Sigma_0}\Sigma_0'$ and define $\Sh(((U_\bullet',M_{U_\bullet'})/\Sigma^\sharp)_\CRIS)$ in the same way.

We have a commutative diagram of topoi
\[
\xymatrix{
\Sh(((U_\bullet,M_{U_\bullet})/\Sigma^\sharp)_\CRIS)\ar[r]^-{\pi_{U_\bullet/\Sigma_\bullet}}\ar[d]_-{a_\CRIS}\ar[dr]^-{\pi_{U_\bullet/\Sigma}} & \Sh(\Sigma_\et)\times\Delta_n\ar[d]^-b\\
\Sh(((X,M_X)/\Sigma^\sharp)_\CRIS)\ar[r]^-{\pi_{X/\Sigma}} &\Sh(\Sigma_\et),
}
\]
where $\Sh(\Sigma)\times\Delta_n$ denotes the constant topos indexed by $\Delta_n$. We have a similar diagram for $(X',M_{X'})\rightarrow \Sigma'^\sharp$ and name the morphisms in the same way. We also have a commutative diagram of topoi
\[
\xymatrix{
\Sh(((U_\bullet',M_{U_\bullet'})/\Sigma'^\sharp)_\CRIS)\ar[r]^-{\pi'_{U'_\bullet/\Sigma'_\bullet}}\ar@/^20pt/[rr]^-{\pi'_{U'_\bullet/\Sigma'}}\ar[d]_-{g'_{\bullet,\CRIS}} & \Sh(\Sigma'_\et)\times\Delta_n\ar[d]^-{g_{\bullet,\et}}\ar[r]^-{b'} & \Sh(\Sigma'_\et)\ar[d]^-{g_\et}\\
\Sh(((U_\bullet,M_{U_\bullet})/\Sigma^\sharp)_\CRIS)\ar[r]^-{\pi_{U_\bullet/\Sigma_\bullet}} \ar@/_15pt/[rr]_-{\pi_{U_\bullet/\Sigma}}& \Sh(\Sigma_\et)\times\Delta_n\ar[r]^-b &\Sh(\Sigma_\et).
}
\]
As in \cite[pp.~344-347]{Berthelot-book}, the functoriality of adjunction 
for the relevant ringed topoi gives the following commutative diagram
\[
\xymatrix{
Lg^\ast_\et R\pi_{X/\Sigma,\ast}\calE\ar[r]^-{\text{\eqref{eq:crystalline base change}}}\ar[d]_-\cong & R\pi'_{X'/\Sigma',\ast}g'^\ast_\CRIS\calE\ar[d]^-\cong\\
Lg^\ast_\et R\pi_{U_\bullet/\Sigma,\ast}a_\CRIS^\ast\calE\ar[r]\ar[d]_-\cong & R\pi'_{U'_\bullet/\Sigma',\ast}g'^\ast_{\bullet,\CRIS} a_\CRIS^\ast\calE\ar[d]^-\cong\\
Lg^\ast_\et Rb_\ast R\pi_{U_\bullet/\Sigma_\bullet,\ast}a_\CRIS^\ast\calE\ar[r]\ar[d]_-\cong & Rb'_\ast R\pi'_{U'_\bullet/\Sigma'_\bullet,\ast}g'^\ast_{\bullet,\CRIS} a_\CRIS^\ast\calE \ar@{=}[d]\\
Rb'_\ast 
Lg^\ast_{\bullet,\et} R\pi_{U_\bullet/\Sigma_\bullet,\ast}a_\CRIS^\ast\calE\ar[r] & Rb'_\ast R\pi'_{U'_\bullet/\Sigma'_\bullet,\ast}g'^\ast_{\bullet,\CRIS} a_\CRIS^\ast\calE.
}
\]
Here the horizontal maps are given by adjunction; the top vertical maps are quasi-isomorphisms by \cite[Prop.~V.3.4.8]{Berthelot-book}; the middle vertical quasi-isomorphisms are given by the composition of right derived functors; the bottom left vertical map is given by adjunction and is a quasi-isomorphism by \cite[Prop.~V.3.4.9]{Berthelot-book}.

In particular, to show that \eqref{eq:crystalline base change} is a quasi-isomorphism, it suffices to show that the map $
Lg^\ast_{\bullet,\et} R\pi_{U_\bullet/\Sigma_\bullet,\ast}a_\CRIS^\ast\calE\rightarrow R\pi'_{U'_\bullet/\Sigma'_\bullet,\ast}g'^\ast_{\bullet,\CRIS} a_\CRIS^\ast\calE$ is a quasi-isomorphism. 
Arguing as in \cite[p.~347]{Berthelot-book}, we see that the latter map is a quasi-isomorphism if and only if \eqref{eq:crystalline base change} is a quasi-isomorphism for each $(U_{i_0\cdots i_l}',M_{U_{i_0\cdots i_l}'})\rightarrow (U_{i_0\cdots i_l},M_{U_{i_0\cdots i_l}})$ in place of $g'\colon (X',M_{X'})\rightarrow (X,M_X)$. 
Recall that $\Sigma_0$ is affine and $X$ is quasi-separated. 
If we take a finite affine open covering $X=\bigcup U_l$, then each $U_{i_0\cdots i_l}$ becomes separated (see \cite[pp.~347-348]{Berthelot-book}). So we may reduce to the case where $X$ is separated. In this case, taking a finite affine open covering makes each $U_{i_0\cdots i_l}$ affine. Hence we may assume that $X$ is affine.

Since $f_0$ is a smooth morphism between fine log schemes and $\calJ$ is a nil ideal, it follows from \cite[Lem.~2.3.14]{nakkajima-shiho} (a generalization of \cite[Prop.~3.14(1)]{Kato-log} in the case of $\calJ$ being not necessarily nilpotent) that there exist a smooth morphism $f\colon (Z,M_Z)\rightarrow (\Sigma,M_\Sigma)$ and an isomorphism $(X,M_X)\cong (Z,M_Z)\times_{(\Sigma,M_\Sigma)}(\Sigma_0,M_{\Sigma_0})$ over $(\Sigma_0,M_{\Sigma_0})$.
Moreover, $f$ is integral since $f_0$ is integral: recall from \cite[Prop.~4.1(2)]{Kato-log} that the integrality of $f$ is equivalent to the condition that the monoid homomorphism $f^{-1}_\et (M_\Sigma/\calO_\Sigma^\ast)_{\overline{x}}\rightarrow (M_Z/\calO_Z^\ast)_{\overline{x}}$ is integral for each geometric point $\overline{x}$ of $Z$. In our case, the latter map is naturally identified with $f^{-1}_{0,\et} (M_{\Sigma_0}/\calO_{\Sigma_0}^\ast)_{\overline{x}}\rightarrow (M_X/\calO_X^\ast)_{\overline{x}}$ by regarding $\overline{x}$ as a geometric point of $X$. In particular, the underlying morphism $Z\rightarrow \Sigma$ is flat by \cite[Cor.~4.5]{Kato-log}. It follows that the PD-structure $\gamma$ on $\calJ$ extends to $Z$. Since $i$ is an exact closed immersion defined by a sub PD-ideal of $\calJ$, we conclude that the PD-envelope of $(X,M_X)\hookrightarrow (Z,M_Z)$ relative to $\Sigma^\sharp$ is $(Z,M_Z)$.
We set $(Z',M_{Z'})\coloneqq (Z,M_Z)\times_{(\Sigma,M_\Sigma)}(\Sigma',M_{\Sigma'})$. Then the PD-envelope of $(X',M_{X'})\hookrightarrow (Z',M_{Z'})$ relative to $\Sigma'^\sharp$ is $(Z',M_{Z'})$.

Let $(E,\nabla)$ denote the integrable log connection on $Z_\et$ corresponding to $\calE$. 
By Theorem~\ref{thm:cohomology of crystal in terms of log de Rham complex}, we have
\[
R\pi_{X/\Sigma,\ast}\calE\cong Rf_{\et,\ast}(E\otimes\omega^\bullet_{(Z,M_Z)/(\Sigma,M_\Sigma)},\nabla)\cong f_{\et,\ast}(E\otimes\omega^\bullet_{(Z,M_Z)/(\Sigma,M_\Sigma)},\nabla),
\]
where the second quasi-isomorphism follows from the the fact that $f$ is affine.
Write $g'_Z\colon (Z', M_{Z'}) \rightarrow (Z, M_Z)$ for the projection. Then the crystal $g_{\CRIS}'^\ast\calE$ corresponds to $(E',\nabla')\coloneqq g_{Z,\et}'^\ast(E,\nabla)$. Hence $R\pi_{X'/\Sigma',\ast}g_{\CRIS}'^\ast\calE\cong f'_{\et,\ast}( E'\otimes\omega^\bullet_{(Z',M_{Z'})/(\Sigma',M_{\Sigma'})},\nabla')$ by Theorem~\ref{thm:cohomology of crystal in terms of log de Rham complex}; strictly speaking, the base log PD-scheme in Set-up~\ref{set-up:description of crystals in terms of connections} is assumed to be fine, whereas $M_{\Sigma'}$ is only assumed to be integral and quasi-coherent. However, $f'$ is the base change of $f$ with $M_{\Sigma}$ being fine, so one can check that the results in \S~\ref{sec:Embedding into a smooth log scheme} continue to hold for $f'$.

Consider the diagram
\[
\xymatrix{
Lg^\ast_\et R\pi_{X/\Sigma,\ast}\calE\ar[d]_-{\text{\eqref{eq:crystalline base change}}}\ar[r]^-\cong &Lg^\ast_\et f_{\et,\ast}(E\otimes\omega^\bullet_{(Z,M_Z)/(\Sigma,M_\Sigma)},\nabla)\ar[d] \\
R\pi'_{X'/\Sigma',\ast}g'^\ast_\CRIS\calE\ar[r]^-\cong & f'_{\et,\ast}( E'\otimes\omega^\bullet_{(Z',M_{Z'})/(\Sigma',M_{\Sigma'})},\nabla')\\
}
\]
where the right vertical map is given by adjunction.
We deduce from Theorem~\ref{thm:compatibility of cohomology of crystal in terms of log de Rham complex and pushforward} via adjunction that this is commutative. Since $\calE$ is flat, $E$ is a flat $\calO_Z$-module and thus a flat $f^{-1}_\et\calO_\Sigma$-module. Moreover, $f$ is a smooth morphism of fine log schemes. Hence $f_{\et,\ast}(E\otimes\omega^\bullet_{(Z,M_Z)/(\Sigma,M_\Sigma)},\nabla)$ is a complex of flat quasi-coherent $\calO_\Sigma$-modules.
We conclude from the base change theorem of quasi-coherent sheaves for affine morphisms (in the small \'etale sites) that the right vertical map in the above diagram is a quasi-isomorphism, and thus so is \eqref{eq:crystalline base change}.
\end{proof}

\begin{thm}[{cf.~\cite[Thm.~6.10]{Kato-log}}]\label{thm:crystalline base change for crystalline pushforward}
Suppose that we have the following diagram
\[
\xymatrix{
(X',M_{X'})\ar[r]^-{g'} \ar[d]_-{f'}& (X,M_X)\ar[d]_-f\\
(Y',M_{Y'})\ar[r]^-g \ar[d]& (Y,M_Y)\ar[d]\\
\Sigma'^\sharp=(\Sigma',M_{\Sigma'},\calJ',\gamma') \ar[r]& \Sigma^\sharp=(\Sigma,M_\Sigma,\calJ,\gamma),
}
\]
where the upper-square is Cartesian, $M_\Sigma$ is fine, and $M_{\Sigma'}$ is integral and quasi-coherent.
Assume that $f\colon (X,M_X)\rightarrow (Y,M_Y)$ is smooth and integral, $f\colon X\rightarrow Y$ is qcqs, and $p$ is nilpotent in $\calO_\Sigma$.
Then for every flat quasi-coherent $\calO_{X/\Sigma}$-module $\calE$ on $((X,M_X)/\Sigma^\sharp)_\CRIS$, the canonical morphism
\[
g^\ast_\CRIS Rf_{\CRIS,\ast}\calE\rightarrow Rf'_{\CRIS,\ast}g'^\ast_\CRIS\calE
\]
is a quasi-isomorphism.
\end{thm}

Since we work on the big sites, we have $Lg^\ast_\CRIS=g^\ast_\CRIS$, and the proof of this theorem is simpler than that of \cite[Thm.~V.3.5.1]{Berthelot-book}. 

\begin{proof}
Take any $(U',T', M_{T'})\in ((Y',M_{Y'})/\Sigma'^\sharp)_\CRIS$ and let $(U,T,M_T)\coloneqq (U',T',M_{T',g})\in ((Y,M_Y)/\Sigma^\sharp)_\CRIS$ denote the object as in Construction~\ref{construction:strict log structure for functoriality} (relative to $\Sigma'^\sharp\rightarrow \Sigma^\sharp$). We see from Lemma~\ref{lem:evalutation of crystalline pushforward via projection} that the evaluation
\[
(g^\ast_\CRIS Rf_{\CRIS,\ast}\calE)_{(U',T')}\rightarrow (Rf'_{\CRIS,\ast}g'^\ast_\CRIS\calE)_{(U',T')}
\]
is naturally identified with
\begin{equation}\label{eq:base change map in proof}
R\pi_{f^{-1}(U)/T,\ast}(\calE|_{((f^{-1}(U),M_{f^{-1}(U)})/T^\sharp)_\CRIS})\rightarrow R\pi_{f^{-1}(U')/T',\ast}((g_\CRIS'^\ast\calE)|_{((f^{-1}(U'),M_{f^{-1}(U')})/T'^\sharp)_\CRIS}).
\end{equation}

It suffices to show that \eqref{eq:base change map in proof} is a quasi-isomorphism under the additional assumption that $T'$ is quasi-compact (by taking an affine covering).
Consider the diagram
\[
\xymatrix{
(f^{-1}(U'),M_{f^{-1}(U')})\ar[r] \ar[d]& ((f^{-1}(U),M_{f^{-1}(U)})\ar[d]\\
(U',M_{U'})\ar[r] \ar@{^{(}->}[d]& (U,M_U)\ar@{^{(}->}[d]\\
T'^\sharp \ar[r]^-g& T^\sharp.
}
\]
Here the upper diagram is Cartesian, and $g$ induces an isomorphism of the underlying schemes by construction. Hence we deduce from Theorem~\ref{thm:crystalline base change for projection} that \eqref{eq:base change map in proof} is a quasi-isomorphism.
\end{proof}

\section{Crystalline derived pushforward of finite locally free modules}\label{sec:Crystalline derived pushforward of finite locally free crystals}

\begin{thm}\label{thm:crystalline derived pushforward of perfect complexes}
Let $\Sigma^\sharp$ be a fine log PD-scheme in which $p$ is nilpotent, and let $f\colon (X,M_X)\rightarrow (Y,M_Y)$ be a smooth and integral morphism of log $\Sigma^\sharp$-schemes such that $f\colon X\rightarrow Y$ is proper.
For every finite locally free $\calO_{X/\Sigma}$-module $\calE$, the derived pushforward $Rf_{\CRIS,\ast}\calE$ is a perfect complex of $\calO_{Y/\Sigma}$-modules.
The same holds for the small crystalline sites.
\end{thm}

See \cite[Thm.~13]{Faltings-very-ramified} and \cite[07MX]{stacks-project} for relevant results in the non-log case; Theorem~\ref{thm:crystalline derived pushforward of perfect complexes} for $\calE=\calO_{X/\Sigma}$ is proved in \cite[1.11, Thm.~(ii)]{Beilinson-crystalline-period-map} under the additional assumption that $f$ is of Cartier type via the Cartier isomorphism.
The theorem in the above generality is essentially due to Faltings;  \cite{faltings-almostetale} sketches the proof when $Y$ is the formal spectrum of a $p$-adic DVR (in the framework of \S~\ref{sec:Big absolute crystalline sites}). Below, we give a detailed proof, following \cite[pp.~247-248]{faltings-almostetale}.

\begin{proof}
We treat the big crystalline site case, as the small crystalline site case is proved in a similar way.
We see from Lemma~\ref{lem:evalutation of crystalline pushforward via projection}(2) and Theorem~\ref{thm:crystalline base change for projection} that it suffices to show that, for each affine $(U,T)\in ((Y,M_Y)/\Sigma^\sharp)_\CRIS$, the object $(Rf_{\CRIS,\ast}\calE)_{(U,T)}\in D^b(T_\et,\calO_T)$ is perfect.

We can also deduce from Propositions~\ref{prop:vanishing of crystal under projection} and \ref{prop:quasi-coherence of crystal under projection} that $(R^if_{\CRIS,\ast}\calE)_{(U,T)}$ is a quasi-coherent $\calO_T$-module for each $i$ and vanishes $i\gg 0$.
By \cite[07LU]{stacks-project}, it suffices to prove that $(Rf_{\CRIS,\ast}\calE)_{(U,T)}\otimes_{\Z}^\mathbf{L}\F_p$ is perfect, and the latter is identified with $(Rf_{\CRIS,\ast}\calE)_{(U_1,T_1)}$, where $(U_1,T_1)\coloneqq (U,T)\times_\Z \F_p$, again by Theorem~\ref{thm:crystalline base change for projection}. This means that we may assume that $\Sigma$ is an $\F_p$-scheme. 
Hence Lemma~\ref{lem:evalutation of crystalline pushforward via projection}(2) reduces the theorem to Proposition~\ref{prop:crystalline derived pushforward of perfect complex in characteristic p} below.
\end{proof}

We now assume that $\Sigma$ is an $\F_p$-scheme.
Fix an affine $(U,T)\in ((Y,M_Y)/\Sigma^\sharp)_\CRIS$ and set $f_U\colon (X_U,M_{X_U})\coloneqq (X,M_X)\times_{(Y,M_Y)}(U,M_U)\rightarrow (U,M_U)$. Note that $X_U= X\times_YU$ as schemes. Write $\pi$ for the morphism $X_U\rightarrow T$.
Recall the morphisms of topoi
\[
\xymatrix{
\Sh(((X_U,M_{X_U})/T^\sharp)_\CRIS)\ar[r]^-{u_{X_U/T}}\ar@/_15pt/[rr]_-{\pi_{X_U/T}} & \Sh(X_{U,\et})\ar[r]^-{\pi_\et}&\Sh(T_\et).
}
\]

\begin{prop}\label{prop:crystalline derived pushforward of perfect complex in characteristic p}
For every finite locally free $\calO_{X_U/T}$-module $\calE$ on $((X_U,M_{X_U})/T^\sharp)_\CRIS$, $R\pi_{X_U/T,\ast}\calE\in D^+(T_\et,\calO_T)$ is a perfect complex. The same holds for the small crystalline site.
\end{prop}

For a fine log scheme $(Z,M_Z)$ over $\F_p$, let $F_Z\coloneqq F_{(Z,M_Z)}\colon (Z,M_Z)\rightarrow (Z,M_Z)$ denote the absolute $p$th Frobenius. Since $F_U\circ f_U=f_U\circ F_{X_U}$, the map $f_U$ factors through the relative Frobenius as $f_U\colon (X_U,M_{X_U})\xrightarrow{F_{X_U/U}}(X_U',M_{X_U'})\xrightarrow{f_U'}(U,M_U)$. Moreover, since $(U,M_U)\hookrightarrow (T,M_T)$ is a PD-immersion, the Frobenius $F_T$ factors as $(T,M_T)\xrightarrow{F_{(U,T)}}(U,M_U)\hookrightarrow (T,M_T)$.
Define $(X_T',M_{X_T'})$ to be the fiber product $(X_U,M_{X_U})\times_{(U,M_U),F_{(U,T)}}(T,M_T)$. We name the morphisms as in the following commutative diagram of fine log schemes (with log structures omitted for simplicity), where the squares are Cartesian (even as schemes since $f$ is integral),
\[
\xymatrix{
X_U \ar[r]_-{F_{X_U/U}}\ar@/_10pt/[rd]_-{f_U}\ar@/^30pt/[rrr]_(.7){F_{X_U}}\ar@/^20pt/[rr]_-{F_{X_U/(U,T)}} &X_U'\ar[r]\ar[d]^-{f_U'}&X_T'\ar[r]\ar[d]^-{f_T'}& X_U\ar[d]^-{f_U}\ar[rd]^-{\pi} & \\
&U\ar[r]\ar@/_20pt/[rr]^-{F_{U}}\ar@{^{(}->}[r]&T\ar[r]^-{F_{(U,T)}}\ar[r]\ar@/_20pt/[rr]^-{F_{T}} & U \ar@{^{(}->}[r] & T.
}
\]

\begin{lem}\label{lem:crystalline structure sheaf and Frobenius}
For every $(V,Q)\in ((X_U,M_{X_U})/T^\sharp)_\CRIS$, the structure morphism $Q\rightarrow T$ uniquely factors through $f_T'\colon X_T'\rightarrow T$. In particular, the structure sheaf $\calO_{X_U/T}$ of $((X_U,M_{X_U})/T^\sharp)_\CRIS$ is an algebra over $(F_{X_U/(U,T),\et}\circ u_{X_U/T})^{-1}\calO_{X_T'}$.
\end{lem}

\begin{proof}
The structure morphism $V\rightarrow X_U\rightarrow U$ is compatible with the absolute Frobenii $F_V$ and $F_U$. Since $F_V$ and $F_U$ factor uniquely as $F_V\colon V\hookrightarrow Q\rightarrow V$ and $F_U\colon U\hookrightarrow T\rightarrow U$, we conclude that $Q\rightarrow T$ uniquely factors as $Q\rightarrow X_T'\rightarrow T$.
\end{proof}

\begin{proof}[Proof of Proposition~\ref{prop:crystalline derived pushforward of perfect complex in characteristic p}]
We first work on the big crystalline site case. We are going to show that $R\pi_{X_U/T,\ast}\calE$ is both pseudo-coherent and has finite Tor-dimension (see \cite[08G8]{stacks-project}).

For the pseudo-coherence, note that  $R\pi_{X_U/T,\ast}\calE=Rf_{T,\et,\ast}'RF_{X_U/(U,T),
\et,\ast}Ru_{X_U/T,\ast}\calE$ and the complex $RF_{X_U/(U,T),\ast}Ru_{X_U/T,\ast}\calE$ is an object of $D^+((X_T')_\et,\calO_{X_T'})$ by Lemma~\ref{lem:crystalline structure sheaf and Frobenius}.
Since $f_T'$ is flat and proper, it is enough to show that $RF_{X_U/(U,T),\ast}Ru_{X_U/T,\ast}\calE$ is pseudo-coherent (see \cite[0CSD]{stacks-project}\footnote{Or, more precisely, see \cite[0CTN]{stacks-project} as we work on the \'etale site.}). The latter property can be checked \'etale locally on $X_T'$, and the above formulation is compatible with \'etale localization.
So by \cite[Lem.~2.3.14]{nakkajima-shiho}, we replace $X_U$ with an \'etale cover and assume that $X_U$ is affine and there exists a smooth and integral morphism $f_T\colon (X_T,M_{X_T})\rightarrow (T,M_T)$ together with an isomorphism $(X_U,M_{X_U})\cong (X_T,M_{X_T})\times_{(T,M_T)}(U,M_U)$.
Note that the morphism $F_{X_U/(U,T)}$ agrees with the composite of $X_U\hookrightarrow X_T$ and the relative Frobenius $F_{X_T/T}\colon X_T\rightarrow X_T'$.

Let $(E,\nabla\colon E\rightarrow E\otimes\omega_{X_T/T}^1)$ denote a finite locally free $\calO_{X_T}$-module with quasi-nilpotent integrable log connection corresponding to $\calE$ on $((X_U,M_{X_U})/T^\sharp)_\CRIS$ by Proposition~\ref{prop: crystals and quasi-nilpotent connections}. 
Consider the log de Rham complex $(E\otimes_{\calO_Z}\omega^\bullet_{Z/\Sigma},\nabla)$ on $X_{T,\et}$. Since $F_{X_T/T}$ is affine, the derived pushforward $RF_{X_T/T,\et,\ast}(E\otimes_{\calO_Z}\omega^\bullet_{Z/\Sigma},\nabla)$ is represented by the complex $F_{X_T/T,\et,\ast}(E\otimes_{\calO_Z}\omega^\bullet_{Z/\Sigma},\nabla)$, and the local computation shows that the latter is a complex of $\calO_{X_T'}$-modules. We know from Theorem~\ref{thm:cohomology of crystal in terms of log de Rham complex} that there is a natural quasi-isomorphism 
$Ru_{X_U/T,\ast}\calE\cong (E\otimes_{\calO_Z}\omega^\bullet_{Z/\Sigma},\nabla)$
 in $D^+(X_{U,\et},\pi^{-1}_\et\calO_T)$ under the equivalence $\Sh(X_{U,\et})\cong \Sh(X_{T,\et})$. Moreover, we can check from the proof that $RF_{X_U/(U,T),\ast}Ru_{X_U/T,\ast}\calE\in D^+((X_T')_\et,\calO_{X_T'})$ is represented by $F_{X_T/T,\et,\ast}(E\otimes_{\calO_Z}\omega^\bullet_{Z/\Sigma},\nabla)$.
We are going to use a limit argument to show the pseudo-coherence of the latter complex.

Write $T=\Spec A$ and $X_T=\Spec B$. Set $B'=B\otimes_{A,F_A}A$ so that $X_T'=\Spec B'$. By assumption, $B$ and $B'$ are both flat of finite presentation over $A$. Hence the relative Frobenius $B'\rightarrow B$ is finite. Let $M$ (resp.~$\omega_{B/A}^i$) denote the finite locally free $B$-module corresponding to $E$ (resp.~$\omega_{X_T/T}^i$). 
We need to show that the complex $M\otimes_B\omega_{B/A}^\bullet$ of finitely generated $B'$-modules is pseudo-coherent by \cite[08E7]{stacks-project}.\footnote{In fact, $F_{X_T/T}$ is finite and flat in the non-log case, in which case $F_{X_T/T,\et,\ast}(E\otimes_{\calO_Z}\omega^\bullet_{Z/\Sigma},\nabla)$ is perfect and the proof ends here.}

Write $A=\bigcup_\lambda A_\lambda$ where each $A_\lambda$ is a sub $\Z$-algebra of finite type (and thus Noetherian). It follows from the standard limit argument that for a sufficiently large $\lambda$, there exists an $A_\lambda$-algebra $B_\lambda$ flat of finite presentation together with an isomorphism $B_\lambda\otimes_{A_\lambda}A\cong B$. If we set $B_\lambda'\coloneqq B_\lambda\otimes_{A_\lambda,F_{A_\lambda}}A_\lambda$, we have the relative Frobenius $B_\lambda'\rightarrow B_\lambda$ and the canonical identification $B_\lambda'\otimes_{A_\lambda}A\cong B'$.
By increasing $\lambda$ if necessary, we may also assume that $M$ and $\omega_{B/A}^1$ descend to finite locally free $B$-modules $M_\lambda$ and $\omega_\lambda$, and the complex of finitely generated $B'$-modules $M\otimes_B\omega_{B/A}^i$ descends to a complex of finitely generated $B_\lambda'$-modules $M_\lambda'^\bullet\coloneqq M_\lambda\otimes_{B_\lambda}\bigwedge^\bullet_{B_\lambda}\omega_{\lambda}$. Since $B_\lambda'$ is Noetherian and $B_\lambda'\rightarrow B_\lambda$ is finite, we know from \cite[066E]{stacks-project} that $M_\lambda'^\bullet$ is a pseudo-coherent complex of $B_\lambda'$-modules. 
The natural map of complex of $B_\lambda'$-modules $M_\lambda'^\bullet\rightarrow M\otimes_B\omega_{B/A}^\bullet$ induces a map $M_\lambda'^\bullet\otimes_{B_\lambda'}^L B'\rightarrow M\otimes_B\omega_{B/A}^\bullet$ in $D(B')$, and it suffices to show that the latter map is an isomorphism by \cite[0650]{stacks-project}. Since $B_\lambda'$ is flat over $A_\lambda$ and $B'=B_\lambda'\otimes_{A_\lambda}A$, one can check that the map in question is identified with $M_\lambda'^\bullet\otimes_{A_\lambda}^L A\rightarrow M\otimes_B\omega_{B/A}^\bullet$ (e.g.~by taking a $K$-flat resolution of $M_\lambda'^\bullet$ as $B_\lambda'$-modules). As each term $M_\lambda'^i=M_\lambda\otimes_{B_\lambda}\bigwedge^i_{B_\lambda}\omega_{\lambda}$ is a finite locally free $B_\lambda$-module and thus flat over $A_\lambda$ and satisfies $M_\lambda'^i\otimes_{A_\lambda}A\cong M\otimes_B\omega_{B/A}^i$, we conclude that the latter map is an isomorphism. This completes the proof of the pseudo-coherence of $R\pi_{X_U/T,\ast}\calE$.

We turn to the finite Tor-dimensionality of $R\pi_{X_U/T,\ast}\calE$.
It is enough to represent $R\pi_{X_U/T,\ast}\calE$ as a finite complex of flat $\calO_T$-modules by \cite[08G1]{stacks-project}.
Take an open covering $X_U=V_0\cup\cdots \cup V_n$ by finitely many affines. Then each $(V_i,M_{V_i})$ lifts to a fine log scheme $(Z_i,M_{Z_i})$ smooth and integral over $(T,M_T)$ by \cite[Lem.~2.3.14]{nakkajima-shiho}.

Let $\Delta_n$ denote the opposite category of the category consisting of non-empty subsets of $\{0,\ldots,n\}$ with morphisms being inclusions.
For each $\underline{i}=\{i_0<\cdots<i_l\}\in \Delta_n$, set $V_{\underline{i}}\coloneqq V_{i_0}\cap \cdots\cap V_{i_l}$ and $(Z_{\underline{i}},M_{Z_{\underline{i}}})\coloneqq (Z_{i_0},M_{Z_{i_0}})\times_{(T,M_T)} \cdots \times_{(T,M_T)} (Z_{i_l},M_{Z_{i_l}})$. Define $(D_{\underline{i}},M_{D_{\underline{i}}})$ to be the PD-envelope of the closed immersion $(V_{\underline{i}},M_{V_{\underline{i}}})\hookrightarrow(Z_{\underline{i}},M_{Z_{\underline{i}}})$ relative to $T^\sharp$. We note that $D_{\underline{i}}$ is flat over $T$; see \cite[Lem.~2.22, Pf.]{hyodo-kato}.

With the notation as in Definition~\ref{defn:localization of crystalline site by open}, we obtain the topos $\Sh(((V_\bullet,M_{V_\bullet})/T^\sharp)_\CRIS)$ associated to the diagram of topoi indexed by $\Delta_n$ associating to $\underline{i}$ the topos 
\[
\Sh(((X_U,M_{X_U})/T^\sharp)_\CRIS)/V_{\underline{i},\CRIS}\cong \Sh((V_{\underline{i}},M_{V_{\underline{i}}})/T^\sharp)_\CRIS).
\]
We have a commutative diagram of ringed topoi
\[
\xymatrix{
(\Sh(((V_\bullet,M_{V_\bullet})/T^\sharp)_\CRIS),\calO_{V_\bullet/T})\ar[r]^-{\pi_{V_\bullet/T_\bullet}}\ar[d]_-{a_\CRIS} & (\Sh(T_\et),\calO_T)\times\Delta_n\ar[d]^-b\\
(\Sh(((X_U,M_{X_U})/T^\sharp)_\CRIS),\calO_{X_U/T})\ar[r]^-{\pi_{X_U/T}} &(\Sh(T_\et),\calO_T),
}
\]
where $(\Sh(T_\et),\calO_T)\times\Delta_n$ denotes the constant ringed topos indexed by $\Delta_n$.
We know from \cite[Prop.~V.3.4.8]{Berthelot-book} that the adjunction map $\calE\rightarrow Ra_{\CRIS,\ast}a_{\CRIS}^\ast \calE$ is a quasi-isomorphism. Hence it is enough to show that $Rb_\ast R\pi_{V_\bullet/T_\bullet,\ast}a_{\CRIS}^\ast \calE$ has finite Tor-dimension.

Let $(E_{\underline{i}},\nabla\colon E_{\underline{i}}\rightarrow E_{\underline{i}}\otimes \omega_{Z_{\underline{i}}/T}^1)$ denote a finite locally free $\calO_{D_{\underline{i}}}$-module with quasi-nilpotent integrable log connection corresponding to $\calE$ on $((V_{\underline{i}},M_{V_{\underline{i}}})/T^\sharp)_\CRIS$ under Proposition~\ref{prop: crystals and quasi-nilpotent connections}. 
We know from Theorem~\ref{thm:cohomology of crystal in terms of log de Rham complex} that $R\pi_{V_{\underline{i}}/T,\ast}\calE\in D(T_\et,\calO_T)$ is represented by $(E_{\underline{i}}\otimes \omega_{Z_{\underline{i}}/T}^\bullet,\nabla)$ regarded as a complex  of $\calO_T$-modules (note that $T$, $Z_{\underline{i}}$, $D_{\underline{i}}$ are all affine). For each $\underline{i},\underline{j}\in \Delta_n$ with $\underline{i}\subset \underline{j}$, the corresponding projection $p_{\underline{j},\underline{i}}\colon D_{\underline{j}}\rightarrow D_{\underline{i}}$ yields a morphism of connections $p_{\underline{j},\underline{i}}^\ast(E_{\underline{i}},\nabla)\rightarrow (E_{\underline{j}},\nabla)$ on $D_{\underline{j}}$. We deduce from Theorem~\ref{thm:compatibility of cohomology of crystal in terms of log de Rham complex and pushforward} that the map $R\pi_{V_{\underline{i}}/T,\ast}\calE\rightarrow R\pi_{V_{\underline{j}}/T,\ast}\calE$ in $D(T_\et,\calO_T)$ is represented by the induced morphism of complexes of $\calO_T$-modules 
\[
q_{\underline{j},\underline{i}}\colon(E_{\underline{i}}\otimes \omega_{Z_{\underline{i}}/T}^\bullet,\nabla)\rightarrow (E_{\underline{j}}\otimes \omega_{Z_{\underline{j}}/T}^\bullet,\nabla).
\]
Combining this with \cite[Exp.~Vbis, Cor.~1.3.11]{SGA4-2},  we see  that $R\pi_{V_\bullet/T_\bullet,\ast}a_{\CRIS}^\ast \calE\in D((\Sh(T_\et),\calO_T)\times\Delta_n)$ is represented by the system $((E_{\underline{i}}\otimes \omega_{Z_{\underline{i}}/T}^\bullet,\nabla)_{\underline{i}\in \Delta_n}, (q_{\underline{j},\underline{i}})_{\underline{i}\subset \underline{j}})$.

Define a bounded double complex $C^{\star,\bullet}$ of $\calO_T$-modules by
\[
C^{l,m}=\prod_{\underline{i}=\{i_0<\cdots<i_l\}}E_{\underline{i}}\otimes \omega_{Z_{\underline{i}}/T}^m
\]
with differentials given by
\[
d_1^{l,m}\colon C^{l,m}\rightarrow C^{l+1,m},\quad d_1^{l,m}(x)_{\underline{j}}=\sum_{k=0}^{l+1}(-1)^k q_{\underline{j},\underline{j}\smallsetminus\{j_k\}}\quad \text{for}\quad \underline{j}=\{j_0<\cdots<j_{l+1}\}
\]
and $d_2^{l,m}=\prod_{\underline{i}}(-1)^m\nabla$ where $\nabla$ denotes $E_{\underline{i}}\otimes \omega_{Z_{\underline{i}}/T}^m\rightarrow E_{\underline{i}}\otimes \omega_{Z_{\underline{i}}/T}^{m+1}$. We claim that $Rb_\ast R\pi_{V_\bullet/T_\bullet,\ast}a_{\CRIS}^\ast \calE\in D(T_\et,\calO_T)$ is represented by the associated total complex $\operatorname{Tot}(C^{\star,\bullet})$. This follows from \cite[Prop.~V.3.4.9 (i)]{Berthelot-book} via the finite stupid filtration in the $\bullet$-direction; note that for a fixed $m$, the complex $C^{\star,m}$ agrees with the one defined in \cite[V.3.4.6]{Berthelot-book}. Finally, we conclude that $R\pi_{X_U/T,\ast}\calE\cong Rb_\ast R\pi_{V_\bullet/T_\bullet,\ast}a_{\CRIS}^\ast \calE$ has finite Tor-dimension because each $E_{\underline{i}}\otimes \omega_{Z_{\underline{i}}/T}^m$ is a finite locally free $\calO_{D_{\underline{i}}}$-module and $D_{\underline{i}}$ is flat over $T$.

The proof in the small crystalline case is similar with Theorem~\ref{thm:cohomology of crystal in terms of log de Rham complex} replaced by \cite[Thm.~6.4]{Kato-log}. 
\end{proof}

\section{Big crystalline sites in the \texorpdfstring{$p$}{p}-adic base and absolute cases}\label{sec:Big absolute crystalline sites}

We generalize the theory of big crystalline sites to allow a $p$-adic base as in \cite[\S~1.12]{Beilinson-crystalline-period-map}.
A \emph{$p$-adic log PD-formal scheme} consists of a quadruple $\calS^\sharp\coloneqq (\calS,M_\calS,\calJ_\calS,\gamma_\calS)$ where $(\calS,M_\calS)$ is a $p$-adic formal log scheme, $\calJ_\calS\subset \calO_\calS$ is a quasi-coherent ideal sheaf containing $p$, and $\gamma_\calS$ is a PD-structure on $\calJ_\calS$. For each $n\geq 1$, write $S_n^\sharp$ for the log PD-scheme $\calS^\sharp\times_{\Spf \Z_p}\Spec \Z_p/p^n$ (in the obvious sense).

Fix a $p$-adic log PD-formal scheme $\calS^\sharp$ such that $M_\calS$ is quasi-coherent.

\begin{defn}\label{def:big absolute log crystallien site}
Let $(Y,M_Y)$ be a log $S_{n_0}^\sharp$-scheme for some $n_0\geq 1$ and $M_Y$ is integral and quasi-coherent.
Define \emph{the big relative logarithmic crystalline site} 
\[
((Y,M_Y)/\calS^\sharp)_\CRIS
\]
to be the colimit over $n\geq n_0$ of the sites $((Y,M_{Y})/S_n^\sharp)_\CRIS$ (cf.~ \cite[1.12]{Beilinson-crystalline-period-map} and \cite[0EXI]{stacks-project}): concretely, an object of $((Y,M_Y)/\calS^\sharp)_\CRIS$ is an object of $((Y,M_Y)/S_n^\sharp)_\CRIS$ for some $n\geq n_0$. 
When $\calS^\sharp=S_n^\sharp$ for some $n$, this is consistent with Definition~\ref{def:big relative logarithmic crystalline site}.

As in the relative case, it has the structure sheaf $\calO_{Y/\calS}$ and an ideal sheaf $\calJ_{Y/\calS}$, and an obvious analogue of Proposition~\ref{prop:quasi-coherent crystals} holds.
We also have a canonical morphism of topoi
\[
u_{Y/\calS}\colon \Sh(((Y,M_Y)/\calS^\sharp)_\CRIS)\rightarrow \Sh(Y_\et).
\]

When $\calS^\sharp=\Z_p^\sharp$, namely, $\Spf \Z_p$ together with the trivial log structure and the canonical PD-structure on the ideal $(p)$, we simply write $(Y,M_Y)_\CRIS$, $\calO_{Y/\Z_p}$ and so on. We call $(Y,M_Y)_\CRIS$
\emph{the big absolute logarithmic crystalline site} . In this case, $(Y,M_Y)$ is any log scheme over $\Z_p$ such that $p^{n_0}\calO_Y=0$ for some $n_0\geq 1$ and $M_Y$ is integral and quasi-coherent.
\end{defn}

\begin{eg}\label{eg:absolute crystalline site over k}
Let $k$ be a perfect field of characteristic $p$ and assume that $(Y,M_Y)$ is a log scheme over $k$. Then for every $(U,T)\in (Y,M_Y)_\CRIS$, the sheaf $\calO_T$ has the canonical $W(k)$-algebra structure. In particular, $(Y,M_Y)_\CRIS$ coincides with $((Y,M_Y)/W(k)^\sharp)$, where $W(k)^\sharp$ is $\Spf W(k)$ equipped with the trivial log structure and the canonical PD-structure on $(p)$.
\end{eg}

\begin{rem}\label{rem:absolute small crystalline site}
Let $((Y,M_Y)/\calS^\sharp)_\cris\subset ((Y,M_Y)/\calS^\sharp)_\CRIS$ denote the full subcategory consisting of objects $(U,T)$ such that $U$ is \'etale over $Y$. Then $((Y,M_Y)/\calS^\sharp)_\cris$ agrees with the crystalline site introduced by Beilinson in \cite[\S~1.12]{Beilinson-crystalline-period-map}.
As in Remark~\ref{rem:small crystalline site}, there are morphisms of topoi
\[
r_{Y/\calS}\colon \Sh(((Y,M_Y)/\calS^\sharp)_\cris)\rightleftarrows \Sh(((Y,M_Y)/\calS^\sharp)_\CRIS)\colon p_{Y/\calS}
\]
such that $r_{Y/\calS}^\ast=p_{Y/\calS,\ast}$ is the restriction from $((Y,M_Y)/\calS^\sharp)_\CRIS$ to $((Y,M_Y)/\calS^\sharp)_\cris$. If we continue to write $\calO_{Y/\calS}$ for the restriction $r_Y^\ast\calO_{Y/\calS}$ to $((Y,M_Y)/\calS^\sharp)_\cris$, then the above morphisms define morphisms of ringed topoi for $\calO_{Y/\calS}$, similar to \cite[Cor.~III.2.2.4]{Berthelot-book}. It is straightforward to check that the restriction functor
\[
r_{Y/\calS}^\ast=Lr_{Y/\calS}^\ast\colon D_\perf(((Y,M_Y)/\calS^\sharp)_\CRIS) \rightarrow D_\perf(((Y,M_Y)/\calS^\sharp)_\cris)
\]
between the categories of perfect complexes of $\calO_{Y/\calS}$-modules (see \cite[08H6]{stacks-project}) is an equivalence and the quasi-inverse is given by $Lp_{Y/\calS^\sharp}^\ast$. One can also check that the category $D_\perf(((Y,M_Y)/\calS^\sharp)_\cris)$ defined as in \cite[08G5]{stacks-project} agrees with the category of perfect crystals defined in \cite[\S~1.11]{Beilinson-crystalline-period-map} if $Y$ is quasi-compact (as each object in the latter category is bounded by definition).
\end{rem}

The cohomology $R\Gamma(((Y,M_Y)/\calS^\sharp)_\CRIS,\calF)$ and the projection $Ru_{Y/\calS,\ast}\calF$ can be computed by the derived limits. 
This is explained for quasi-coherent modules in \cite[Prop.~7.22]{berthelot-ogus-book} and mentioned in \cite[\S~1.12]{Beilinson-crystalline-period-map}. Since we cannot find a good reference, let us spell out the arguments. 

We first introduce an auxiliary (fibered) site $((Y,M_Y)/S_\bullet^\sharp)_\CRIS$ as in \cite[7-22, 7-23]{berthelot-ogus-book}: the objects are the pairs $(n,T)$ where $n\in\N$\footnote{Strictly speaking, we only consider $n\geq n_0$ with $n_0$ as in Definition~\ref{def:big absolute log crystallien site}.} and $T\in ((Y,M_Y)/S_n^\sharp)_\CRIS$; the morphisms $(n',T')\rightarrow (n,T)$ are the obvious ones if $n'\leq n$ and void if $n'>n$; the covering families of $(n,T)$ are $\{(n,T_i)\rightarrow (n,T)\}$ for a cover $\{T_i\rightarrow T\}$ in $((Y,M_Y)/S_n^\sharp)_\CRIS$. A sheaf on $((Y,M_Y)/S_\bullet^\sharp)_\CRIS$ can be described as a collection $(\calF_n)_n$ of sheaves with $\calF_n\in \Sh(((Y,M_Y)/S_n^\sharp)_\CRIS)$ together with the natural compatibility maps.

We have a commutative diagram of topoi
\[
\xymatrix{
\Sh(((Y,M_Y)/S_n^\sharp)_\CRIS)\ar[r]^-{j_n}\ar[rd]_-{i_n} & \Sh(((Y,M_Y)/S_\bullet^\sharp)_\CRIS)\ar[d]^-j\\ &\Sh(((Y,M_Y)/\calS^\sharp)_\CRIS)
}
\]
defined as follows. 
The functor $i_n^\ast$ is the restriction.
For $\calF\in \Sh(((Y,M_Y)/\calS^\sharp)_\CRIS)$, $(i_{n,\ast}\calF)(U,T)=\calF(U,T_n)$, where $(T_n,M_{T_n})\coloneqq (T,M_T)\times_{\Spec \Z_p}\Spec \Z_p/p^n$ (obviously, $(U,T_n)$ defines an object of $((Y,M_Y)/S_n^\sharp)_\CRIS$). The functors $j_n^\ast$ and $j_{n,\ast}$ are defined similarly. 
We set $j^\ast(\calF)=(i_n^\ast\calF)_n$ for $\calF\in \Sh(((Y,M_Y)/\calS^\sharp)_\CRIS)$ and define $j_\ast(\calF_n)_n$ for $(\calF_n)_n\in \Sh(((Y,M_Y)/S_\bullet^\sharp)_\CRIS)$ by $j_\ast(\calF_n)_n(U,T)=\varprojlim_n \calF_n(n,T_n)$. It is straightforward to check that they define morphisms of topoi and $j\circ j_n=i_n$.

For a site $\calC$, let $\calC\times \N$ be the site defined as in \cite[0940]{stacks-project} so that the sheaves on $\calC\times\N$ are the inverse systems of sheaves on $\calC$ and there is a morphism of topoi $\pi\colon \Sh(\calC\times\N)\rightarrow \Sh(\calC)$ where $\pi_\ast$ takes the inverse limit and $\pi^\ast$ assigns the constant inverse system.

\begin{prop}\label{prop:j functor for colimit of sites}
\hfill
\begin{enumerate}
 \item If $\calI\in \Ab(((Y,M_Y)/\calS^\sharp)_\CRIS)$ is injective, then $j^\ast\calI$ is $j_\ast$-acyclic, and the adjunction map $\calI\rightarrow j_\ast j^\ast \calI$ is an isomorphism.
 \item For every $\calF\in \Ab(((Y,M_Y)/\calS^\sharp)_\CRIS)$, the adjunction map $\calF\rightarrow Rj_\ast j^\ast\calF$ is a quasi-isomorphism.
\end{enumerate}
\end{prop}

\begin{proof}
For the first statement of (1), we need to show that for every $q\geq 1$ and $(U,T)\in ((Y,M_Y)/\calS^\sharp)_\CRIS$, the sheaf $(R^qj_\ast j^\ast\calI)_{(U,T)}$ on $T_\et$ is zero. 
Consider a diagram 
\[
\xymatrix{
\Ab(((Y,M_Y)/S_\bullet^\sharp)_\CRIS)\ar[d]_-{j_\ast}\ar[r]^-\alpha & \Ab(T_\et\times \N)\ar[d]^-{\pi_{T,\ast}}\\
\Ab(((Y,M_Y)/\calS^\sharp)_\CRIS)\ar[r]^-\beta & \Ab(T_\et)
}
\]
where $j_\ast$ and $\pi_{T,\ast}$ are the pushforwards of the corresponding morphisms of topoi; $\alpha((\calF_n)_n)\coloneqq ((\calF_n)_{(U,T_n)})_n$ and $\beta(\calF)\coloneqq \calF_{(U,T)}$ under the natural identification $(T_n)_\et=T_\et$. It is straightforward to see $\pi_{T,\ast}\circ \alpha=\beta\circ j_\ast$. We also claim that $\alpha$ and $\beta$ are exact and send injectives to injectives. For $\beta$, observe that it is obtained as $\beta=\varphi_{T,\ast}\circ j_T^\ast$ from the morphisms of topoi
\[
\Sh(((Y,M_Y)/\calS^\sharp)_\CRIS)\xleftarrow{j_T} \Sh(((Y,M_Y)/\calS^\sharp)_\CRIS/(U,T))\xrightarrow{\varphi_T} \Sh(T_\et)
\]
where $j_T$ is the localization by $(U,T)$ and $\varphi_T$ is defined by $\varphi_{T,\ast}(\calH)=\calH_{\id_{(U,T)}}$ for $\calH=(\calH_u)_u$ and $\varphi_T^\ast(\calK)=(u^\ast\calK)_u$ (here $(u\colon (U',T')\rightarrow (U,T))$ runs over objects of the localized site $((Y,M_Y)/\calS^\sharp)_\CRIS/(U,T)$ and $u^\ast$ denotes $\Sh(T_\et)\rightarrow \Sh(T_\et')$): see also \cite[Prop.~5.26]{berthelot-ogus-book}. Since $j_T^\ast$ has an exact left adjoint $j_{T,!}$ by a general nonsense and $\varphi_{T,\ast}$ is easily seen to be exact, $\beta$ is exact and sends injectives to injectives. One can run a similar argument for $\alpha$ by using the localized topos $\Sh(((Y,M_Y)/S_\bullet^\sharp)_\CRIS)/j^\ast h_{(U,T)}$ where $h_{(U,T)}$ denotes the sheaf represented by $(U,T)$.

By the claim, we only need to show $R^q\pi_{T,\ast}(\alpha(j^\ast\calI))=0$ for $q\geq 1$. By definition, $\alpha(j^\ast\calI)$ is the inverse system $(\calI_{(U,T_n)})_n$. We show that each transition map is a surjection of \'etale presheaves. To see this, take any $T'\in T_\et$ and set $U'\coloneqq U\times_TT'$.
Then $(U',T')$ with the induced log and PD structures from $(U,T)$ defines an object of $((Y,M_Y)/\calS^\sharp)_\CRIS$. The morphism $(U',T'_{n})\rightarrow  (U',T'_{n+1})$ is a monomorphism in $((Y,M_Y)/\calS^\sharp)_\CRIS$. Hence
\[
\calI_{(U,T_{n+1})}(T')=\calI(U',T'_{n+1})\rightarrow \calI_{(U,T_n)}(T')=\calI(U',T'_n)
\]
is surjective by the injectivity of $\calI$ and \cite[093X]{stacks-project}. Hence $R^q\pi_{T,\ast}(\alpha(j^\ast\calI))=0$ for $q\geq 1$ by \cite[08TC, 0941]{stacks-project}. We also see
\[
(j_\ast j^\ast\calI)(U,T)=\Gamma(T,\pi_{T,\ast}(\alpha(j^\ast\calI)))=\varprojlim_n \calI_{(U,T_n)}(T_n)=\calI(U,T)
\]
as $(U,T)=(U,T_n)$ for $n\gg 0$, which proves the second statement.

Part (2) follows from (1) by taking an injective resolution $\calF\rightarrow \calI^\bullet$.
\end{proof}

\begin{prop}\label{prop:limit formula for crystalline cohomology}
For every $\calF\in\Ab(((Y,M_Y)/\calS^\sharp)_\CRIS)$, we have quasi-isomorphisms
\[
R\Gamma(((Y,M_Y)/\calS^\sharp)_\CRIS,\calF)\cong R\varprojlim R\Gamma(((Y,M_Y)/S_n^\sharp)_\CRIS, i_n^\ast\calF)
\]
and
\[
Ru_{Y/\calS,\ast}\calF\cong R\varprojlim Ru_{Y/S_n,\ast} i_n^\ast\calF.
\]
Here $R\varprojlim$ denotes the homotopy limit in $D(Y_\et)$ or $D(\Z)$ (cf.~\cite[0940]{stacks-project}).
\end{prop}

\begin{proof}
We follow \cite[Prop.~7.22, Pf.]{berthelot-ogus-book}. Consider the commutative diagram of morphisms of topoi
\[
\xymatrix{
\Sh(((Y,M_Y)/S_n^\sharp)_\CRIS)\ar[d]^-{j_n}\ar@/_75pt/[dd]_-{i_n}\ar[r]^-{u_{Y/S_n}} &  \Sh(Y_\et)\ar[r]\ar[d]^-{l_n} & \Sh(\mathrm{pt})\ar[d]^-{k_n}\\ 
\Sh(((Y,M_Y)/S_\bullet^\sharp)_\CRIS)\ar[d]^-j\ar[r]^-{v_{Y/S_\bullet}}\ar[rd]^-{u_{Y/S_\bullet}} & \Sh(Y_\et\times\N)\ar[r]\ar[d] & \Sh(\N)\ar[d] \\ 
\Sh(((Y,M_Y)/\calS^\sharp)_\CRIS)\ar[r]^-{u_{Y/\calS}} & \Sh(Y_\et)\ar[r] & \Sh(\mathrm{pt})
}
\]
where $u_{Y/S_\bullet}$, $v_{Y/S_\bullet}$, $l_n$, $k_n$ and unspecified ones are defined in an obvious way.
Observe that $j_n^\ast\colon \Ab(((Y,M_Y)/S_\bullet^\sharp)_\CRIS)\rightarrow \Ab(((Y,M_Y)/S_n^\sharp)_\CRIS)$ has an exact right adjoint $j_{n,!}$ given by
$(j_{n,!}\calG)_{n'}=\calG|_{((Y,M_Y)/S_{n'}^\sharp)_\CRIS}$ if $n'\leq n$ and $0$ if $n'>n$. Hence $j_n^\ast$ is exact and sends injectives to injectives. The similar properties hold for $l_n^\ast$ and $k_n^\ast$. Hence 
\[
l_n^\ast Rv_{Y/S_\bullet,\ast}j^\ast\calF\cong Ru_{Y/S_n,\ast}j_n^\ast j^\ast\calF.
\]
Combining this with Proposition~\ref{prop:j functor for colimit of sites}(2) and \cite[0941]{stacks-project}, we conclude
\[
Ru_{Y/\calS,\ast}\calF\cong Ru_{Y/S_\bullet,\ast}j^\ast\calF \cong R\varprojlim l_n^\ast Rv_{Y/S_\bullet,\ast}j^\ast\calF\cong R\varprojlim Ru_{Y/S_n,\ast} i_n^\ast\calF.
\]
The first statement is proved similarly by using $k_n^\ast$.
\end{proof}

It is often convenient to regard certain $p$-adic log PD-formal schemes as ind-objects of $(Y,M_Y)_\CRIS$.

\begin{construction}\label{construction: eval of isocrystal at log PD-formal scheme}
Let $(\calT,M_\calT)\rightarrow (\calS,M_\calS)$ be a morphism of $p$-adic log formal schemes such that $M_\calT$ is integral and quasi-coherent.
Assume that $\gamma_\calS$ extends to $\calJ_\calS\calO_\calT$, and let $\calJ_\calT\subset \calO_\calT$ be a quasi-coherent PD-ideal compatible with the PD-structure on $\calJ_\calS\calO_\calT$. Set $(\overline{T},M_{\overline{T}})\coloneqq (V(\calJ_{\calT}),M_{\calT}|_{\overline{T}})$ and write $(T_n,M_{T_n},\calJ_{T_n})$ for the base change of $(\calT,M_{\calT}, \calJ_{\calT})$ to $\Z_p/p^n$.

Consider the big logarithmic crystalline site $((Y,M_Y)/\calS^\sharp)_\CRIS$ as before and let $f\colon (\overline{T},M_{\overline{T}})\rightarrow (Y,M_Y)$ be a morphism of log schemes such that $M_{\overline{T}}\cong f^\ast M_Y$. Then $(\overline{T},T_n)\in ((Y,M_Y)/\calS^\sharp)_\CRIS$ and write $h_{(\overline{T},T_n)}$ for the sheaf represented by it. It is easy to see that the colimit $\varinjlim_n h_{(\overline{T},T_n)}$ as presheaf is indeed a sheaf. By abuse of notation, we write $(\overline{T},\calT)$ (or $h_{(\overline{T},\calT)}$) for the colimit
and say that $\calT$ is a \emph{$p$-adic log PD-formal scheme in $((Y,M_Y)/\calS^\sharp)_\CRIS$}.
By definition, for $\calF\in \Sh(((Y,M_Y)/\calS^\sharp)_\CRIS)$, we have
\[
\calF(\overline{T},\calT)=\Hom_{\Sh(((Y,M_Y)/\calS^\sharp)_\CRIS)}\bigl(\varinjlim_n h_{(\overline{T},T_n)},\calF\bigr)=\varprojlim_n \calF(\overline{T},T_n).
\]
\end{construction}

\begin{prop}\label{prop:limit formula for evaluation at log PD-formal scheme}
Keep the set-up as in Construction~\ref{construction: eval of isocrystal at log PD-formal scheme}.
For every abelian sheaf $\calF$ on $((Y,M_Y)/\calS^\sharp)_\CRIS$, we have a quasi-isomorphism
\[
R\Gamma((\overline{T},\calT),\calF)\cong R\varprojlim R\Gamma((\overline{T},T_n), \calF)
\]
\end{prop}

\begin{proof}
Since the localized sites $((Y,M_Y)/S_n^\sharp)_\CRIS/(\overline{T},T_n)$ and  $((Y,M_Y)/\calS^\sharp)_\CRIS/(\overline{T},T_n)$ are identical, we have $R\Gamma((\overline{T},T_n), \calF)=R\Gamma((\overline{T},T_n), i_n^\ast\calF)$ by locality of cohomology.
Now the assertion is deduced from Proposition~\ref{prop:j functor for colimit of sites}(2) as in the proof of Proposition~\ref{prop:limit formula for crystalline cohomology}.
\end{proof}

We turn to the functoriality of the crystalline sites relative to $\calS^\sharp$. Let $f\colon (X,M_X)\rightarrow (Y,M_Y)$ be a morphism of integral and quasi-coherent log $S_{n_0}^\sharp$-schemes  for some $n_0$.

\begin{prop}\label{prop:functoriality of absolute crystalline topoi}
There exists a unique morphism of topoi
\[
f_\CRIS=(f_\CRIS^\ast,f_{\CRIS,\ast})\colon \Sh(((X,M_{X})/\calS^\sharp)_\CRIS)\rightarrow \Sh(((Y,M_Y)/\calS^\sharp)_\CRIS)
\]
such that $(f_\CRIS^\ast\calF)(U',T',i',M_{T'},\gamma')=\calF(U',T',i',M_{(T',f)},\gamma')$ (see Construction~\ref{construction:strict log structure for functoriality} for $M_{(T',f)}$).
Moreover, $f^\ast_\CRIS$ is exact, and $f^\ast_\CRIS\calO_{Y/\calS}=\calO_{X/\calS}$. 
\end{prop}

\begin{proof}
The proof of Proposition~\ref{prop:functoriality of relative crystalline topoi} works verbatim.
\end{proof}

Consider the morphism of ringed topoi
\[
f_\CRIS=(f^\ast_\CRIS,f_{\CRIS,\ast})\colon(\Sh(((X,M_X)/\calS^\sharp)_\CRIS),\calO_{X/\calS})\rightarrow (\Sh(((Y,M_Y)/\calS^\sharp)_\CRIS),\calO_{Y/\calS}).
\]
For a sheaf $\calF$ of $\calO_{X/\calS}$-modules, we can describe $Rf_{\CRIS,\ast}\calF\in D^+(((Y,M_Y)/\calS^\sharp)_\CRIS,\calO_{Y/\calS})$ using the big relative crystalline site as follows.

Take $(U,T)\in ((Y,M_Y)/\calS^\sharp)_\CRIS$. Set $(f^{-1}(U),M_{f^{-1}(U)})\coloneqq (X,M_X)\times_{(Y,M_Y)}(U,M_U)$ where the fiber product is taken in the category of integral and quasi-coherent log schemes.  Let $T^\sharp$ denote the log PD-scheme $(T,M_T,\calJ_T+\calJ_\calS\calO_T,\tilde{\gamma})$ where $\tilde{\gamma}$ is the unique PD-structure on $\calJ_T+\calJ_\calS\calO_T$ extending $\gamma$ and $\gamma_\calS$.

\begin{lem} \label{lem:log-cris-higher-direct-image}
With the notation as above,  we continue to write $\calF$ for the pullback to $((f^{-1}(U),M_{f^{-1}(U)})/T^\sharp)_\CRIS$.
\begin{enumerate}
 \item We have
\[
(f_{\CRIS,\ast}\calF)_{(U,T)}=\pi_{f^{-1}(U)/T,\ast}\calF,
\]
where $\pi_{f^{-1}(U)/T}\colon \Sh(((f^{-1}(U),M_{f^{-1}(U)})/T^\sharp)_\CRIS,\calO_{f^{-1}(U)/T})\rightarrow\Sh(T_\et,\calO_T)$ denotes the morphism of ringed topoi in Definition~\ref{def:projection to etale site of base}.
 \item There are natural quasi-isomorphisms
\[
(Rf_{\CRIS,\ast}\calF)_{(U,T)}=R\pi_{f^{-1}(U)/T,\ast}\calF
\]
and
\[
R\Gamma((U,T), Rf_{\CRIS,\ast}\calF)\cong R\Gamma(((f^{-1}(U),M_{f^{-1}(U)})/T^\sharp)_\CRIS,\calF).
\]
\end{enumerate}
\end{lem}

\begin{proof}
The proof of Lemma~\ref{lem:evalutation of crystalline pushforward via projection} works verbatim.
\end{proof}

Let us list the $p$-adic variants of the theorems we have proved in the previous sections. For the rest of the section, we further assume that $M_\calS$ is \emph{fine}.

\begin{thm}\label{thm:absolute boundedness of crystalline pushforward}
Assume that $Y$ is quasi-compact and $f$ is quasi-separated of finite type of fine log schemes. Then there exists an integer $r$ such that $R^qf_{\CRIS,\ast}\calE=0$ for every $q\geq r$ and every quasi-coherent $\calO_{X/\calS}$-module $\calE$ on $((X,M_X)/\calS^\sharp)_\CRIS$.
\end{thm}

\begin{proof}
By Theorem~\ref{thm:boundedness of crystalline pushforward}, the statement is true for the morphism of ringed topoi
\[
(\Sh((X,M_{X})/S_n^\sharp)_\CRIS,\calO_{X/S_n})\xrightarrow{f_\CRIS} (\Sh((Y,M_{Y})/S_n^\sharp)_\CRIS,\calO_{Y/S_n})
\]
for each $n\geq n_0$. In fact, the proof therein (especially, Construction~\ref{construction:embedding system}) shows that the integer $r$ can be chosen independent of $r$, which is what we need to prove thanks to Lemma~\ref{lem:log-cris-higher-direct-image}.
\end{proof}

\begin{thm}\label{thm:crystal property of higher direct image}
Assume that $f\colon (X,M_X)\rightarrow (Y,M_Y)$ is smooth and integral of fine log schemes and that $f\colon X\rightarrow Y$ is qcqs. Then for every flat quasi-coherent $\calO_{X/\calS}$-module $\calE$ on $((X,M_X)/\calS^\sharp)_\CRIS$, the higher direct image 
\[
Rf_{\CRIS,\ast}\calE\in D(((Y,M_Y)/\calS^\sharp)_{\CRIS},\calO_{Y/\calS})
\]
is a crystal in the derived category of $\calO_{Y/\calS}$-modules in the following sense:
for every morphism $g\colon (U',T')\rightarrow (U,T)$ in $((Y,M_Y)/\calS^\sharp)_\CRIS$, the natural morphism
\[
Lg^\ast_{\et}((Rf_{\CRIS,\ast}\calE)_{(U,T)})=\calO_{T'}\otimes_{g^{-1}_\et\calO_T}^\mathbf{L}g^{-1}_\et(Rf_{\CRIS,\ast}\calE)_{(U,T)}\rightarrow (Rf_{\CRIS,\ast}\calE)_{(U',T')}
\]
is a quasi-isomorphism compatible with the compositions $(U'',T'')\rightarrow(U',T')\rightarrow (U,T)$.
\end{thm}

\begin{proof}
To show the statement, one may replace $T'$ by affine opens. Hence we may assume that both $T$ and $T'$ are affine. Then the assertion is reduced to Theorem~\ref{thm:crystalline base change for projection}, thanks to Lemma~\ref{lem:log-cris-higher-direct-image}.
\end{proof}

\begin{thm}\label{thm:absolute crystalline base change for crystalline pushforward}
Suppose that we have the following Cartesian diagram
\[
\xymatrix{
(X',M_{X'})\ar[r]^-{g'} \ar[d]_-{f'}& (X,M_X)\ar[d]_-f\\
(Y',M_{Y'})\ar[r]^-g & (Y,M_Y).\\
}
\]
Assume that $f\colon (X,M_X)\rightarrow (Y,M_Y)$ is smooth and integral of fine log schemes and that $f\colon X\rightarrow Y$ is qcqs.
Then for every flat quasi-coherent $\calO_{X/\calS}$-module $\calE$ on $((X,M_X)/\calS^\sharp)_\CRIS$, the canonical morphism
\[
g^\ast_\CRIS Rf_{\CRIS,\ast}\calE\rightarrow Rf'_{\CRIS,\ast}g'^\ast_\CRIS\calE
\]
is a quasi-isomorphism.
\end{thm}

\begin{proof}
It suffices to show that the morphism in question is a quasi-isomorphism after evaluating at each affine $(U,T)\in ((X',M_{X'})/\calS^\sharp)_\CRIS$, which follows from Theorem~\ref{thm:crystalline base change for projection} and Lemma~\ref{lem:log-cris-higher-direct-image}.
\end{proof}

\begin{thm}\label{thm:absolute crystalline derived pushforward of perfect complexes}
Assume that $f\colon (X,M_X)\rightarrow (Y,M_Y)$ is smooth and integral of fine log schemes and that $f\colon X\rightarrow Y$ is proper.
For every finite locally free $\calO_{X/\calS}$-module $\calE$, the derived pushforward $Rf_{\CRIS,\ast}\calE$ is a perfect complex of $\calO_{Y/\calS}$-modules.
\end{thm}

\begin{proof}
This follows from Theorem~\ref{thm:crystalline derived pushforward of perfect complexes} and Lemma~\ref{lem:log-cris-higher-direct-image}.
\end{proof}

\begin{rem}
With the notation as in Remark~\ref{rem:absolute small crystalline site} and Theorem~\ref{thm:absolute crystalline derived pushforward of perfect complexes}, we have a canonical quasi-isomorphism
\[
r_{Y/\calS}^\ast Rf_{\CRIS,\ast}\calE \cong Rf_{\cris,\ast}r_{X/\calS}^\ast \calE,
\]
where $Rf_{\cris,\ast}$ is induced by a morphism of topoi given in \cite[\S~1.5, Cor; \S~1.12]{Beilinson-crystalline-period-map}. This is proved as in \cite[III.4.2.2]{Berthelot-book} by noting $r_{Y/\calS}^\ast=p_{Y/\calS,\ast}=Rp_{Y/\calS,\ast}$.
\end{rem}

\section{Frobenius isogeny property}\label{sec:Frobenius isogeny property}

Finally, we discuss the Frobenius isogeny property.
Let $(Y,M_Y)$ be a smooth and integral fine log scheme over $\F_p$. Let $D((Y,M_Y)_\CRIS)$ denote the derived category of $\calO_{Y/\Z_p}$-modules and write $D_\perf((Y,M_Y)_\CRIS)$ for the strictly full triangulated subcategory of perfect complexes of $\calO_{Y/\Z_p}$-modules.
Define $D_\perf((Y,M_Y)_\CRIS)_\Q$ to be the isogeny category of $D_\perf((Y,M_Y)_\CRIS)$: the objects are exactly those of $D_\perf((Y,M_Y)_\CRIS)$ (but often written as $\calK_\Q$ for distinction) and the morphisms are defined as
\[
\Hom_{D_\perf((Y,M_Y)_\CRIS)_\Q}(\calK_\Q,\calL_\Q)\coloneqq \Hom_{D_\perf((Y,M_Y)_\CRIS)}(\calK,\calL)\otimes_\Z\Q.
\]
Note that a morphism $f_\Q\colon \calK_\Q\rightarrow \calL_\Q$ in $D_\perf((Y,M_Y)_\CRIS)_\Q$ is an isomorphism if and only if it is so after restricting to an \'etale surjection $Y'\rightarrow Y$: to see this, we may assume that $f_\Q$ comes from $f\in \Hom(\calK,\calL)$ and take a distinguished triangle of the form $\calK\xrightarrow{f}\calL\rightarrow\calM$. Then $f_\Q$ is an isomorphism if and only if $\calM_\Q\cong 0$, or $p^r=\id_\calM=0$ for some $r\in\N$, and the formation commutes with \'etale pullbacks. The assertion now follows easily from these observations. 

Since $(Y,M_Y)$ is a log scheme over $\F_p$, we have the absolute $p$th Frobenius $F_Y\colon (Y,M_Y)\rightarrow (Y,M_Y)$. The morphism of topoi $F_{Y,\CRIS}\colon \Sh((Y,M_Y)_\CRIS)\rightarrow \Sh((Y,M_Y)_\CRIS)$ gives rise to an endomorphism 
\[
F_{Y,\CRIS}^\ast\colon D_\perf((Y,M_Y)_\CRIS)_\Q\rightarrow D_\perf((Y,M_Y)_\CRIS)_\Q
\]
(note $LF_{Y,\CRIS}^\ast=F_{Y,\CRIS}^\ast$ as we work on the big crystalline site).

\begin{defn}\label{def:F-isocrystals}
An \emph{$F$-isocrystal in perfect complexes} on $(Y,M_Y)_\CRIS$ is a pair $(\calK_\Q,\varphi_{\calK_\Q})$ where $\calK_\Q\in D_\perf((Y,M_Y)_\CRIS)_\Q$ and $\varphi_{\calK_\Q}\colon F_{Y,\CRIS}^\ast\calK_\Q\rightarrow \calK_\Q$ is an isomorphism in $D_\perf((Y,M_Y)_\CRIS)_\Q$. 
Similarly, an \emph{$F$-isocrystal structure} on a locally free $\calO_{Y/\Z_p}$-module $\calE$ is an element $\varphi_{\calE_\Q}$ in $\Hom(F_{Y,\CRIS}^\ast\calE_\Q,\calE_\Q)\coloneqq \Hom_{\calO_{Y/\Z_p}}(F_{Y,\CRIS}^\ast\calE,\calE)\otimes_\Z\Q$ that is an isomorphism in the obvious sense. We simply call the pair $(\calE,\varphi_{\calE_\Q})$ a \emph{finite locally free $F$-crystal} on $(Y,M_Y)_\CRIS$.
\end{defn}

\begin{rem}\label{rem:Beilonson's non-degenerate perfect F-crystals}
Assume that $Y$ is quasi-compact. Let $(\calK_\Q,\varphi_{\calK_\Q})$ be an $F$-isocrystal in perfect complexes  on $(Y,M_Y)_\CRIS$ and take $\calK\in D_\perf((Y,M_Y)_\CRIS)$. Then $\varphi_{\calK_\Q}$ can be written as  $\varphi_{\calK_\Q}=p^{r}\varphi_\calK$ for some $r\in\Z$ and a morphism $\varphi_\calK\colon F_{Y,\CRIS}^\ast\calK\rightarrow \calK$ in $D_\perf((Y,M_Y)_\CRIS)$.
A pair $(\calK,\varphi_\calK\colon F_{Y,\CRIS}^\ast\calK\rightarrow \calK)$ such that $(\varphi_\calK)_\Q$ is an isomorphism in $D_\perf((Y,M_Y)_\CRIS)_\Q$ is called a \emph{non-degenerate perfect $F$-crystal}.
Similarly, we define the notion of \emph{non-degenerate finite locally free $F$-crystals}.
\end{rem}

Let $f\colon (X,M_X)\rightarrow (Y,M_Y)$ be a morphism of fine log schemes over $\F_p$. Consider $(X',M_{X'})\coloneqq (Y,M_Y)\times_{F_Y,(Y,M_Y)}(X,M_X)$; it sits in the following commutative diagram:
\begin{equation}\label{eq:relative Frobenius}
\xymatrix{
(X,M_X) \ar[r]^-{F_{X/Y}}\ar[dr]_-{f} \ar@/^2pc/[rr]^-{F_X} & (X',M_{X'})\ar[d]_-{f'}\ar[r]^g \ar@{}[dr] | {\square} &(X,M_X)\ar[d]^-{f}\\
& (Y,M_Y) \ar[r]^-{F_Y} & (Y,M_Y).
}
\end{equation}
Note that if $f$ is integral, then $X'=Y\times_{F_Y,Y}X$ as schemes.
We say that $f$ is \emph{of Cartier type} if $f$ is integral and $F_{X/Y}$ is exact; for example, the mod $p$ reduction of the structure morphism of a semistable formal scheme is of Cartier type (cf.~\cite[Def.~4.8]{Kato-log}, Remark~\ref{rem:Frobenius isogeny property} below).

\begin{thm}\label{thm:Frobenius isogeny property}
Let $f\colon (X,M_X)\rightarrow (Y,M_Y)$ be smooth of Cartier type between fine log schemes over $\F_p$ such that $f\colon X\rightarrow Y$ is proper and $Y$ is quasi-compact.
Assume moreover that $(Y,M_Y)$ is smooth and integral over a perfect field $k$ of characteristic $p$ with trivial log structure. 
If $(\calE,\varphi_{\calE_\Q})$ is a finite locally free $F$-crystal on $(X,M_X)_\CRIS$, then $(Rf_{\CRIS,\ast}\calE_\Q,Rf_{\CRIS,\ast}\varphi_{\calE_\Q})$ is an $F$-isocrystal in perfect complexes on $(Y,M_Y)_\CRIS$ (see the proof below for the precise definition).
Moreover, $Rf_{\CRIS,\ast}$ preserves the non-degeneracy property.
\end{thm}

\begin{proof}
Let us start with the definition of $Rf_{\CRIS,\ast}\varphi_{\calE_\Q}$.
Write $\varphi_{\calE_\Q}=p^{r}\varphi_\calE$ for some $r\in\Z$ and a morphism $\varphi_\calE\colon F_{X,\CRIS}^\ast\calE\rightarrow \calE$ of $\calO_{X/\Z_p}$-modules. We already know from Theorem~\ref{thm:absolute crystalline derived pushforward of perfect complexes} that $Rf_{\CRIS,\ast}\calE$ is a perfect complex of $\calO_{Y/\Z_p}$-modules.

Name the morphisms as in \eqref{eq:relative Frobenius}.
Set $\calE'\coloneqq g_{\CRIS}^\ast \calE$.
Consider the composite $\alpha$ of the maps
\begin{equation}\label{eq:Frobenius pullback and higher direct image}
\xymatrix{
F_{Y,\CRIS}^\ast Rf_{\CRIS,\ast}\calE \ar[r]^-\cong\ar@{-->}[rd]_-{\alpha}
&Rf'_{\CRIS,\ast}\calE'\ar[r]_-{Rf'_{\CRIS,\ast}\mathrm{adj}}^-{\beta}
&Rf'_{\CRIS,\ast}RF_{X/Y,\CRIS,\ast}F_{X/Y,\CRIS}^\ast\calE'
\ar[d]^-\cong\\
& Rf_{\CRIS,\ast}\calE
&
Rf_{\CRIS,\ast}F_{X,\CRIS}^\ast\calE\ar[l]_-{Rf_{\CRIS,\ast}\varphi_\calE},
}
\end{equation}
where the horizontal quasi-isomorphism comes from Theorem~\ref{thm:absolute crystalline base change for crystalline pushforward} and the vertical one is given by the composite of right derived functors induced from the composite of morphisms of topoi.
The morphism $Rf_{\CRIS,\ast}\varphi_{\calE_\Q}\colon F_{Y,\CRIS}^\ast Rf_{\CRIS,\ast}\calE_\Q\rightarrow Rf_{\CRIS,\ast}\calE_\Q$ in the theorem is defined to be $p^r\alpha$; it is easy to check that this morphism is independent of the choice of $(\varphi_\calE,r)$.
By construction, for the first statement, it remains to show that $\beta_\Q$ is an isomorphism.

To study $\beta$, recall $\calO_{X/\Z_p}=F_{X/Y,\CRIS}^\ast\calO_{X'/\Z_p}=LF_{X/Y,\CRIS}^\ast\calO_{X'/\Z_p}$ and consider the diagram
\[
\xymatrix{
\calE' \ar[r]^-{\mathrm{adj}}\ar@{=}[d]
&RF_{X/Y,\CRIS,\ast}(F_{X/Y,\CRIS}^\ast\calE'\otimes F_{X/Y,\CRIS}^\ast\calO_{X'/\Z_p})\\
\calE'\otimes \calO_{X'/\Z_p}\ar[r]^-{\id\otimes\mathrm{adj}} &\calE'\otimes RF_{X/Y,\CRIS,\ast} F_{X/Y,\CRIS}^\ast\calO_{X'/\Z_p}\ar[u]_-\gamma^-\cong,
}
\]
where $\gamma$ is the canonical map and a quasi-isomorphism by the projection formula, and the above diagram is commutative by construction (see \cite[0943]{stacks-project} for example). Since $\beta=Rf'_{\CRIS,\ast}(\gamma\circ \id\otimes\mathrm{adj})$, it remains to prove that the canonical morphism given by adjunction
\[
\calO_{X'/\Z_p}\rightarrow RF_{X/Y,\CRIS,\ast}F_{X/Y,\CRIS}^\ast\calO_{X'/\Z_p}=RF_{X/Y,\CRIS,\ast}\calO_{X/\Z_p}
\]
is an isogeny (i.e., an isomorphism in $D_\perf((X,M_X)_\CRIS)_\Q$), or equivalently, its cone $\calL$ becomes zero in the isogeny category. Since $Y$ is quasi-compact, one may check this \'etale locally on $Y$. 

So we may assume that $Y$ is affine. Using smoothness and \cite[Prop.~3.14(1)]{Kato-log}, one can inductively find, for each $n$, an affine smooth and integral fine log scheme $(Y_n,M_{Y_n})$ over $W_n(k)$ such that $(Y_1,M_{Y_1})=(Y,M_Y)$ and $(Y_{n+1},M_{Y_{n+1}})\times_{W_{n+1}(k)}W_n(k)=(Y_n,M_{Y_n})$, where $W_n(k)$ is equipped with trivial log structure. Note that $(Y,Y_n)\in (Y,M_Y)_\CRIS$ and $Y_n$ is flat over $W_n(k)$. Write $Y_n=\Spec R_n$ and set $R=\varprojlim_n R_n$. Then $\calY=\Spf R$ is a $p$-adic formal scheme equipped with fine log structure $M_\calY$ such that $(\calY,M_\calY)\times_{\Spf W(k)}\Spec W_n(k)=(Y_n,M_{Y_n})$. As in Construction~\ref{construction: eval of isocrystal at log PD-formal scheme}, we regard $(Y,\calY)$ as a $p$-adic log PD-formal scheme in $(Y,M_Y)_\CRIS$. It follows from Example~\ref{eg:absolute crystalline site over k} and \cite[Prop. IV.3.1.4.2]{Ogus-log} that every affine $(U,T)\in (Y,M_Y)_\CRIS$ admits a morphism 
to $(Y,\calY)$. In particular, we know from Theorem~\ref{thm:crystal property of higher direct image} and (a slight generalization of) Proposition~\ref{prop:limit formula for evaluation at log PD-formal scheme} that, to show $\calL_\Q=0$, it is enough to prove that 
\[
R\Gamma((Y,\calY),\calL)=R\varprojlim R\Gamma((Y,Y_n),\calL)\in D(R)
\]
is zero in $D(R)_\Q$. Now this and the second assertion of the theorem follow from Proposition~\ref{prop:hyodo-kato-relative-Frobenius} below of Hyodo and Kato, Lemma~\ref{lem:log-cris-higher-direct-image}(2), and \cite[06Z0]{stacks-project}.
\end{proof}

\begin{prop} \label{prop:hyodo-kato-relative-Frobenius}
Let $f\colon (X,M_X)\rightarrow (Y,M_Y)$ be a smooth morphism of Cartier type between fine log schemes over $\F_p$.
Assume that there are fine log schemes $(T_n,M_{T_n})$ with exact closed immersions
\[
(Y,M_Y)\hookrightarrow (T_1,M_{T_1})\hookrightarrow (T_2,M_{T_2})\hookrightarrow\cdots
\]
and a PD-structure on the ideal of $Y$ in $T_n$ such that
\begin{enumerate}
 \item $T_n\rightarrow T_{n+1}$ is a PD-morphism, and
 \item $T_n$ is flat over $\Z/p^n$ with $T_n\xrightarrow{\cong}T_{n+1}\otimes\Z/p^n$.
\end{enumerate}
Assume also that $r\coloneqq \max_x \rank_x \omega^1_{(X,M_X)/(Y,M_Y)}$ exists (e.g.~$X$ is quasi-compact). Then for the natural map of projective systems 
\[
\varphi\colon Ru_{X'/T_\bullet,\ast}\calO_{X'/T_\bullet}\rightarrow Ru_{X/T_\bullet,\ast}\calO_{X/T_\bullet}
\]
(under the identification of $X_\et\cong X'_\et$),
there exists a morphism of projective systems 
\[
\psi\colon Ru_{X/T_\bullet,\ast}\calO_{X/T_\bullet}\rightarrow Ru_{X'/T_\bullet,\ast}\calO_{X'/T_\bullet}
\]
such that $\varphi\circ \psi=p^r$ and $\psi\circ \varphi=p^r$.
\end{prop}

\begin{proof}
This is \cite[Prop.~2.24]{hyodo-kato}. Note that the small crystalline site is used in \textit{loc. cit.}, but the same proof works for the big crystalline site by using Theorem~\ref{thm:cohomology of crystal in terms of log de Rham complex}. 
\end{proof}

\begin{rem}
With the set-up and notation as in Theorem~\ref{thm:Frobenius isogeny property} and its proof, assume further that $Y$ is affine as in the last paragraph of the proof. In this case, we can give a bound of the Frobenius height of the crystalline pushforward. Namely, the proof shows that if $\operatorname{Cone}(\varphi_\calK)\otimes^\mathbf{L}_{\Z_p}\Z_p/p^s=0$, then $\operatorname{Cone}(\alpha)\otimes^\mathbf{L}_{\Z_p}\Z_p/p^{s+r}=0$ where $r=\max_x \rank_x \omega^1_{(X,M_X)/(Y,M_Y)}$. 
\end{rem}

\begin{rem}\label{rem:Frobenius isogeny property}\hfill
\begin{enumerate}
 \item 
In Theorem~\ref{thm:Frobenius isogeny property}, we assume that $Y$ is smooth and integral over a perfect field $k$. As the proof shows, this can be relaxed to the condition that $(Y,M_Y)_\CRIS$ admits \'etale locally a weakly final ind-object that is flat over $\Z_p$; namely, \'etale locally, $(Y,M_Y)$ admits a chain of exact closed immersions 
\[
(Y,M_Y)\hookrightarrow (T_1,M_{T_1})\hookrightarrow (T_2,M_{T_2})\hookrightarrow\cdots
\]
satisfying the assumptions of Proposition~\ref{prop:hyodo-kato-relative-Frobenius} and the condition that every $(U,T)\in (Y,M_Y)_\CRIS$ admits, \'etale locally, a morphism to $(T_n,M_{T_n})$ for some $n$.

For example, this is the case when $Y$ is the mod $p$ reduction of a semistable formal scheme. More precisely, let $k$ be a perfect field of characteristic $p$ and $K$ a finite totally ramified extension of $W(k)[p^{-1}]$. Then Theorem~\ref{thm:Frobenius isogeny property} also holds if $(Y,M_Y)$ is the mod $p$ reduction of a semistable $p$-adic formal scheme $\calY$ over $\calO_K$ equipped with log structure associated to $\calO_\calY\cap (\calO_\calY[p^{-1}])^\times$. To see this, recall that, by definition, $\calY$ is \'etale locally of the form $\Spf R$ where $R$ is a $p$-completely \'etale over
\[
R_0\coloneqq\calO_K\langle t_1,\ldots,t_m,t_{m+1}^{\pm 1},\ldots,t_{d}^{\pm 1}\rangle/(t_1\cdots t_m-\pi)
\]
for a uniformizer $\pi$ of $\calO_K$. Let $T_n\coloneqq \Spec S/p^n$ where $S$ is defined as in \cite[Ex.~B.25]{du-liu-moon-shimizu-purity-F-crystal} (when $R=R_0$, it is defined as the $p$-adic completion of the log PD-envelope of $\calO_K\langle t_1,\ldots,t_m,t_{m+1}^{\pm 1},\ldots,t_{d}^{\pm 1}\rangle\rightarrow R/p$).
Then $\{T_n\}_n$ (with the induced log structure as in \textit{loc.~cit.}) satisfies the weakly finial condition by \cite[Lem.~B.24]{du-liu-moon-shimizu-purity-F-crystal}, and it is easy to check that $S$ is $p$-torsion free and thus flat over $\Z_p$.
 \item In Theorem~\ref{thm:Frobenius isogeny property}, the assumption that $\calE$ is a finite locally free $\calO_{X/Z_p}$-module is used to deduce that (a)  $Rf_{\CRIS,\ast}\calE$ is a perfect (Theorem~\ref{thm:absolute crystalline derived pushforward of perfect complexes}) and (b) the top left horizontal map in \eqref{eq:Frobenius pullback and higher direct image} is a quasi-isomorphism (Theorem~\ref{thm:absolute crystalline base change for crystalline pushforward}). One obtains the Frobenius isogeny property for a quasi-coherent $\calO_{X/Z_p}$-module with an $F$-isocrystal structure provided that (a) and (b) hold.
\end{enumerate}
\end{rem}

\bibliographystyle{amsalpha}
\bibliography{library}

\providecommand{\bysame}{\leavevmode\hbox to3em{\hrulefill}\thinspace}
\providecommand{\MR}{\relax\ifhmode\unskip\space\fi MR }
\providecommand{\MRhref}[2]{%
  \href{http://www.ams.org/mathscinet-getitem?mr=#1}{#2}
}
\providecommand{\href}[2]{#2}
\begin{thebibliography}{DLMS26}

\bibitem[Bei13]{Beilinson-crystalline-period-map}
A.~Beilinson, \emph{On the crystalline period map}, Camb. J. Math. \textbf{1}
  (2013), no.~1, 1--51. \MR{3272051}

\bibitem[Ber74]{Berthelot-book}
Pierre Berthelot, \emph{Cohomologie cristalline des sch\'{e}mas de
  caract\'{e}ristique {$p>0$}}, Lecture Notes in Mathematics, Vol. 407,
  Springer-Verlag, Berlin-New York, 1974. \MR{0384804}

\bibitem[BMS18]{bhatt-morrow-scholze-integralpadic}
Bhargav Bhatt, Matthew Morrow, and Peter Scholze, \emph{Integral {$p$}-adic
  {H}odge theory}, Publ. {M}ath. {I}nst. {H}autes \'{E}tudes {S}ci.
  \textbf{128} (2018), 219--397.

\bibitem[BO78]{berthelot-ogus-book}
Pierre Berthelot and Arthur Ogus, \emph{Notes on crystalline cohomology},
  Princeton University Press, Princeton, NJ; University of Tokyo Press, Tokyo,
  1978. \MR{491705}

\bibitem[BS22]{bhatt-scholze-prismaticcohom}
Bhargav Bhatt and Peter Scholze, \emph{Prisms and prismatic cohomology}, Ann.
  of {M}ath. (2) \textbf{196} (2022), 1135--1275.

\bibitem[DJ95]{deJong-dieudonnemodule}
Aise~Johan De~Jong, \emph{Crystalline {D}ieudonn\'{e} module theory via formal
  and rigid geometry}, Publ. {M}ath. {I}nst. {H}autes \'{E}tudes {S}ci.
  \textbf{82} (1995), 5--96.

\bibitem[DLMS26]{du-liu-moon-shimizu-purity-F-crystal}
Heng Du, Tong Liu, Yong~Suk Moon, and Koji Shimizu, \emph{Log prismatic
  {$F$}-crystals and purity}, 2026, arXiv:2404.19603v2.

\bibitem[Fal99]{Faltings-very-ramified}
Gerd Faltings, \emph{Integral crystalline cohomology over very ramified
  valuation rings}, J. Amer. Math. Soc. \textbf{12} (1999), no.~1, 117--144.
  \MR{1618483}

\bibitem[Fal02]{faltings-almostetale}
\bysame, \emph{Almost \'{e}tale extensions}, Ast\'{e}risque (2002), no.~279,
  185--270, Cohomologies $p$-adiques et applications arithm\'{e}tiques, II.
  \MR{1922831}

\bibitem[Gro68]{Grothendieck-crystals-deRhamCohomology}
A.~Grothendieck, \emph{Crystals and the de {R}ham cohomology of schemes}, Dix
  expos\'es sur la cohomologie des sch\'emas, Adv. Stud. Pure Math., vol.~3,
  North-Holland, Amsterdam, 1968, Notes by I. Coates and O. Jussila,
  pp.~306--358. \MR{269663}

\bibitem[HK94]{hyodo-kato}
Osamu Hyodo and Kazuya Kato, \emph{Semi-stable reduction and crystalline
  cohomology with logarithmic poles}, no. 223, 1994, P\'eriodes $p$-adiques
  (Bures-sur-Yvette, 1988), pp.~221--268. \MR{1293974}

\bibitem[Kat89]{Kato-log}
Kazuya Kato, \emph{Logarithmic structures of {F}ontaine-{I}llusie}, Algebraic
  analysis, geometry, and number theory ({B}altimore, {MD}, 1988), Johns
  Hopkins Univ. Press, Baltimore, MD, 1989, pp.~191--224. \MR{1463703}

\bibitem[Kos22]{koshikawa}
Teruhisa Koshikawa, \emph{Logarithmic prismatic cohomology {I}}, 2022,
  arXiv:2007.14037.

\bibitem[NS08]{nakkajima-shiho}
Yukiyoshi Nakkajima and Atsushi Shiho, \emph{Weight filtrations on log
  crystalline cohomologies of families of open smooth varieties}, Lecture Notes
  in Mathematics, vol. 1959, Springer-Verlag, Berlin, 2008. \MR{2450600}

\bibitem[Ogu94]{ogus-griffiths}
Arthur Ogus, \emph{{$F$}-crystals, {G}riffiths transversality, and the {H}odge
  decomposition}, Ast\'{e}risque (1994), no.~221, ii+183. \MR{1280543}

\bibitem[Ogu04]{Ogus-Higgs}
\bysame, \emph{Higgs cohomology, {$p$}-curvature, and the {C}artier
  isomorphism}, Compos. Math. \textbf{140} (2004), no.~1, 145--164.
  \MR{2004127}

\bibitem[Ogu18]{Ogus-log}
\bysame, \emph{Lectures on logarithmic algebraic geometry}, Cambridge Studies
  in Advanced Mathematics, vol. 178, Cambridge University Press, Cambridge,
  2018. \MR{3838359}

\bibitem[Ols07]{Olsson-crystallinecohomology}
Martin~C. Olsson, \emph{Crystalline cohomology of algebraic stacks and
  {H}yodo-{K}ato cohomology}, Ast\'erisque (2007), no.~316, 412. \MR{2451400}

\bibitem[SGA72]{SGA4-2}
\emph{Th\'eorie des topos et cohomologie \'etale des sch\'emas. {T}ome 2},
  Lecture Notes in Mathematics, vol. Vol. 270, Springer-Verlag, Berlin-New
  York, 1972, S\'eminaire de G\'eom\'etrie Alg\'ebrique du Bois-Marie
  1963--1964 (SGA 4), Dirig\'e{} par M. Artin, A. Grothendieck et J. L.
  Verdier. Avec la collaboration de N. Bourbaki, P. Deligne et B. Saint-Donat.
  \MR{354653}

\bibitem[Shi00]{Shiho-I}
Atsushi Shiho, \emph{Crystalline fundamental groups. {I}. {I}socrystals on log
  crystalline site and log convergent site}, J. Math. Sci. Univ. Tokyo
  \textbf{7} (2000), no.~4, 509--656. \MR{1800845}

\bibitem[Shi02]{Shiho-II}
\bysame, \emph{Crystalline fundamental groups. {II}. {L}og convergent
  cohomology and rigid cohomology}, J. Math. Sci. Univ. Tokyo \textbf{9}
  (2002), no.~1, 1--163. \MR{1889223}

\bibitem[{Sta}21]{stacks-project}
The {Stacks project authors}, \emph{The stacks project},
  \url{https://stacks.math.columbia.edu}, 2021.

\bibitem[Tsu99]{Tsuji-PoincareDuality}
Takeshi Tsuji, \emph{Poincar\'e{} duality for logarithmic crystalline
  cohomology}, Compositio Math. \textbf{118} (1999), no.~1, 11--41.
  \MR{1705975}

\end{thebibliography}

\end{document}